\documentclass[11pt,a4paper,titlepage,twoside
]{book}

\usepackage[utf8x]{inputenc}
\usepackage[T1]{fontenc}
\usepackage{kpfonts}

\usepackage[a4paper,includeheadfoot,pdftex,textwidth=16cm,textheight=24cm,
bottom=3.6cm]{geometry}
\usepackage[svgnames]{xcolor}
\usepackage{graphicx}

\usepackage[bookmarks=true,
pdfborder={0 0 1},
colorlinks=true,urlcolor=blue,citecolor=Purple, 
linkcolor=NavyBlue,hypertexnames=false]{hyperref}

\usepackage{enumitem}
\setlist{parsep=0pt} 
\setlist[itemize,enumerate]{nolistsep,itemsep=1pt,topsep=1pt} 
\setlist{leftmargin=5mm}

\usepackage{fancybox}
\usepackage[Lenny]{fncychap}
\usepackage{fancyhdr}
\usepackage{amsmath}
\usepackage{amsfonts}
\usepackage{amssymb}
\usepackage{amsthm}
\usepackage{ upgreek }

\usepackage{bbm}

\usepackage{mathtools}
\usepackage{mdframed}

\usepackage{tikz}
\usetikzlibrary{matrix,arrows,calc}
\usetikzlibrary{decorations.pathmorphing}
\usetikzlibrary{decorations.markings}
\usetikzlibrary{cd}
\usepgflibrary{fpu}

\usepackage[margin=10pt,font=small,labelfont=bf,
labelsep=endash]{caption}

\definecolor{ThmColor}{rgb}{0.85,0.85,0.99}
\definecolor{DefColor}{rgb}{0.81,0.92,0.97}
\definecolor{RemColor}{rgb}{0.92,0.85,0.92}
\definecolor{ExoColor}{rgb}{0.81,0.99,0.81}

\mdfdefinestyle{thmstyle}{backgroundcolor=ThmColor,nobreak}
\mdfdefinestyle{defstyle}{backgroundcolor=DefColor,nobreak} 
\mdfdefinestyle{remstyle}{backgroundcolor=RemColor} 
\mdfdefinestyle{exostyle}{backgroundcolor=ExoColor} 

\mdtheorem[style=thmstyle]{theorem}%
{Theorem}[section]

\mdtheorem[style=thmstyle]{proposition}[theorem]{Proposition}[section]
\mdtheorem[ntheorem,style=thmstyle]{corollary}[theorem]{Corollary}[section]
\mdtheorem[ntheorem,style=thmstyle]{lemma}[theorem]{Lemma}[section]

\mdtheorem[ntheorem,style=defstyle]{definition}[theorem]{Definition}[section]
\mdtheorem[ntheorem,style=defstyle]{notation}[theorem]{Notation}[section]
\mdtheorem[ntheorem,style=defstyle]{assumption}[theorem]{Assumption}[section]

\mdtheorem[ntheorem,style=remstyle]{example}[theorem]{Example}[section]
\mdtheorem[ntheorem,style=remstyle]{remark}[theorem]{Remark}[section]

\mdtheorem[ntheorem,style=exostyle]{exercise}[theorem]{Exercise}[section]

\usepackage{cleveref}
\crefname{exercise}{exercise}{exercises}

\usepackage{autonum} 

\input{sarajevo.sty}

\begin{document}

\pagestyle{empty}
\newgeometry{margin=1in}

\hypersetup{pageanchor=false}

\thispagestyle{empty}

\vspace*{1cm}
\begin{center}

{\Huge\bfseries\scshape
Long-time dynamics \\[1mm]  
of stochastic differential equations \\[1mm]
}

\vspace*{12mm}
{\large Nils Berglund}\\[2mm]
{\large Institut Denis Poisson -- UMR 7013}\\[2mm]
{\large Universite d'Orleans, Universite de Tours, CNRS}

\vspace*{12mm}
{\Large Lecture notes\\[4mm]
Summer School --- From kinetic equations to
statistical mechanics}\\[2mm]
{\large Saint Jean de Monts, June--July 2021}\\
\vspace*{12mm}

\begin{figure}[h]
\begin{center}
\includegraphics[width=12cm]{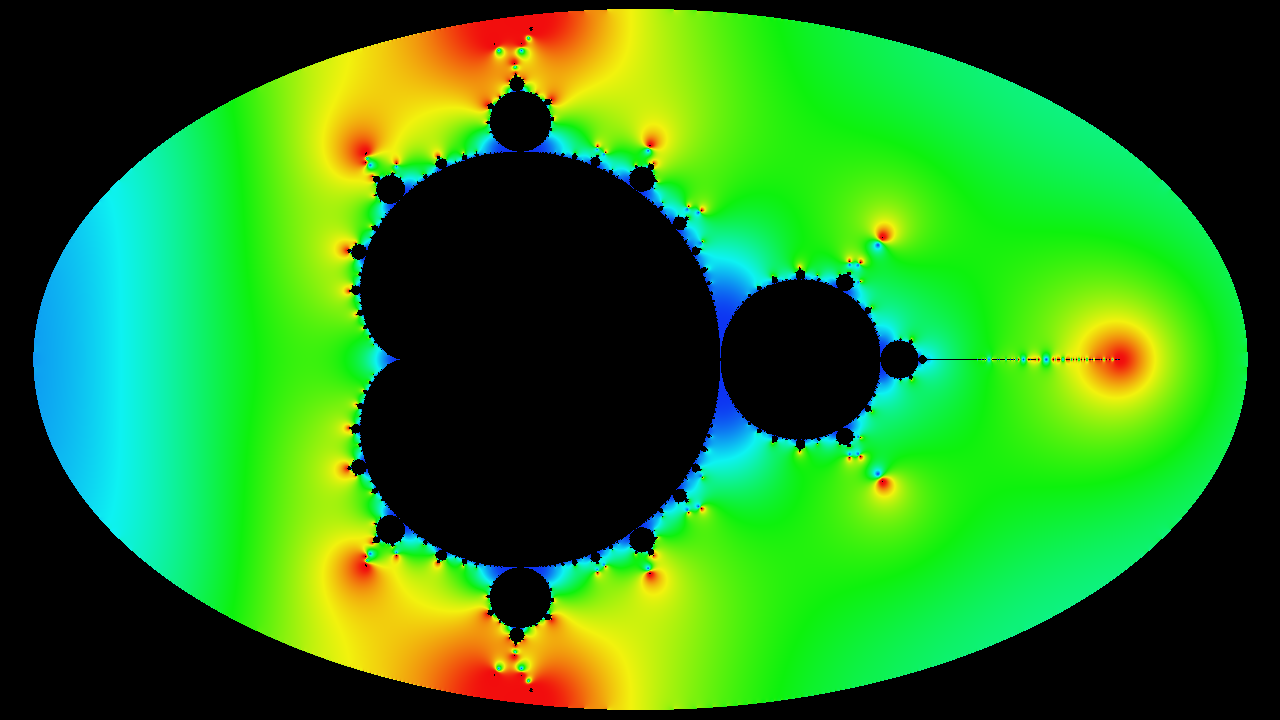}
\end{center}
\end{figure}

\vspace*{27mm}
--- Version of August 24, 2021 ---\\[2mm]

\end{center}
\hypersetup{pageanchor=true}

\cleardoublepage

\pagestyle{fancy}
\fancyhead[RO,LE]{\thepage}
\fancyhead[LO]{\nouppercase{\rightmark}}
\fancyhead[RE]{\nouppercase{\leftmark}}
\cfoot{}
\setcounter{page}{1}
\pagenumbering{roman}
\restoregeometry

\tableofcontents

\cleardoublepage


\chapter*{Preface}

These lecture notes have been prepared for a series of lectures given at 
the Summer School \lq\lq 
\href{https://www.lebesgue.fr/content/sem2021-equat_cynet}{From kinetic 
equations to statistical mechanics}\rq\rq, organised by the Henri Lebesgue 
Center in Saint Jean de Monts, from June 28th to July 2nd 2021.

The author wishes to thank the organisers of the Summer School, Fr\'ed\'eric 
H\'erau (Universit\'e de Nantes), Laurent Michel (Universit\'e de Bordeaux), and
Karel Pravda-Starov (Universit\'e de Rennes 1) for the kind invitation that 
provided the motivation to compile these lecture notes.

\vfill

\cleardoublepage
\setcounter{page}{1}
\pagenumbering{arabic}


\chapter{Stochastic Differential Equations and Partial Differential Equations}
\label{chap:SDE_PDE} 


\section{Brownian motion}
\label{sec:BM} 

The fundamental building block of the theory of stochastic differential 
equations is a mathematical object called \emph{Wiener process}, 
or~\emph{Brownian motion}. This should not be confused with the physical 
phenomenon of Brownian motion, describing for instance the erratic movements of 
a small particle in a fluid, though the mathematical model has of course been 
introduced as a simplified description of the physical process. There is a huge 
literature on properties of Brownian motion. In what follows, we will focus on 
only a few of these properties that will be important for links between 
stochastic and partial differential equations. 


\subsection{Construction of Brownian motion}
\label{ssec:BM_construct} 

Heuristically, Brownian motion can be defined as a scaling limit of a random 
walk. Let $\set{X_n}_{n\geqs0}$ be a symmetric random walk on $\Z$, defined as  
\begin{equation}
 \label{mbma1}
X_n = \sum_{i=1}^n \xi_i\;, 
\end{equation} 
where the $\xi_i$ are i.i.d.\ (independent and identically distributed) random 
variables, taking values $\pm1$ with probability $\frac12$. The following 
properties are easy to check:
\begin{enumerate}
\item	$X_n$ has zero expectation: $\expec{X_n}=0$ for all $n$;
\item	The variance of $X_n$ satisfies $\variance(X_n)=n$;
\item	$X_n$ takes values in $\set{-n,-n+2,\dots,n-2,n}$, with  
\begin{equation}
 \label{mbma2}
\bigprob{X_n=k} = \frac{1}{2^n}
\frac{n!}{\bigpar{\frac{n+k}{2}}!\bigpar{\frac{n-k}{2}}!} \;.
\end{equation} 
\item	{\it Independent increments:}\/ for all
$n>m\geqs0$, $X_n-X_m$ is independent of $X_1,\dots,X_m$;
\item	{\it Stationary increments:}\/ for all 
$n>m\geqs0$, $X_n-X_m$ has the same distribution as $X_{n-m}$.
\end{enumerate}

Consider now the sequence of processes 
\begin{equation}
 \label{mbma3}
W_t^{(n)} = \frac{1}{\sqrt{n}} X_{\intpart{nt}}\;, 
\qquad
t\in\R_+\;, \quad
n\in\N\;.
\end{equation} 
At stage $n$, space has been compressed by a factor $n$, while time has been 
sped up by a factor $\sqrt{n}$ (\figref{fig:BM}). 

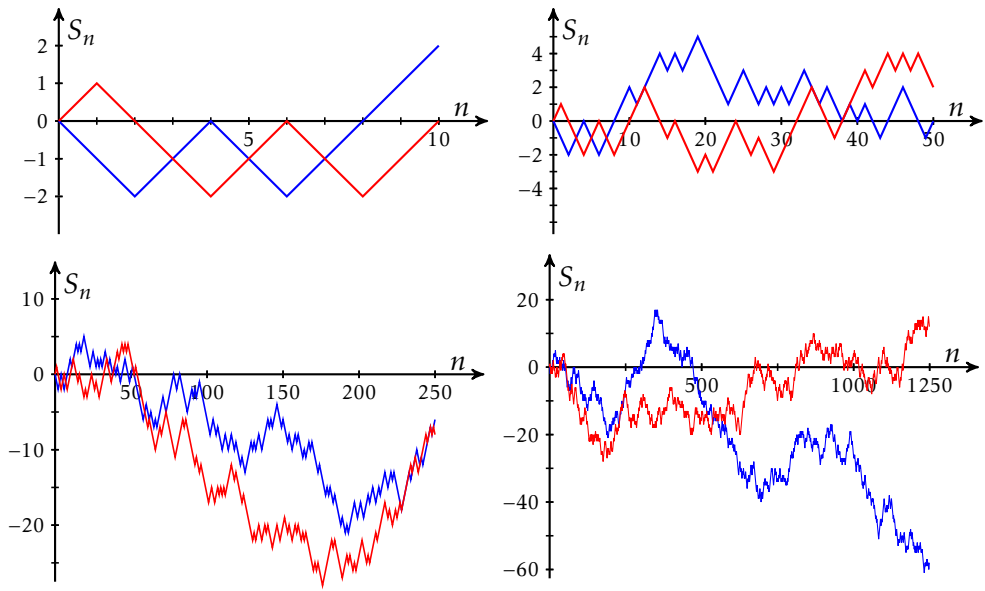
\begin{figure}
\begin{center}
\pgfmathsetmacro{\firstseed}{4}
\pgfmathsetmacro{\secondseed}{12345}
\newcommand*{\firstbmcolor}{blue}
\newcommand*{\secondbmcolor}{red}
\begin{tikzpicture}[>=stealth',main
node/.style={draw,semithick,circle,fill=white,minimum
size=2pt,inner sep=0pt},x=0.5cm,y=0.5cm]



\draw[->,thick] (0,0) -- (11.3,0);
\draw[->,thick] (0,-3) -- (0,3);


\foreach \x in {1,...,10}
\draw[semithick] (\x,-0.1) -- (\x,0.1);
\foreach \x in {5,10}
\draw[semithick] (\x,-0.1) -- node[below] {{\scriptsize $\x$}}
(\x,0.1);

\draw[semithick,white] (-0.1,-2) -- node[left] {{\scriptsize $-20$}}
(0.1,-2);	

\foreach \y in {-2,...,2}
\draw[semithick] (-0.1,\y) -- (0.1,\y);
\foreach \y in {-2,...,2}
\draw[semithick] (-0.1,\y) -- node[left] {{\scriptsize $\y$}}
(0.1,\y);


\pgfmathsetseed{\firstseed}
\draw[thick,\firstbmcolor] (0,0)
\foreach \x in {1,...,10}
{   -- ++(1,{-int(rnd*2)*2+1})
}
;

\pgfmathsetseed{\secondseed}
\draw[thick,\secondbmcolor] (0,0)
\foreach \x in {1,...,10}
{   -- ++(1,{int(rnd*2)*2-1})
}
;


\node[] at (10.6,0.3) {$n$};
\node[] at (0.6,2.35) {$S_n$};
\end{tikzpicture}
\begin{tikzpicture}[>=stealth',main
node/.style={draw,semithick,circle,fill=white,minimum
size=2pt,inner sep=0pt},x=0.5cm,y=0.5cm]


\pgfmathsetmacro{\yscale}{sqrt(0.2)}


\draw[->,thick] (0,0) -- (11.3,0);
\draw[->,thick] (0,-3) -- (0,3);


\foreach \x in {1,...,5}
\draw[semithick] (2*\x,-0.1) -- (2*\x,0.1);
\foreach \x in {1,...,5}
\draw[semithick] (2*\x,-0.1) -- node[below] {{\scriptsize $\x0$}}
(2*\x,0.1);

\draw[semithick,white] (-0.1,-2*\yscale) -- node[left] {{\scriptsize $-20$}}
(0.1,-2*\yscale);	

\foreach \y in {-6,...,5}
\draw[semithick] (-0.1,\y*\yscale) -- (0.1,\y*\yscale);
\foreach \y in {-4,-2,0,2,4}
\draw[semithick] (-0.1,\y*\yscale) -- node[left] {{\scriptsize $\y$}}
(0.1,\y*\yscale);


\pgfmathsetseed{\firstseed}
\draw[thick,\firstbmcolor] (0,0)
\foreach \x in {1,...,50}
{   -- ++(0.2,{(-int(rnd*2)*2+1)*\yscale})
}
;

\pgfmathsetseed{\secondseed}
\draw[thick,\secondbmcolor] (0,0)
\foreach \x in {1,...,50}
{   -- ++(0.2,{(int(rnd*2)*2-1)*\yscale})
}
;


\node[] at (10.6,0.3) {$n$};
\node[] at (0.6,2.35) {$S_n$};
\end{tikzpicture}

\vspace{2mm}
\begin{tikzpicture}[>=stealth',main
node/.style={draw,semithick,circle,fill=white,minimum
size=2pt,inner sep=0pt},x=0.5cm,y=0.5cm]


\pgfmathsetmacro{\yscale}{0.2}


\draw[->,thick] (0,0) -- (11.3,0);
\draw[->,thick] (0,-5.5) -- (0,3);


\foreach \x in {1,...,5}
\draw[semithick] (2*\x,-0.1) -- (2*\x,0.1);

\draw[semithick] (2,-0.1) -- node[below] {{\scriptsize $50$}}
(2,0.1);
\draw[semithick] (4,-0.1) -- node[below] {{\scriptsize $100$}}
(4,0.1);
\draw[semithick] (6,-0.1) -- node[below] {{\scriptsize $150$}}
(6,0.1);
\draw[semithick] (8,-0.1) -- node[below] {{\scriptsize $200$}}
(8,0.1);
\draw[semithick] (10,-0.1) -- node[below] {{\scriptsize $250$}}
(10,0.1);

\foreach \y in {-25,-20,-15,-10,-5,0,5,10}
\draw[semithick] (-0.1,\y*\yscale) -- (0.1,\y*\yscale);
\foreach \y in {-20,-10,0,10}
\draw[semithick] (-0.1,\y*\yscale) -- node[left] {{\scriptsize $\y$}}
(0.1,\y*\yscale);


\pgfmathsetseed{\firstseed}
\draw[semithick,\firstbmcolor] (0,0)
\foreach \x in {1,...,250}
{   -- ++(0.04,{(-int(rnd*2)*2+1)*\yscale})
}
;

\pgfmathsetseed{\secondseed}
\draw[semithick,\secondbmcolor] (0,0)
\foreach \x in {1,...,250}
{   -- ++(0.04,{(int(rnd*2)*2-1)*\yscale})
}
;


\node[] at (10.6,0.3) {$n$};
\node[] at (0.6,2.35) {$S_n$};
\end{tikzpicture}
%
\begin{tikzpicture}[>=stealth',main
node/.style={draw,semithick,circle,fill=white,minimum
size=2pt,inner sep=0pt},x=0.5cm,y=0.5cm]


\pgfmathsetmacro{\yscale}{0.2*sqrt(0.2)}



\draw[->,thick] (0,0) -- (11.3,0);
\draw[->,thick] (0,-5.6) -- (0,3);


\foreach \x in {1,...,5}
\draw[semithick] (2*\x,-0.1) -- (2*\x,0.1);

\draw[semithick] (4,-0.1) -- node[below] {{\scriptsize $500$}}
(4,0.1);
\draw[semithick] (8,-0.1) -- node[below] {{\scriptsize $1000$}}
(8,0.1);
\draw[semithick] (10,-0.1) -- node[below] {{\scriptsize $1250$}}
(10,0.1);

\foreach \y in {-60,-50,-40,-30,-20,-10,0,10,20}
\draw[semithick] (-0.1,\y*\yscale) -- (0.1,\y*\yscale);
\foreach \y in {-60,-40,-20,0,20}
\draw[semithick] (-0.1,\y*\yscale) -- node[left] {{\scriptsize $\y$}}
(0.1,\y*\yscale);


\pgfmathsetseed{\firstseed}
\draw[thin,\firstbmcolor] (0,0)
\foreach \x in {1,...,1250}
{   -- ++(0.008,{(-int(rnd*2)*2+1)*\yscale})
}
;

\pgfmathsetseed{\secondseed}
\draw[thin,\secondbmcolor] (0,0)
\foreach \x in {1,...,1250}
{   -- ++(0.008,{(int(rnd*2)*2-1)*\yscale})
}
;


\node[] at (10.6,0.3) {$n$};
\node[] at (0.6,2.35) {$S_n$};
\end{tikzpicture}
\vspace{-2mm}
\end{center}
\caption[]{Two realisations (one in red, the other one in blue) of a symmetric 
random walk on $\Z$, seen at different scales. From one picture to the next, 
the horizontal scale is compressed by a factor $5$, while the vertical scale is 
compressed by a factor $\sqrt{5}$.
}
\label{fig:BM}
\end{figure}

Formally, as $n\to\infty$, the processes $\set{W_t^{(n)}}_{t\geqs0}$ should 
converge to a stochastic process $\set{W_t}_{t\geqs0}$ satisfying the following 
properties. 

\begin{enumerate}
\item	$\expec{W_t}=0$ for all $t\geqs0$;
\item	The variance of $W_t$ satisfies
\begin{equation}
 \label{mbma5}
\variance(W_t) = \lim_{n\to\infty} \Biggpar{\frac{1}{\sqrt{n}}}^2\intpart{nt} =
t\;. 
\end{equation} 
\item	By the central limit theorem, $X_{\intpart{nt}}/\sqrt{\intpart{nt}}$ 
converges in distribution to a standard normal random variable. Therefore, for 
each $t$, $W_t$ follows a normal law $\cN(0,t)$.
\item	{\it Independent increments:}\/ for all
$t>s\geqs0$, $W_t-W_s$ est independent of $\set{W_u}_{0\leqs u\leqs s}$;
\item	{\it Stationary increments:}\/ for all
$t>s\geqs0$, $W_t-W_s$ has the same distribution as $W_{t-s}$.
\end{enumerate}

This motivates the following definition. 

\begin{definition}[Brownian motion]
\label{def_mB}
{\em Standard Brownian motion} (also called the {\em standard Wiener process})
is the stochastic process $\set{W_t}_{t\geqs0}$ satisfying:
\begin{enumerate}
\item	$W_0 = 0$;
\item 	Independent increments: for all $t>s\geqs 0$, $W_t-W_s$ est
independent of $\set{W_u}_{u\leqs s}$;
\item	Stationary increments: for all $t>s\geqs0$, $W_t-W_s$ follows a normal 
law $\cN(0,t-s)$. 
\end{enumerate}
\end{definition}

\begin{theorem}[Existence of Brownian motion]
\label{thm_cW}
There exists a stochastic process $\set{W_t}_{t\geqs0}$
satisfying Definition~\ref{def_mB}, and whose trajectories
$t\mapsto B_t(\omega)$ are continuous. 
\end{theorem}
\begin{proof}\hfill
\begin{enumerate}
\item	We start by constructing $\set{W_t}_{0\leqs t\leqs1}$ 
from a collection of independent Gaussian random variables $V_1, V_{1/2}, 
V_{1/4}, V_{3/4}, V_{1/8}, \dots$, all with zero mean, where $V_1$ and 
$V_{1/2}$ have variance $1$ and each $V_{k2^{-n}}$ has variance $2^{-(n-1)}$ 
($k<2^n$ odd). 

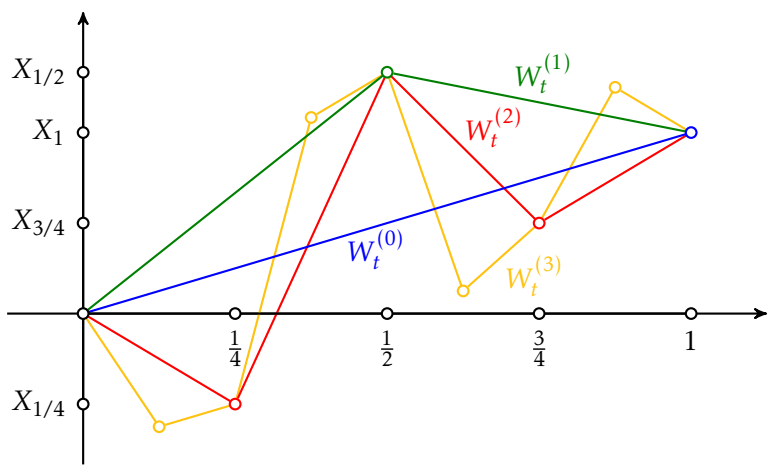
\begin{figure}
\begin{center}
\begin{tikzpicture}[-,auto,node distance=1.0cm, thick,
main node/.style={draw,circle,fill=white,minimum size=4pt,inner sep=0pt},
full node/.style={draw,circle,fill=black,minimum size=4pt,inner sep=0pt}]

  \path[->,>=stealth'] 
     (-1,0) edge (9,0)
     (0,-2) edge (0,4)
  ;
  
  \path[color=yellow!50!orange]
     (0,0) edge (1,-1.5)
     node[main node] at(1,-1.5) {}
     edge (2,-1.2)
     (2,-1.2) edge (3,2.6)
     node[main node] at(3,2.6) {}
     edge (4,3.2)
     (4,3.2) edge (5,0.3)
     node[main node] at(5,0.3) {}
     edge[below right] 
     node[label={[label distance=-0.6cm]-45:$W_t^{(3)}$}] {} (6,1.2)
     (6,1.2) edge (7,3)
     node[main node] at(7,3) {}
     edge (8,2.4)
  ;

  \path[color=red] 
     (0,0) edge (2,-1.2) 
     node[main node] at(2,-1.2) {}
     edge (4,3.2)
     (4,3.2) edge 
     node[label={[label distance=-0.6cm]45:$W_t^{(2)}$}] {} (6,1.2)
     node[main node] at(6,1.2) {}
     edge (8,2.4) 
  ;

  \path[color=green!50!black] 
     (0,0) edge (4,3.2) 
     node[main node] at(4,3.2) {}
     edge[above]
     node[label={[label distance=-0.3cm]90:$W_t^{(1)}$}] {} (8,2.4) 
  ;
  \path[color=blue] 
     (0,0) edge
     node[label={[label distance=-1.2cm]120:$W_t^{(0)}$}] {} (8,2.4) 
     node[main node] at(8,2.4) {}
  ;

  \draw (0,3.2) node[main node] {} node[left=0.08cm] {$X_{1/2}$} 
   -- (0,2.4) node[main node] {} node[left=0.08cm] {$X_{1}$} 
   -- (0,1.2) node[main node] {} node[left=0.08cm] {$X_{3/4}$} 
   -- (0,-1.2) node[main node] {} node[left=0.08cm] {$X_{1/4}$} 
   ;
     
  \draw (0,0) node[main node] {}
   -- (2,0) node[main node] {} node[below=0.08cm] {$\frac14$} 
   -- (4,0) node[main node] {} node[below=0.08cm] {$\frac12$} 
   -- (6,0) node[main node] {} node[below=0.08cm] {$\frac34$} 
   -- (8,0) node[main node] {} node[below=0.08cm] {$1$} 
   ;
\end{tikzpicture}
\end{center}
\vspace{-4mm}
 \caption[]{Construction of Brownian motion by interpolation.}
 \label{fig_bm_cons1}
\end{figure}

We first show that if  $X_s$ et $X_t$ are two random variables such that 
$X_t-X_s$ is centred, Gaussian with variance $t-s$, then there exists a random 
variable $X_{(t+s)/2}$ such that the random variables $X_t - X_{(t+s)/2}$ and 
$X_{(t+s)/2} - X_s$ are i.i.d.\ with law $\cN(0,(t-s)/2)$. If $U=X_t-X_s$ and 
$V$ is independent of $U$, with the same distribution, il suffices to define 
$X_{(t+s)/2}$ by
\begin{align}
\nonumber
X_t - X_{(t+s)/2} &= \frac{U+V}2 \\
X_{(t+s)/2} - X_s &= \frac{U-V}2\;.
\label{cW9}
\end{align}
Indeed, it is easy to check that these variables have the required 
distributions, and that they are independent, since $\expec{(U+V)(U-V)} =
\expec{U^2}-\expec{V^2} = 0$, and normal random variables are independent if 
and only if they are uncorrelated. 

Let us set $X_0=0$, $X_1=V_1$, and construct $X_{1/2}$ with the above
procedure, taking $V=V_{1/2}$. Then we construct $X_{1/4}$
with the help of $X_0$, $X_{1/2}$ and $V_{1/4}$, and so on, to obtain 
a family of variables $\set{X_t}_{t=k2^{-n},n\geqs1,k<2^n}$ such that for 
$t>s$, $X_t-X_s$ is independent of $X_s$ and has distribution $\cN(0,t-s)$. 

\item	For $n\geqs0$, let $\set{W^{(n)}_t}_{0\leqs t\leqs1}$ be the 
stochastic process with piecewise linear trajectories on intervals
$[k2^{-n},(k+1)2^{-n}]$, $k<2^n$, and such that 
$\smash{W^{(n)}}_{k2^{-n}}=X_{k2^{-n}}$ (\figref{fig_bm_cons1}).
We want to show that the sequence $\smash{W^{(n)}}(\omega)$ converges
uniformly on $[0,1]$ for any realisation $\omega$ of the $V_i$.  
We thus have to estimate
\begin{align}
\nonumber
\Delta^{(n)}(\omega) &= 
\sup_{0\leqs t\leqs 1} \bigabs{W^{(n+1)}_t(\omega)-W^{(n)}_t(\omega)} \\
\nonumber &=
\max_{0\leqs k\leqs 2^{n-1}}\max_{k2^{-n}\leqs t\leqs(k+1)2^{-n}}
\bigabs{W^{(n+1)}_t(\omega)-W^{(n)}_t(\omega)} \\
&=
\max_{0\leqs k\leqs 2^{n-1}}
\Bigabs{X_{(2k+1)2^{-(n+1)}}(\omega)
- \frac12\bigpar{X_{k2^{-n}}(\omega) 
+ X_{(k+1)2^{-n}}(\omega)}
}
\label{cW10}
\end{align}
(see \figref{fig_bm_cons2}). 
The term in the absolute value is $\frac12 V_{(2k+1)2^{-(n+1)}}$ by 
construction, c.f.~\eqref{cW9}, which is Gaussian with variance $2^{-n}$. 
Therefore,
\begin{align}
\nonumber
\bigprob{\Delta^{(n)}>\sqrt{n2^{-n}}} 
&= \biggprob{\max_{0\leqs k\leqs 2^{n-1}} \bigabs{V_{(2k+1)2^{-(n+1)}}} \geqs
2\sqrt{n2^{-n}}} \\
\nonumber
&\leqs 2\cdot2^n \int_{2\sqrt{n2^{-n}}}^\infty \e^{-x^2/2\cdot2^{-n}} 
\frac{\6x}{\sqrt{2\pi2^{-n}}} \\
&= 2\cdot2^n \int_{2\sqrt{n}}^\infty \e^{-y^2/2} 
\frac{\6x}{\sqrt{2\pi}} \leqs \const 2^n \e^{-2n}\;, 
\label{cW11}
\end{align}
and thus
\begin{equation}
\label{cW12}
\sum_{n\geqs0} \bigprob{\Delta^{(n)}>\sqrt{n2^{-n}}} 
\leqs \const \sum_{n\geqs0} (2\e^{-2})^n < \infty\;. 
\end{equation}
The Borel--Cantelli lemma shows that with probability
$1$, there exist only finitely many $n$ for which 
$\Delta^{(n)}>\sqrt{n2^{-n}}$. It follows that 
\begin{equation}
\label{cW13}
\biggprob{\sum_{n\geqs0}\Delta^{(n)} < \infty} = 1\;. 
\end{equation}
The sequence $\set{W^{(n)}_t}_{0\leqs t\leqs1}$ is thus a Cauchy sequence for 
the sup norm with probability $1$, and therefore converges
uniformly. For $t\in[0,1]$ we set
\begin{equation}
\label{cW14}
W^0_t = 
\begin{cases}
\lim_{n\to\infty} W^{(n)}_t
& \text{if the sequence converges uniformly} \\
0 & \text{otherwise (with probability $0$).}
\end{cases}
\end{equation}
It is easy to check that $B^0$ satisfies the three properties of the 
definition. 

\begin{figure}
\begin{center}
\begin{tikzpicture}[-,auto,node distance=1.0cm, thick,
main node/.style={draw,circle,fill=white,minimum size=4pt,inner sep=0pt},
full node/.style={draw,circle,fill=black,minimum size=4pt,inner sep=0pt}]

  \path[->,>=stealth'] 
     (-1,0) edge (9,0)
     (0,-1) edge (0,5)
  ;
 
  \path[color=green!50!black] 
     node[main node] at(1,0.7) {} edge (4,4) 
     node[main node] at(4,4) {} edge 
     node[label={[label distance=-0.2cm]0:$W_t^{(n+1)}$}] {} (7,2.6) 
  ;
  \path[color=blue] 
     node[main node] at(1,0.7) {} edge (4,1.65) 
     node[main node] at(4,1.65) {} edge 
     node[label={[label distance=-1.4cm]130:$W_t^{(n)}$}] {} (7,2.6) 
     node[main node] at(7,2.6) {}
  ;

  \draw (0,4.0) node[main node] {} node[left=0.1cm] {$X_{(2k+1)2^{-(n+1)}}$} 
   -- (0,2.6) node[main node] {} node[left=0.1cm] {$X_{(k+1)2^{-n}}$} 
   -- (0,1.65) node[main node] {} node[left=0.1cm]
      {$\frac{X_{k2^{-n}}+X_{(k+1)2^{-n}}}{2}$} 
   -- (0,0.7) node[main node] {} node[left=0.1cm] {$X_{k2^{-n}}$} 
   ;
     
  \draw (0,0) 
   -- (1,0) node[main node] {} node[below=0.05cm] {$k2^{-n}$} 
   -- (4,0) node[main node] {} node[below=0.05cm] {$(2k+1)2^{-(n+1)}$} 
   -- (7,0) node[main node] {} node[below=0.05cm] {$(k+1)2^{-n}$} 
   ;
\end{tikzpicture}
\end{center}
\vspace{-4mm}
 \caption[]{Computation of $\Delta^{(n)}$.}
 \label{fig_bm_cons2}
\end{figure}
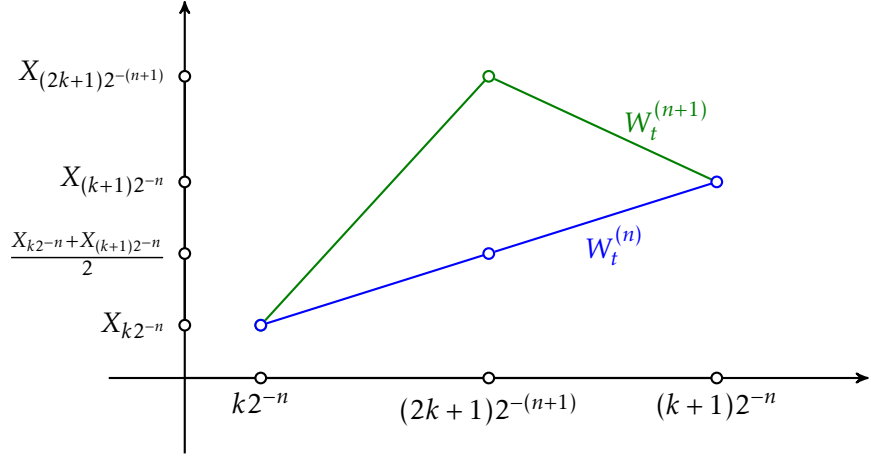

\item	To extend the process to all times, we build independent copies  
$\set{W^i}_{i\geqs0}$ and set 
\begin{equation}
\label{cW15}
W_t = 
\begin{cases}
W^0_t & 0\leqs t < 1 \\ 
W^0_1 + W^1_{t-1} & 1\leqs t < 2 \\ 
W^0_1 + W^1_1 + W^2_{t-2} & 2\leqs t < 3 \\ 
\dots &
\end{cases}
\end{equation}
This concludes the proof. \qed
\end{enumerate}
\renewcommand{\qed}{}
\end{proof}

\begin{remark}[$n$-dimensional Brownian motion]
For any $n\in\N$, one can define $n$-dimensional Brownian motion in the same 
way as in Definition~\ref{def_mB}, except that the normal laws are 
$n$-dimensional. Its components are then simply independent $1$-dimensional 
Brownian motions.  
\end{remark}


\subsection{Basic properties of Brownian motion}
\label{ssec:BM_proposition}

The following basic properties of Brownian motion follow more or less 
immediately from Definition~\ref{def_mB}.  

\begin{enumerate}
\item	{\it Markov property:}\/ For any Borel set $A\subset\R$,
\begin{equation}
\label{pW1}
\Bigpcond{W_{t+s}\in A}{W_t=x} = \int_A p(t+s,y|t,x)\6y\;,
\end{equation}
independently of $\set{W_u}_{u<t}$, with Gaussian transition probabilities
\begin{equation}
\label{pW2}
p(t+s,y|t,x) = \frac{\e^{-(y-x)^2/2s}}{\sqrt{2\pi s}}\;. 
\end{equation}
The proof follows directly from the decomposition 
$W_{t+s}=W_t + (W_{t+s}-W_t)$, where the second term is independent of the 
first one, with distribution $\cN(0,s)$. In particular, one checks the 
Chapman--Kolmogorov equation: For $t>u>s$, 
\begin{equation}
\label{pW3}
p(t,y|s,x) = \int_\R p(t,y|u,z)p(u,z|s,x)\6u\;.
\end{equation}

\item	{\it Differential property:}\/ For all $t\geqs0$,
$\set{W_{t+s}-W_t}_{s\geqs0}$ is a standard Brownian motion,
independent of $\set{W_u}_{u<t}$. 

\item	{\it Scaling proprerty:}\/ For all $c>0$,
$\set{cW_{t/c^2}}_{s\geqs0}$ is a standard Brownian motion. 

\item	{\it Symmetry:}\/ $\set{-W_t}_{t\geqs0}$ is a standard Brownian motion.

\item	{\it Gaussian process :}\/ The Wiener process is Gaussian with zero 
mean (meaning that its finite-dimensional joint distributions 
$\prob{W_{t_1} \leqs x_1, \dots, W_{t_n} \leqs x_n}$ are centred normal), 
and characterised by its covariance 
\begin{equation}
\label{pW4}
\cov\set{W_t,W_s} \equiv \expec{W_tW_s} = s \wedge t 
\end{equation}
(where $s \wedge t$ denotes the minimum of $s$ and $t$). 
\begin{proof}
For $s<t$, we have 
\begin{equation}
\label{pW5}
\expec{W_tW_s} = \expec{W_s(W_s+W_t-W_s)} = \expec{W_s^2} +
\expec{W_s(W_t-W_s)} = s\;,
\end{equation}
since the second term vanishes by the independent increments property.
\end{proof}
In fact, one can show that a centred Gaussian process whose covariance
satisfies~\eqref{pW4} is a standard Wiener process. 

\end{enumerate}

One important consequence of the scaling and independent increments properties 
is then the following. 

\begin{theorem}[Non-differentiability of Brownian paths]
\label{thm_mB2}
The paths $t\mapsto W_t(\omega)$ are almost surely nowhere 
Lipschitzian, and thus nowhere differentiable.  
\end{theorem}

\goodbreak

\begin{proof}
Fix $C<\infty$ and introduce, for $n\geqs1$, the event
\begin{equation}
A_n = \Biggsetsuch{\omega}{\exists s\in[0,1] \text{ s.t. }
\abs{W_t(\omega)-W_s(\omega)}\leqs C\abs{t-s} \text{ if } \abs{t-s}\leqs
\frac3n}\;.
\end{equation}
We have to show that $\fP(A_n)=0$ for all $n$. 
Observe that if $n$ increases, the condition gets weaker, so that $A_n\subset
A_{n+1}$. For $n\geqs3$ and $1\leqs k\leqs n-2$, define  
\begin{align}
Y_{k,n}(\omega) &= \max_{j=0,1,2}\biggset{\Bigabs{W_{(k+j)/n}(\omega) -
W_{(k+j-1)/n}(\omega)}}\;,\\
B_n &= \bigcup_{k=1}^{n-2} 
\biggsetsuch{\omega}{Y_{k,n}(\omega)\leqs
\frac{5C}{n}}\;.
\end{align}
The triangular inequality implies $A_n\subset B_n$. Indeed, 
let $\omega\in A_n$. If for instance $s=1$, then for $k=n-2$, one has 
\begin{equation}
\bigabs{W_{(n-3)/n}(\omega) - W_{(n-2)/n}(\omega)}
\leqs \bigabs{W_{(n-3)/n}(\omega) - W_1(\omega)} 
+ \bigabs{W_1(\omega) - W_{(n-2)/n}(\omega)}
\leqs C\biggpar{\frac 3n + \frac 2n}
\end{equation} 
and thus $\omega\in B_n$.
It follows from the independent increments and scaling properties that 
\begin{equation}
\fP(A_n) \leqs \fP(B_n) \leqs n \fP\Biggpar{\abs{W_{1/n}}\leqs \frac{5C}n}^3
= n \fP\Biggpar{\abs{W_1}\leqs \frac{5C}{\sqrt{n}}}^3
\leqs n \Biggpar{\frac{10C}{\sqrt{2\pi n}}}^3\;.
\end{equation} 
Therefore $\fP(A_n)\to0$ for all $n\to\infty$. But since
$\fP(A_n)\leqs\fP(A_{n+1})$ for all $n$, this implies $\fP(A_n)=0$ for all
$n$. 
\end{proof}

\begin{remark}[H\"older regularity of Brownian paths]
Even though paths of Brownian motion are nowhere differentiable, one can show 
that they do have a regularity that is better than continuity: namely, the 
paths are almost surely (locally) H\"older continuous of exponent $\alpha$ for 
any $\alpha < \frac12$. This can be shown by applying the Kolmogorov--Centsov 
continuity criterion. 
\end{remark}


\subsection{Brownian motion and heat equation}
\label{ssec:BM_heat}

Observe that the Gaussian transition probabilities~\eqref{pW3} of the Wiener 
process are, up to a scaling, equal to the heat kernel. In particular, 
$p(t,x|0,0)$ satisfies the heat equation 
\begin{align}
 \frac{\partial}{\partial t} p(t,x|0,0)
 &= \frac12 \Delta p(t,x|0,0) \\
 p(0,x|0,0) &= \delta(x)\;,
\end{align} 
where we write $\Delta$ for the second derivative with respect to $x$. This 
reflects the fact that paths of Brownian motion have the same diffusive 
behaviour as solutions of the heat equation.

Similarly, transition probabilities of $n$-dimensional Brownian motion are 
given by 
\begin{equation}
 p(t+s,y|t,x) = \frac{\e^{-\norm{y-x}^2/2s}}{(2\pi s)^{n/2}}\;, 
\end{equation} 
and satisfy therefore the $n$-dimensional heat equation. 

It is, however, important to realise that Brownian motion contains much more 
information than the solutions of the heat equation, since it gives a 
probability distribution on paths $t \mapsto W_t(\omega)$, rather than just a 
collection of probability distributions for the $W_t(\omega)$ with $t\geqs0$. 
To illustrate the difference, we discuss two examples of modifications of 
Brownian motion. 

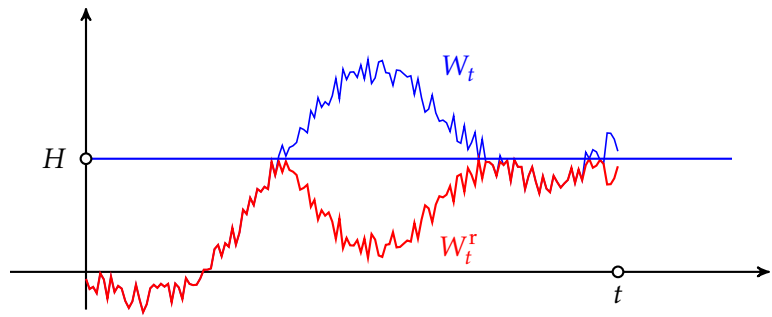
\begin{figure}
\begin{center}
\pgfmathsetmacro{\firstseed}{1234}
\begin{tikzpicture}[-,auto,node distance=1.0cm, thick,
main node/.style={draw,circle,fill=white,minimum size=4pt,inner sep=0pt},
full node/.style={draw,circle,fill=black,minimum size=4pt,inner sep=0pt},
declare function={g(\x)=2.1*sin(15*\x) - 0.9*sin(75*\x) + 0.2*rand;
f(\x)=1.5 - abs(g(\x) - 1.5); }]


\path[->,>=stealth'] 
     (-1,0) edge (9,0)
     (0,-0.5) edge (0,3.5);

\pgfmathsetseed{\firstseed}
\draw[blue,semithick,-,domain=0:7,samples=150,/pgf/fpu,
/pgf/fpu/output format=fixed] 
plot ({\x}, {g(\x)});

\pgfmathsetseed{\firstseed}
\draw[red,thick,-,domain=0:7,samples=150,/pgf/fpu,
/pgf/fpu/output format=fixed] 
plot ({\x}, {f(\x)});

\draw[thick,blue] (0,1.5) -- (8.5,1.5);

\node[blue] at (4.9,2.7) {$W_t$};
\node[red] at (4.9,0.3) {$W^{\text{r}}_t$};

\draw (0,1.5) node[main node] {} node[left=0.1cm] {$H$};
     
\draw (7,0) node[main node] {} node[below=0.05cm] {$t$};
\end{tikzpicture}
\end{center}
\vspace{-4mm}
 \caption[]{Brownian motion reflected at level $H$.}
 \label{fig_bm_reflected}
\end{figure}

\begin{example}[Reflected Brownian motion]
\label{ex:Bm_reflected} 
Denote by $W_t^{\text{r}}$ Brownian motion reflected on the line $x = H$ in 
$(t,x)$-space, for a constant $H>0$. For any $x \leqs H$ we can write 
\begin{equation}
 \bigprob{W_t^{\text{r}} \leqs x} 
 = \bigprob{W_t \leqs x} + \bigprob{W_t \geqs 2H - x}\;.
\end{equation} 
Indeed, at least heuristically, if $W_t(\omega) \leqs x$, then one obtains a 
reflected path of Brownian motion from an original path simply by reflecting 
all parts of the path that lie above the line $x = H$. If $W_t(\omega) \geqs 
2H-x$, reflecting again all parts above the line also yields a reflected path 
ending up below $H$. We thus have 
\begin{equation}
 \bigprob{W_t^{\text{r}} \leqs x} 
 = \Phi\Biggpar{\frac{x}{\sqrt{t}}} + \Phi\Biggpar{\frac{x-2H}{\sqrt{t}}}\;, 
\end{equation} 
where
\begin{equation}
 \Phi(x) = \frac{1}{\sqrt{2\pi}} \int_0^x \e^{-y^2/2}\6y
\end{equation} 
denotes the distribution function of the standard normal law. 
Note in particular that since $\Phi(-x) = 1 - \Phi(x)$, 
one has $\bigprob{W_t^{\text{r}} \leqs H} = 1$ for all $t\geqs0$, as it should 
be. Taking the derivative with respect to $x$, we obtain the density 
\begin{equation}
 p^{\text{r}}(t,x) = \frac{1}{\sqrt{2\pi t}}
 \Bigpar{\e^{-x^2/(2t)} + \e^{-(2H-x)^2/(2t)}}\;,
\end{equation} 
which solves 
\begin{align}
 \frac{\partial}{\partial t} p^{\text{r}}(t,x)
 &= \frac12 \Delta p^{\text{r}}(t,x) 
  \qquad x \leqs H\;, \\
 \nabla p(t,H) &= 0\;,
\end{align} 
that is, the heat equation with Neumann boundary conditions.
\end{example}

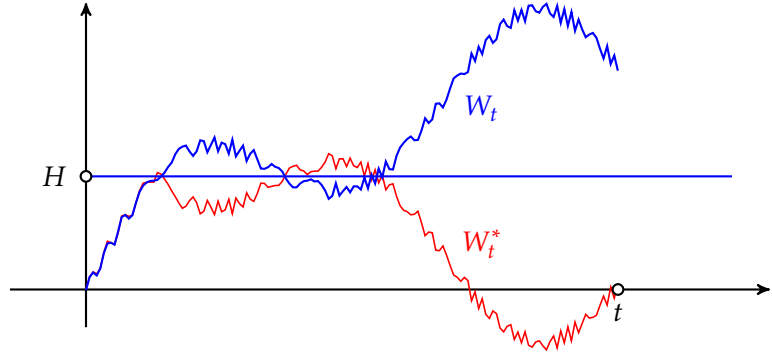
\begin{figure}
\begin{center}
\pgfmathsetmacro{\thirdseed}{1234}
\begin{tikzpicture}[-,auto,node distance=1.0cm, thick,
main node/.style={draw,circle,fill=white,minimum size=4pt,inner sep=0pt},
full node/.style={draw,circle,fill=black,minimum size=4pt,inner sep=0pt},
declare function={g(\x)=2.8*sin(15*\x) + 0.9*sin(75*\x) 
+ 0.1*sin(1200*(1.2+0.2*sin(200*\x))*\x) + 0.02*rand;
f(\x)=ifthenelse(\x<0.9,g(\x),{3.0-g(\x)});}]


\path[->,>=stealth'] 
     (-1,0) edge (9,0)
     (0,-0.5) edge (0,3.8);

\pgfmathsetseed{\thirdseed}
\draw[red,semithick,-,domain=0:7,samples=150,/pgf/fpu,
/pgf/fpu/output format=fixed] 
plot ({\x}, {f(\x)});

 \pgfmathsetseed{\thirdseed}
\draw[blue,thick,-,domain=0:7,samples=150,/pgf/fpu,
/pgf/fpu/output format=fixed] 
plot ({\x}, {g(\x)});

\draw[thick,blue] (0,1.5) -- (8.5,1.5);

\node[blue] at (5.2,2.4) {$W_t$};
\node[red] at (5.2,0.6) {$W^*_t$};

\draw (0,1.5) node[main node] {} node[left=0.1cm] {$H$};
     
\draw (7,0) node[main node] {} node[below=0.05cm] {$t$};
\end{tikzpicture}
\end{center}
\vspace{-4mm}
 \caption[]{Reflection principle.}
 \label{fig_reflection_principle}
\end{figure}

To discuss our second example, we will need the so-called~\emph{reflection 
principle}, which we give here without proof. 

\goodbreak

\begin{proposition}[Andre's reflection principle]
\label{prop:reflection_principle}
For any $H>0$ and setting $\tau=\inf\setsuch{t\geqs0}{W_t\geqs H}$, the process
\begin{equation}
\label{apr3}
W^*_t = 
\begin{cases}
W_t & \text{if $t\leqs\tau$} \\
2H - W_t & \text{if $t>\tau$}
\end{cases}
\end{equation}
is a standard Brownian motion. 
\end{proposition}

\begin{example}[Brownian motion killed upon reaching level $H$]
\label{ex:Bm_killed}  
Denote by $W^{\text{k}}_t$ Brownian motion killed at level $H > 0$, which is 
defined by 
\begin{equation}
 \bigprob{W^{\text{k}}_t \leqs x} 
 = \bigprob{W_t \leqs x, \tau > t}
\end{equation} 
for any $x \leqs H$. Note that $W^{\text{k}}_t$ is an improper random variable 
for any $t>0$, in the sense that the total probability is strictly smaller than 
$1$. Since paths of Brownian motion are continuous, we can write for any $y 
\geqs H$ 
\begin{align}
 \bigprob{W_t \geqs y}
 &= \bigprob{W_t \geqs y, \tau \leqs t} \\
 &= \bigprob{2H - W_t \leqs 2H - y, \tau \leqs t} \\
 &= \bigprob{W_t^* \leqs 2H - y, \tau \leqs t} \\
 &= \bigprob{W_t \leqs 2H - y, \tau \leqs t}\;. 
\end{align}
Setting $y = 2H - x$, this provides us with an expression for $\bigprob{W_t 
\leqs x, \tau \leqs t}$ that yields 
\begin{align}
\bigprob {W^{\text{k}}_t \leqs x}  
&= \bigprob{W_t\leqs x} - \bigprob{W_t \leqs x, \tau \leqs t} \\
&= \bigprob{W_t\leqs x} - \bigprob{W_t\geqs 2H - x}\;.
\end{align}
This gives the density 
\begin{equation}
 p^{\text{k}}(t,x) = \frac{1}{\sqrt{2\pi t}}
 \Bigpar{\e^{-x^2/(2t)} - \e^{-(2H-x)^2/(2t)}}\;,
\end{equation} 
which solves 
\begin{align}
 \frac{\partial}{\partial t} p^{\text{k}}(t,x)
 &= \frac12 \Delta p^{\text{k}}(t,x) 
  \qquad x \leqs H\;, \\
 p(t,H) &= 0\;,
\end{align} 
that is, the heat equation with Dirichlet boundary conditions.
\end{example}

\begin{exercise}[$2$-dimensional Brownian motion hitting a straight line]
Let $W_t = (W^{(1)}_t, W^{(2)}_t)$ be a standard $2$-dimensional Brownian 
motion, and let 
\begin{equation}
\tau = \inf\bigsetsuch{t>0}{\smash{W^{(1)}_t} = 1} 
\end{equation} 
be the first-hitting time of the line $\set{x=1}$. Determine the density of 
$\tau$, and use it to compute the distribution of $\smash{W^{(2)}_\tau}$. 
\end{exercise}


\section{Ito calculus}
\label{sec:Ito} 

While Brownian motion is a useful model with many interesting properties, one 
may have to deal with more general processes, such as functions of Brownian 
motion. The question then arises, what kinds of stochastic or 
partial differential equations are associated with these processes. 

\begin{example}[Brownian motion squared]
Consider Brownian motion squared. Since 
\begin{equation}
 \bigprob{W_t^2 \leqs x}
 = \bigprob{\abs{W_t} \leqs \sqrt{x}}
 = \Phi\Biggpar{\sqrt{\frac{x}{t}}}
 - \Phi\Biggpar{-\sqrt{\frac{x}{t}}}\;,
\end{equation} 
we obtain that the density of $W_t^2$ is given by 
\begin{equation}
 2\frac{\partial}{\partial x} \Phi\Biggpar{\sqrt{\frac{x}{t}}}
 = \frac{1}{\sqrt{2\pi tx}} \e^{-x/(2t)}\;.
\end{equation} 
This function should solve some PDE, which, however, 
is not straightforward to guess. In this section, we will develop methods that 
will ultimately allow to determine associated PDEs quite easily. 
\end{example}


\subsection{Ito's integral}
\label{ssec:Ito_integral} 

The key notion of stochastic calculus is the Ito integral, which allows to give 
a meaning to the quantity 
\begin{equation}
 \int_0^t f(s) \6W_s
\end{equation} 
for suitable, possibly random functions $f$. Since the Wiener process is not 
Lipschitzian, it does not have bounded variation, so that the above integral 
cannot be defined as a Riemann--Stieltjes integral. There are nowadays several 
(essentially equivalent) ways of defining the above integral, the oldest one 
going back to Kiyoshi Ito. 

Fix a Brownian motion $\set{W_t}_{t\geqs0}$. The random variables 
$\set{W_s}_{0\leqs s\leqs t}$ define an increasing sequence $\sigma$-algebras 
$\set{\cF_t}_{t\geqs0}$ (called a \emph{filtration}) that will play 
a key role in what follows. In particular, we will use the notion of random 
variables that are~\emph{measurable} with respect to a given $\cF_t$. 
Intuitively, these are exactly the random variables that depend only on the 
behaviour of the Wiener process up to time $t$. 

\begin{definition}[Ito integral of elementary functions]
\label{def_ito1}
Fix a time interval $[0,T]$. 
A random function $\set{e_t}_{t\in[0,T]}$ is called 
{\em simple} or {\em elementary} if there exists a partition
$0=t_0<t_1<\dots<t_N=T$ of $[0,T]$ such that  
\begin{equation}
\label{ito3}
e_t = \sum_{k=1}^N e_{t_{k-1}} \indicator{[t_{k-1},t_k)}(t)\;.
\end{equation}
It is called \emph{adapted} to the filtration $\set{\cF_t}_{t\geqs0}$ if each 
$e_{t_{k-1}}$ is a random variable measurable with respect to 
$\cF_{t_{k-1}}$. 
For such an elementary function, \emph{Ito's integral} is defined as  
\begin{equation}
\label{ito4}
\int_0^t e_s \6W_s = \sum_{k=1}^m e_{t_{k-1}} \bigbrak{W_{t_k}-W_{t_{k-1}}}
+ e_{t_m} \bigbrak{W_{t}-W_{t_{m}}} 
\end{equation}
for any $t\in [0,T]$, where $m$ is such that $t\in[t_m,t_{m+1})$. 
\end{definition}

One easily sees that this integral 
is a linear functional of the integrand, and is additive with respect to time 
intervals. Furthermore, since each increment $\bigbrak{W_{t_k}-W_{t_{k-1}}}$ is 
independent of $e_{t_{k-1}}$, the integral has zero expectation. The key 
property is then the following. 

\begin{lemma}[Ito isometry]
 If $\displaystyle\int_0^t \expec{e_s^2} \6s<\infty$, then
\begin{equation}
\label{ito7}
\Biggexpec{\Biggpar{\int_0^t e_s \6W_s}^2} = 
\int_0^t \expec{e_s^2} \6s\;. 
\end{equation}
\end{lemma}
\begin{proof}
Set $t_{m+1}=t$. Then  
\begin{align}
\nonumber
\Biggexpec{\Biggpar{\int_0^t e_s \6W_s}^2} 
&= \Biggexpec{\sum_{k,l=1}^{m+1} e_{t_{k-1}}e_{t_{l-1}}
\bigpar{W_{t_k}-W_{t_{k-1}}}\bigpar{W_{t_l}-W_{t_{l-1}}}} \\
\nonumber
&= \sum_{k=1}^{m+1} \bigexpec{e_{t_{k-1}}^2} 
\underbrace{\Bigexpec{\bigpar{W_{t_k}-W_{t_{k-1}}}^2}}_{t_k-t_{k-1}} 
= \int_0^t \expec{e_s^2} \6s\;.
\label{ito8}
\end{align}
We have used the property of independent increments to eliminate the terms  
$k\neq l$ from the double sum, and the fact that each $e_s$
is measurable with respect to $\cF_s$.  
\end{proof}

Ito's isometry is an isometry between the Hilbert space 
$L^2_{\text{ad}}([0,T]\times\Omega,\fP)$ of adapted square-integrable processes 
and the Hilbert space $L^2(\Omega,\fP)$ of square-integrable random variables. 
For a general square-integrable adapted process $(X_t)_{t\geqs0}$, one can find 
a sequence of elementary $e_n$ such that
\begin{equation}
\label{ito9}
\lim_{n\to\infty} \int_0^T \bigexpec{\bigpar{X_s-e^{(n)}_s}^2} \6s = 0\;.
\end{equation}
The isometry~\eqref{ito7} then shows that for any $t\in[0,T]$, the following 
limit exists in 
$L^2(\fP)$:
\begin{equation}
\label{ito10}
\lim_{n\to\infty} \int_0^t  e^{(n)}_s \6W_s \bydef 
\int_0^t  X_s \6W_s\;. 
\end{equation}
This is by definition the Ito integral of $X_s$ against $W_s$. This integral 
has the same linearity and additivity properties as integrals of elementary 
functions, and also satifies Ito's isometry
\begin{equation}
 \label{eq:Ito_isometry}
\Biggexpec{\Biggpar{\int_0^t X_s\6W_s}^2} = 
\int_0^t\expec{X_s^2}\6s\;.
\end{equation}


\subsection{Ito's formula}
\label{ssec:Ito_formula} 

Ito's formula gives a simple answer to a question we have been asking above, 
namely what kind of differential relation governs functions of Brownian motion. 
We start by a simple example, which however contains all the essential ideas of 
the general case. 

\begin{example}
Let us show that 
\begin{equation}
\label{exito1}
\int_0^t W_s \6W_s = \frac12 W_t^2 - \frac t2\;.
\end{equation}
Define the sequence of elementary functions  
$e^{(n)}_t=W_{2^{-n}\intpart{2^nt}}$. It is then sufficient to check that  
\begin{equation}
\label{exito2}
\lim_{n\to\infty} \int_0^t e^{(n)}_s \6W_s = \frac12 W_t^2 - \frac t2\;. 
\end{equation}
Write $t_k = k2^{-n}$ for $k\leqs m=\intpart{2^nt}$ and $t_{m+1}=t$. 
The definition~\eqref{ito4} implies 
\begin{align}
\nonumber
2\int_0^t e^{(n)}_s \6W_s &= 2\sum_{k=1}^{m+1}
W_{t_{k-1}}\bigpar{W_{t_k}-W_{t_{k-1}}} \\
\nonumber
&= \sum_{k=1}^{m+1}\Bigbrak{W_{t_k}^2-W_{t_{k-1}}^2
-\bigpar{W_{t_k}-W_{t_{k-1}}}^2} \\
&= W_t^2 - \sum_{k=1}^{m+1} \bigpar{W_{t_k}-W_{t_{k-1}}}^2\;.
\label{exito3}
\end{align}
Consider now the random variable 
\begin{equation}
\label{exito4}
M^{(n)}_t = \sum_{k=1}^{m+1} \bigpar{W_{t_k}-W_{t_{k-1}}}^2 - t
= \sum_{k=1}^{m+1} \bigbrak{\bigpar{W_{t_k}-W_{t_{k-1}}}^2 -
\bigpar{t_k-t_{k-1}}}\;. 
\end{equation}
Since all terms of the sum are independent and have zero expectation, we obtain 
\begin{align}
\nonumber
\bigexpec{\bigpar{M^{(n)}_t}^2} 
&= \sum_{k=1}^{m+1} \bigexpec{\bigbrak{\bigpar{W_{t_k}-W_{t_{k-1}}}^2 -
\bigpar{t_k-t_{k-1}}}^2} \\
\nonumber
&\leqs (m+1)
\bigexpec{\bigbrak{\bigpar{W_{t_1}-W_{t_0}}^2 - \bigpar{t_1-t_0}}^2} \\
\nonumber
&\leqs \const2^n\bigexpec{\bigbrak{\bigpar{W_{2^{-n}}}^2 - 2^{-n}}^2} \\
\nonumber
&= \const2^{-n}\bigexpec{\bigbrak{\bigpar{W_1}^2 - 1}^2} \\
&\leqs \const 2^{-n}\;,
\label{exito5}
\end{align}
owing to the scaling property. Therefore, $M^{(n)}_t$ converges to 
zero in $L^2$, proving~\eqref{exito1}.
\end{example}

\begin{remark}
Using more sophisticated tools from stochastic analysis, it is possible to 
prove a stronger type of convergence. Indeed, $\smash{M^{(n)}_t}$ is what is 
known as a submartingale, for which Doob's inequality yields  
\begin{equation}
\label{exito6}
\biggprob{\sup_{0\leqs s\leqs t} \bigpar{M^{(n)}_s}^2 > n^2 2^{-n}} 
\leqs 2^n n^{-2} \bigexpec{\bigpar{M^{(n)}_t}^2} 
\leqs \const n^{-2}\;.
\end{equation}
The Borel--Cantelli lemma then shows that  
\begin{equation}
\label{exito7}
\biggprob{\sup_{0\leqs s\leqs t} \bigabs{M^{(n)}_s} < n2^{-n/2},
n\to\infty}=1\;,
\end{equation}
proving almost sure convergence.   
\end{remark}

Consider now a stochastic integral of the form 
\begin{equation}
\label{fito1}
X_t = X_0 + \int_0^t f_s \6s + \int_0^t g_s \6W_s\;,
\qquad t\in[0,T]
\end{equation}
where $X_0$ is a random variable independent of the Brownian motion, and $f$ 
and $g$ are two adapted processes satisfying 
\begin{align}
\nonumber
\biggprob{\int_0^T \abs{f_s}\6s < \infty} &= 1 \\
\biggprob{\int_0^T g_s^2 \6s < \infty} &= 1\;.
\label{fito2}
\end{align}
The process~\eqref{fito1}  can also be written in differential form as 
\begin{equation}
\label{fito3}
\6X_t = f_t \6t + g_t \6W_t\;.
\end{equation}
For instance, Relation~\eqref{exito1} is equivalent to
\begin{equation}
\label{fito4}
\6\,(W_t^2) = \6t + 2 W_t \6W_t\;.
\end{equation}
Ito's formula allows to determine the effect of a change of 
variables on the stochastic integral~\eqref{fito1} in a general way. 

\begin{lemma}[Ito's formula]
\label{lem_Ito}
Let $u : [0,\infty)\times\R \to \R, (t,x)\mapsto u(t,x)$ be 
continuously differentiable with respect to $t$ and twice 
continuously differentiable with respect to $x$. Then the 
stochastic process $Y_t = u(t,X_t)$ satisfies the equation 
\begin{align}
\nonumber
Y_t ={}& Y_0 + \int_0^t \dpar ut(s,X_s) \6s 
+ \int_0^t \dpar ux(s,X_s) f_s \6s 
+ \int_0^t \dpar ux(s,X_s) g_s \6W_s \\ 
&{}+ \frac12 \int_0^t \dpar{^2u}{x^2}(s,X_s) g_s^2 \6s\;. 
\label{fito5A}
\end{align}
\end{lemma}
\begin{proof}
It suffices to prove the result for elementary integrands, and by additivity of 
the integrals, one can reduce the problem to the case of constant integrands. 
In that case, $X_t = f_0 t+g_0 W_t$ and $Y_t=u(t,f_0 t+g_0 W_t)$
can be expressed as functions of $(t,W_t)$. It suffices thus to 
consider the case $X_t=W_t$. Now for a partition
$0=t_0<t_1<\dots<t_n=t$, one has 
\begin{align}
\nonumber
u(t,W_t) - u(0,0) 
&= \sum_{k=1}^n \bigbrak{u(t_k,W_{t_k}) - u(t_{k-1},W_{t_k})} + 
\bigbrak{u(t_{k-1},W_{t_k}) - u(t_{k-1},W_{t_{k-1}})} \\
\nonumber
&=\sum_{k=1}^n \dpar ut(t_{k-1},W_{t_k}) (t_k-t_{k-1}) 
+ \dpar ux(t_{k-1},W_{t_{k-1}})(W_{t_k}-W_{t_{k-1}}) \\
\nonumber
&\phantom{=} {}+ \frac12 \dpar
{^2u}{x^2}(t_{k-1},W_{t_{k-1}})(W_{t_k}-W_{t_{k-1}})^2 + 
\Bigorder{t_k-t_{k-1}} + \Bigorder{(W_{t_k}-W_{t_{k-1}})^2} \\
\nonumber
&= \int_0^t \dpar ut(s,W_s)\6s + \int_0^t\dpar ux(s,W_s)\6W_s + 
\frac12 \int_0^t\dpar{^2u}{x^2}(s,W_s) \6s \\
&\phantom{=} {}+ \sum_{k=1}^n \frac12 
\dpar{^2u}{x^2}(t_{k-1},W_{t_{k-1}})
\bigbrak{(W_{t_k}-W_{t_{k-1}})^2-(t_k-t_{k-1})} +\order{1}\;.
\label{fito10}
\end{align}
The sum can be dealt with as $M^{(n)}_t$ in the above example when 
$t_k-t_{k-1}\to0$, c.f.~\eqref{exito4}.   
\end{proof}

\begin{remark}
\label{rem_fIto}
\begin{enumerate}
\item	Ito's formula can be written in differential form as  
\begin{equation}
\label{fito5}
\6Y_t = \dpar ut(t,X_t) \6t + \dpar ux(t,X_t) \bigbrak{f_t \6t + g_t \6W_t} 
+ \frac12 \dpar{^2u}{x^2}(t,X_t) g_t^2 \6t\;.
\end{equation}
\item	A mnemotechnic way to recover the formula is to write it in the form  
 \begin{equation}
\label{fito6}
\6Y_t = \dpar ut \6t + \dpar ux \6X_t + \frac12 \dpar{^2u}{x^2}\6X_t^2\;,
\end{equation}
where $\6X_t^2$ can be computed using the rules
\begin{equation}
\label{fito7}
\6t^2 = \6t\6W_t=0, \qquad
\6W_t^2 = \6t\;.
\end{equation}

\item	The formula can be generalised to functions
$u(t,X^{(1)}_t,\dots,X^{(n)}_t)$, depending on $n$ processes defined
by $\6X^{(i)}_t=f^{(i)}_t\6t+g^{(i)}_t\6W_t$, to 
\begin{equation}
\label{fito8}
\6Y_t = \dpar ut \6t + \sum_i\dpar u{x_i} \6X^{(i)}_t + \frac12
\sum_{i,j}\dpar{^2u}{x_i\partial x_j} \6X^{(i)}_t\6X^{(j)}_t\;,
\end{equation}
where $\6X^{(i)}_t\6X^{(j)}_t = g^{(i)}_tg^{(j)}_t\6t$.
\end{enumerate}
\end{remark}

\begin{example}
\label{ex_ito1}
\begin{enumerate}
\item	If $X_t = W_t$ and $u(x)=x^2$, one recovers Relation~\eqref{fito4}.
\item	If $\6X_t = g_t\6W_t - \frac12 g_t^2\6t$ and $u(x)=\e^x$, one obtains
\begin{equation}
\label{fito9}
\6\,(\e^{X_t}) = g_t e^{X_t}\6W_t\;.
\end{equation}
Therefore, $M_t = \exp\bigset{\gamma W_t - \gamma^2 \frac t2}$ solves the 
equation 
\begin{equation}
 \6M_t=\gamma M_t\6W_t\;.
\end{equation}
\end{enumerate}
\end{example}

\begin{exercise}[Ornstein--Uhlenbeck process]
Consider the two stochastic processes
\begin{equation}
X_t = \int_0^t \e^s \6W_s\;,
\qquad
Y_t = \e^{-t} X_t\;.
\end{equation}

\begin{enumerate}
\item	Determine $\expec{X_t}$, $\variance(X_t)$, $\expec{Y_t}$ and
$\variance(Y_t)$.
\item	Specify the law of $X_t$ and $Y_t$. 
\item	Show that $Y_t$ converges in distribution to a random variable 
$Y_\infty$ as $t\to\infty$, and specify its law. 
\item	Express $\6Y_t$ as a function of $Y_t$ and $W_t$.
\end{enumerate}
\end{exercise}

\begin{exercise}[Stratonovich integral]
\label{exo_ito4}
Let $\set{W_t}_{t\in[0,T]}$ be a standard Brownian motion. 
Let $0=t_0<t_1<\dots<t_N=T$ be a partition of $[0,T]$, and let 
\begin{equation}
e_t = \sum_{k=1}^N e_{t_{k-1}} \indicator{[t_{k-1},t_k)}(t)
\end{equation}
be an elementary function, adapted to the canonical filtration of Brownian 
motion.
The Stratonovich integral of $e_t$ is defined by 
\begin{equation}
\int_0^T e_t \circ \6W_t = \sum_{k=1}^N \frac{e_{t_k}+e_{t_{k-1}}}{2} \Delta W_k
\qquad
\text{where } \Delta W_k = W_{t_k} - W_{t_{k-1}}\;.
\end{equation}
The Stratonovich integral 
\begin{equation}
 \int_0^T X_t\circ \6W_t
\end{equation} 
of an adapted process $X_t$ is defined as the limit of the sequence 
\begin{equation}
\int_0^T e^{(n)}_t\circ \6W_t\;, 
\end{equation} 
where $e^{(n)}$ is a sequence of elementary functions 
converging to $X_t$ in $L^2$. Assume that this limit exists and is independent 
of the sequence $e^{(n)}$. 

\begin{enumerate}
\item 	Compute
\begin{equation}
\int_0^T W_t \circ \6W_t\;. 
\end{equation}

\item	Let $g:\R\to\R$ be a $\cC^2$ function, and let $X_t$ be an adapted 
process satisfying
\begin{equation}
X_t = \int_0^t g(X_s) \circ \6W_s 
\qquad
\forall t\in[0,T]\;.
\end{equation}
Let $Y_t$ be the Ito integral
\begin{equation}
Y_t = \int_0^t g(X_s) \6W_s\;.
\end{equation}
Show that 
\begin{equation}
X_t - Y_t = \frac12 \int_0^t g'(X_s)g(X_s) \6s
\qquad
\forall t\in[0,T]\;.
\end{equation}
\end{enumerate}
\end{exercise}


\subsection{Stochastic differential equations}
\label{ssec:Ito_SDEs} 

A \emph{stochastic differential equation} (SDE) is an equation of the form 
\begin{equation}
\label{edsf1}
\6X_t = f(X_t,t)\6t + g(X_t,t)\6W_t, 
\end{equation}
where $f, g:\R\times[0,T]\to\R$ are deterministic measurable functions. A 
\emph{strong solution} if this equation is by definition an adapted process 
satisfying 
\begin{equation}
\label{edsf3}
X_t = X_0 + \int_0^t f(X_s,s)\6s + \int_0^t g(X_s,s)\6W_s 
\end{equation} 
almost surely for all $t\in[0,T]$, as well as the regularity conditions
\begin{equation}
\label{edsf2}
\Biggprob{\int_0^T \abs{f(X_s,s)}\6s < \infty} =
\Biggprob{\int_0^T g(X_s,s)^2\6s < \infty} = 1\;.
\end{equation}
Here are two important examples of solvable SDEs. 

\begin{example}[Linear SDE with additive noise]
\label{ex_edslin}
Consider the linear SDE with additive noise 
\begin{equation}
\label{edsf8}
\6X_t = a(t) X_t\6t + \sigma(t)\6W_t\;,
\end{equation}
where $a$ and $\sigma$ are deterministic functions. In the particular case 
$\sigma\equiv0$, the solution can be simply written  
\begin{equation}
 \label{edsf8A}
X_t = \e^{\alpha(t)} X_0\;, \qquad
\alpha(t) = \int_0^t a(s)\,\6s\;. 
\end{equation} 
This suggest applying the method of variation of the constant, that is, looking 
for a solution of the form 
$X_t = \e^{\alpha(t)} Y_t$. Ito's formula applied to 
$Y_t=u(X_t,t)=\e^{-\alpha(t)}X_t$ gives us
\begin{equation}
 \label{edsf8B}
\6Y_t = -a(t) \e^{-\alpha(t)}X_t \,\6t 
+ \e^{-\alpha(t)}\,\6X_t
= \e^{-\alpha(t)} \sigma(t) \,\6W_t\;,
\end{equation} 
so that integrating and using $Y_0=X_0$, one gets 
\begin{equation}
 \label{edsf8C}
Y_t = X_0 + \int_0^t \e^{-\alpha(s)} \sigma(s) \,\6W_s\;. 
\end{equation} 
This finally gives the strong solution of equation~\eqref{edsf8} 
\begin{equation}
\label{edsf8D}
X_t = X_0 \e^{\alpha(t)} + 
\int_0^t \e^{\alpha(t)-\alpha(s)} \sigma(s) \,\6W_s\;. 
\end{equation}
One checks that this process indeed solves~\eqref{edsf3} by applying Ito's 
formula once again. 
Note in particular that if the initial condition $X_0$ is
deterministic, then $X_t$ follows a normal law, with expectation 
$\expec{X_t} = X_0\e^{\alpha(t)}$ and variance 
\begin{equation}
 \label{edsf8E}
\variance(X_t) = 
\int_0^t  \e^{2(\alpha(t)-\alpha(s))} \sigma(s)^2 \,\6s\;, 
\end{equation}  
as a consequence of Ito's isometry. 
\end{example}

\begin{example}[Linear SDE with multiplicative noise]
\label{ex_edslinm}
Consider the linear SDE with multiplicative noise  
\begin{equation}
\label{edsf9}
\6X_t = a(t) X_t\6t + \sigma(t)X_t\6W_t\;,
\end{equation}
with again $a$ and $\sigma$ deterministic functions. We can then write 
\begin{equation}
 \label{edsf9A}
\frac{\6X_t}{X_t} =  a(t) \6t + \sigma(t)\6W_t\;.
\end{equation} 
Integrating the left-hand side, one should get $\log(X_t)$, 
but is this compatible with Ito calculus? To check this, set 
$Y_t=u(X_t)=\log(X_t)$. Then Ito's formula gives  
\begin{align}
\nonumber
\6Y_t &= \frac{1}{X_t}\,\6X_t - \frac{1}{2X_t^2}\,\6X_t^2 \\ 
&= a(t) \6t + \sigma(t)\6W_t - \frac12 \sigma(t)^2 \6t\;.
 \label{edsf9B}
\end{align} 
Integrating and taking the exponential, one obtains the strong solution 
\begin{equation}
 \label{edsf9C} 
X_t = X_0 \exp\biggset{\int_0^t \bigbrak{a(s)-\frac12 \sigma(s)^2}\6s 
+ \int_0^t \sigma(s)\6W_s}\;.
\end{equation} 
In particular, if $a\equiv0$ and $\sigma\equiv\gamma$, one recovers 
$X_t=X_0\exp\set{\gamma W_t-\gamma^2 t/2}$, which is called the 
\defwd{geometric (or exponential) Brownian motion}.
\end{example}

\newpage 

We now state an existence and uniqueness result of solutions for a class of 
SDEs. 

\begin{theorem}[Existence and uniqueness of a strong solution]
\label{thm_edsf}
Assume the functions $f$ and $g$ satisfy the following two conditions:
\begin{enumerate}
\item	{\em Global Lipschitz condition:} there exists a constant $K$
such that 
\begin{equation}
\label{edsf5}
\abs{f(x,t) - f(y,t)} + \abs{g(x,t)-g(y,t)} \leqs K\abs{x-y}
\end{equation}
for all $x, y\in\R$ and $t\in[0,T]$. 

\item	{\em Bounded growth condition:} there exists a constant 
$L$ such that 
\begin{equation}
\label{edsf6}
\abs{f(x,t)} + \abs{g(x,t)} \leqs L(1+\abs{x})
\end{equation}
for all $x\in\R$ and $t\in[0,T]$. 
\end{enumerate}
Then the SDE~\eqref{edsf1} admits, for any square-integrable initial condition 
$X_0$, a strong solution $\set{X_t}_{t\in[0,T]}$, which is almost surely  
continuous. This solution is unique in the sense that if 
$\set{X_t}_{t\in[0,T]}$ 
and $\set{Y_t}_{t\in[0,T]}$ are two almost surely continuous solutions, then 
\begin{equation}
\label{edsf7}
\biggprob{\sup_{0\leqs t\leqs T}\abs{X_t-Y_t}>0} = 0\;. 
\end{equation}
\end{theorem}

We will omit the details of the proof of this result, which is very similar to 
corresponding proofs in the deterministic case. Uniqueness follows by 
estimating the derivative of the expected 
difference $\bigexpec{\abs{X_t-Y_t}^2}$ and applying 
Gronwall's lemma, while existence is obtained by applying a fixed-point 
argument, or more precisely by showing that the sequence of functions 
\begin{equation}
 \label{edsf8:5}
 X^{(k+1)}_t = X_0 + \int_0^t f(X^{(k)}_s,s) \,\6s
+ \int_0^t g(X^{(k)}_s,s) \,\6W_s
\end{equation} 
converges to a limit which solves the SDE. 

\begin{remark}[Weaker conditions on drift and diffusion coefficients]
The conditions on $f$ and $g$ in the above result can be relaxed to the 
following ones:
\begin{enumerate}
\item	{\em Local Lipschitz condition:} For any compact
$\cK\in\R$, there exists a constant $K=K(\cK)$ such that 
\begin{equation}
\label{edsf5B}
\abs{f(x,t) - f(y,t)} + \abs{g(x,t)-g(y,t)} \leqs K\abs{x-y}
\end{equation}
for all $x, y\in\cK$ and $t\in[0,T]$. 

\item	{\em Bounded growth condition:} There exists a constant 
$L$ such that 
\begin{equation}
\label{edsf6B}
xf(x,t) + g(x,t)^2 \leqs L^2(1+x^2)
\end{equation}
for all $x,t$. 
\end{enumerate}
Indeed, one can show that under the local Lipschitz condition, any solution 
path $X_t(\omega)$ either exists up to time $T$, or leaves any compact $\cK$ at 
a time $\tau(\omega)<T$. Therefore, there exists a random blow-up time $\tau$, 
such that either $\tau(\omega)=+\infty$ and then $X_t(\omega)$ exists up to 
time $T$, or $\tau(\omega)\leqs T$, and then $X_t(\omega)\to\pm\infty$ as
$t\to\tau(\omega)$.

Under the bounded growth condition, one shows that solution paths 
$X_t(\omega)$ cannot blow up (because the drift term does not grow fast 
enough, or pulls paths back towards the origin if $xf(x,t)$ is
negative). 
\end{remark}

\begin{exercise}
\label{exo_sde2}
Solve the SDE
\begin{equation}
\6X_t = -\frac12 X_t \6t + \sqrt{1-X_t^2} \6W_t\;,
\qquad 
X_0 = 0
\end{equation}
using the change of variables $Y=\Arcsin(X)$.
\end{exercise}

\begin{exercise}
\label{exo_sde4}
Fix $r, \alpha\in\R$. Solve the SDE
\begin{equation}
\6Y_t = r\6t + \alpha Y_t \6W_t\;,
\qquad Y_0=1
\end{equation}
by using the \lq\lq integrating factor\rq\rq\ 
$F_t = \e^{-\alpha W_t + \frac12 \alpha^2 t}$, 
and considering $X_t=F_tY_t$. 
\end{exercise}



\section{Diffusions}
\label{sec:diffusions} 

A~\defwd{diffusion} is a stochastic process solving an SDE of the form 
\begin{equation}
\label{diff00}
\6X_t = f(X_t)\6t + g(X_t)\6W_t\;,
\end{equation}
with a drift coefficient $f$ (modelling a deterministic force), and a diffusion 
coefficient $g$ (modelling a random effect such as collisions with particles of 
a fluid). When speaking of diffusions, we focus on the dependence of solutions 
on the initial condition $X_0 = x$, which is one of the main mechanisms 
creating links between SDEs and PDEs.  

\begin{definition}[Ito diffusion]
A \defwd{time-homogeneous Ito diffusion} is a stochastic process 
$\set{X_t(\omega)}_{t\geqs0}$ satisfying an SDE of the form  
\begin{equation}
\label{diffM01}
\6X_t = f(X_t)\6t + g(X_t)\6W_t\;, 
\qquad t\geqs s>0\;, \quad X_s = x\;,
\end{equation}
where $W_t$ is a standard Brownian motion of dimension $m$, and the  
\defwd{drift coefficient} $f:\R^n\to\R^n$ and \defwd{diffusion coefficient} 
$g:\R^n\to\R^{n\times m}$ are such that the SDE~\eqref{diffM01} admits a unique 
solution for all times.
\end{definition}

We will denote the solution of~\eqref{diffM01} $X^{s,x}_t$. 


\subsection{The Markov property}
\label{sec:Markov} 

Time homogeneity, that is, the fact that $f$ and $g$ do not depend on time, 
has the following important consequence. 

\begin{lemma}[Time homogeneity of the law]
\label{lem_law_diffusion} 
The processes $\set{X^{s,x}_{s+h}}_{h\geqs0}$ and $\set{X^{0,x}_h}_{h\geqs0}$
have the same distribution. 
\end{lemma}
\begin{proof}
By definition, $X^{0,x}_h$ satisfies the integral equation  
\begin{equation}
 \label{diffM02:1} 
X^{0,x}_h = x + \int_0^h f(X^{0,x}_v)\6v + \int_0^h g(X^{0,x}_v)\6W_v\;.
\end{equation} 
Furthermore, $X^{s,x}_{s+h}$ satisfies the equation
\begin{align}
\nonumber
X^{s,x}_{s+h} &= x + \int_s^{s+h} f(X^{s,x}_u)\6u + \int_s^{s+h}
g(X^{s,x}_u)\6W_u \\
&= x + \int_0^h f(X^{s,x}_{s+v})\6v + \int_0^h
g(X^{s,x}_{s+v})\6\widetilde W_v
\label{diffM02:2} 
\end{align} 
where we have used the change of variables $u=s+v$, and $\widetilde W_v
= W_{s+v} - W_s$. By the differential property, $\widetilde W_v$ is 
a standard Brownian motion, so that by uniqueness of solutions of the 
SDE~\eqref{diffM01}, the integrals~\eqref{diffM02:2} and~\eqref{diffM02:1}
have the same distribution. 
\end{proof}

We will denote $\fP^{\mskip1.5mu x}$ the probability measure on the 
$\sigma$-algebra generated by all random variables $X^{0,x}_t$, $t\geqs0$,
$x\in\R^n$, defined by 
\begin{equation}
 \label{diffM03}
\bigprobin{x}{X_{t_1}\in A_1,\dots,X_{t_k}\in A_k} 
= \bigprob{X^{0,x}_{t_1}\in A_1,\dots,X^{0,x}_{t_k}\in A_k} 
\end{equation} 
for any choice of times $0\leqs t_1 < t_2 < \dots < t_k$ and Borel sets
$A_1, \dots,A_k\subset\R^n$. Expectations with respect to $\fP^{\mskip1.5mu
x}$ will be denoted $\E^{\mskip1.5mu x}$. 

\begin{theorem}[Markov propery for Ito diffusions]
\label{thm_Markov_diffusion} 
For any bounded measurable function $\varphi:\R^n\to\R$, 
\begin{equation}
 \label{diffM04}
\bigecondin{x}{\varphi(X_{t+h})}{\cF_t}(\omega) 
= \bigexpecin{X_t(\omega)}{\varphi(X_h)}\;,
\end{equation} 
where the right-hand side denotes the function 
$\expecin{y}{\varphi(X_h)}$ evaluated at $y=X_t(\omega)$. 
\end{theorem}
\begin{proof}
Consider for $y\in\R^n$ and $s\geqs t$ the function 
\begin{equation}
 \label{diffM04:1}
F(y,t,s,\omega) = X^{t,y}_s(\omega) 
= y + \int_t^s f(X_u(\omega))\6u + \int_t^s g(X_u(\omega))\6W_u(\omega)\;. 
\end{equation} 
Note that $F$ is independent of $\cF_t$. By uniqueness of solutions of 
the SDE~\eqref{diffM01}, we have 
\begin{equation}
 \label{diffM04:2}
X_s(\omega) = F(X_t(\omega),t,s,\omega)\;. 
\end{equation} 
Let $g(y,\omega) = \varphi\circ F(y,t,t+h,\omega)$. One can check that this 
function is measurable. 
Relation~\eqref{diffM04} is thus equivalent to
\begin{equation}
 \label{diffM04:3}
\bigecond{g(X_t,\omega)}{\cF_t} = \bigexpec{\varphi\circ F(y,0,h,\omega)}
\Bigevalat{y=X_t(\omega)}\;.
\end{equation} 
We have 
\begin{equation}
 \label{diffM04:4}
 \bigecond{g(X_t,\omega)}{\cF_t} = \bigecond{g(y,\omega)}{\cF_t}
\Bigevalat{y=X_t(\omega)}\;.
\end{equation} 
Indeed, this relation is true for functions of the form
$g(y,\omega)=\phi(y)\psi(\omega)$, since 
\begin{equation}
 \label{diffM04:5}
 \bigecond{\phi(X_t)\psi(\omega)}{\cF_t}  
= \phi(X_t)\bigecond{\psi(\omega)}{\cF_t}
= \bigecond{\phi(y)\psi(\omega)}{\cF_t}
\Bigevalat{y=X_t(\omega)}\;.
\end{equation} 
It can thus be extended to any bounded measurable function, by approximating it 
by a sequence of linear combinations of functions as above. It follows from the 
independence of $F$ and $\cF_t$ that   
\begin{align}
\nonumber
\bigecond{g(y,\omega)}{\cF_t}
&= \bigexpec{g(y,\omega)} \\
\nonumber
&= \bigexpec{\varphi\circ F(y,t,t+h,\omega)} \\
&= \bigexpec{\varphi\circ F(y,0,h,\omega)}\;,
 \label{diffM04:6}
\end{align}
where the last equality follows from Lemma~\ref{lem_law_diffusion}. 
The result then follows by evaluating the last inequality at  $y=X_t$. 
\end{proof}

There exists an important generalisation of the Markov property to so-called 
stopping times. We have already encountered such a time in Andr\'e's reflection 
principle, see Proposition~\ref{prop:reflection_principle}, with the random 
time 
$\tau = \inf\setsuch{t\geqs0}{W_t\geqs H}$. The general definition of a 
stopping time is as follows.

\begin{definition}[Stopping time]
\label{def_stopping_Bt}
A \defwd{stopping time} is a random variable 
$\tau:\Omega\to[0,\infty]$ such that $\set{\tau<t}\in\cF_t$ for all 
$t\geqs0$. For such a stopping time, the \emph{pre-$\tau$} sigma algebra is 
defined by 
\begin{equation}
\label{mbta4}
\cF_\tau = \bigsetsuch{A\in\cF}{A\cap\set{\tau\leqs t}\in\cF_t\;\forall
t\geqs0} 
\end{equation} 
\end{definition}

In what follows, it will be sufficient to know that first-exit times 
\begin{equation}
 \tau = \inf\setsuch{t>0}{X_t\not\in A}
\end{equation} 
of an open or closed set $A$ are stopping times. The pre-$\tau$ sigma algebra 
is in this case the set of all events that only depend on the behaviour of the 
process as long as it stays in $A$. 

The generalisation of the Markov property to stopping times reads as follows. 

\begin{theorem}[Strong Markov property for Ito diffusions]
\label{thm_strong_Markov_diffusion} 
For any bounded, measurable function $\varphi:\R^n\to\R$ and almost 
surely finite stopping time $\tau$, 
\begin{equation}
 \label{diffM05}
\bigecondin{x}{\varphi(X_{\tau+h})}{\cF_\tau}(\omega) 
= \bigexpecin{X_\tau(\omega)}{\varphi(X_h)}\;.
\end{equation} 
\end{theorem}
\begin{proof}
The proof is a relatively direct adaptation of the previous proof. See for 
instance \cite[Theorem~7.2.4]{Oksendal}.
\end{proof}


\subsection{Semigroups and generators}
\label{sec:generator} 

\begin{definition}[Markov semi-group]
\label{def_semigroup}
To any bounded measurable function $\varphi:\R^n\to\R$, one associates for all 
$t\geqs0$ the function $P_t\varphi$ defined by 
\begin{equation}
 \label{diffsg1}
(P_t\varphi)(x) = \bigexpecin{x}{\varphi(X_t)}\;. 
\end{equation}  
The linear operator  $P_t$ is called the \defwd{Markov semi-group}
of the diffusion. 
\end{definition}

For instance, if $\varphi(x)=\indicator{A}(x)$ denotes the indicator function 
of a Borel set $A\subset\R^n$, one has 
\begin{equation}
 \label{diffsg2}
(P_t\indicator{A})(x) = \bigprobin{x}{X_t\in A}\;. 
\end{equation} 

The name semi-group is justified by the following result. 

\begin{lemma}[Semi-group property]
For any $t, h\geqs 0$, one has 
\begin{equation}
 \label{diffsg3}
P_h\circ P_t = P_{t+h}\;. 
\end{equation} 
\end{lemma}
\begin{proof}
We have
\begin{align}
\nonumber
(P_h\circ P_t)(\varphi)(x)
&= (P_h(P_t\varphi))(x) \\
\nonumber
&= \bigexpecin{x}{(P_t\varphi)(X_h)} \\
\nonumber
&= \bigexpecin{x}{\bigexpecin{X_h}{\varphi(X_t)}} \\
\nonumber
&= \bigexpecin{x}{\bigecondin{x}{\varphi(X_{t+h})}{\cF_t}} \\
\nonumber
&= \bigexpecin{x}{\varphi(X_{t+h})} \\
&= (P_{t+h}\varphi)(x)\;,
\label{diffsg3:1} 
\end{align}
where we have used the Markov property to go from the third to the fourth line.
\end{proof}

The following properties are easy to check:
\begin{enumerate}
\item	$P_t$ preserves constant functions: $P_t (c\indicator{\R^n}) 
= c\indicator{\R^n}$;
\item	$P_t$ preserves non-negative functions: $\varphi(x)\geqs0
\;\forall x \Rightarrow (P_t\varphi)(x)\geqs 0 \;\forall x$;
\item	$P_t$ is contracting (in the non-strict sense) in the $L^\infty$-norm:
\begin{equation}
 \label{diffsg4}
\sup_{x\in\R^n} \bigabs{(P_t\varphi)(x)} = \sup_{x\in\R^n} 
\bigabs{\bigexpecin{x}{\varphi(X_t)}}
\leqs \sup_{y\in\R^n}\bigabs{\varphi(y)} 
\sup_{x\in\R^n} \bigexpecin{x}{1} = \sup_{y\in\R^n}\bigabs{\varphi(y)}\;.
\end{equation} 
\end{enumerate}
The Markov semigroup is thus a positive, linear operator, which is bounded in 
$L^\infty$-norm (in fact, it has operator norm $1$). 

The semi-group property implies that the behaviour of $P_t$ on any interval 
$[0,\eps]$, with $\eps>0$ arbitrarily small, determines its behaviour for any 
$t\geqs0$. It is thus natural to consider the derivative of $P_t$ in $t=0$. 

\begin{definition}[Infinitesimal generator of an Ito diffusion]
\label{def_generateur}
The \defwd{infinitesimal generator} $\cL$ of an Ito diffusion is defined by its 
action on test functions $\varphi$ via 
\begin{equation}
 \label{diffsg5}
(\cL\varphi)(x) = \lim_{h\to0_+}
\frac{(P_h\varphi)(x) - \varphi(x)}{h}\;.
\end{equation}  
The domain of $\cL$ is by definition the set of functions $\varphi$ for which 
the limit~\eqref{diffsg5} exists for all $x\in\R^n$. 
\end{definition}

\begin{remark}
\label{rem_generator} 
Formally, Relation~\eqref{diffsg5} can be written
\begin{equation}
 \label{diffsg5A}
\cL = \dtot{P_t}{t}\Bigevalat{t=0}\;.
\end{equation} 
By the Markov property, this relation generalises to 
\begin{equation}
 \label{diffsg5B} 
\dtot{}{t} P_t = \lim_{h\to0_+} \frac{P_{t+h}-P_t}{h}
= \lim_{h\to0_+} \frac{P_h-\id}{h}P_t
= \cL P_t\;,
\end{equation}
and the semigroup can thus by formally written 
\begin{equation}
 \label{diffsg5C}
P_t = \e^{t\cL}\;. 
\end{equation} 
\end{remark}

\begin{proposition}
\label{prop_generator_Ito}
The generator of the Ito diffusion~\eqref{diffM01} is the differential 
operator
\begin{equation}
 \label{diffsg6}
\cL = \sum_{i=1}^n f_i(x) \dpar{}{x_i}
+ \frac{1}{2} \sum_{i,j=1}^n (gg^T)_{ij}(x) \dpar{^2}{x_i\partial x_j}\;. 
\end{equation} 
The domain of $\cL$ contains the set of twice continuously 
differentiable functions of compact support.  
\end{proposition}
\begin{proof}
Consider the case $n=m=1$. Let $\varphi$ be a twice continuously 
differentiable function of compact support, and let $Y_t=\varphi(X_t)$.
By Ito's formula, 
\begin{equation}
 \label{diffsg7:1}
Y_h = \varphi(X_0) + \int_0^h \varphi'(X_s)f(X_s)\6s 
+ \int_0^h \varphi'(X_s)g(X_s)\6W_s + \frac12\int_0^h
\varphi''(X_s)g(X_s)^2 \6s\;.
\end{equation}  
Taking the expectation, as the expectation of the Ito integral vanishes, one 
gets 
\begin{equation}
 \label{diffsg7:2}
\bigexpecin{x}{Y_h} = \varphi(x) + \biggexpecin{x}{
\int_0^h \varphi'(X_s) f(X_s)\6s + \frac12\int_0^h \varphi''(X_s) g(X_s)^2\6s
 }\;,
\end{equation}
so that 
\begin{equation}
 \label{diffsg7:3}
\frac{\bigexpecin{x}{\varphi(X_h)} - \varphi(x)}{h} 
= \frac 1h
\int_0^h \bigexpecin{x}{\varphi'(X_s) f(X_s)}\6s 
+ \frac1{2h}\int_0^h
\bigexpecin{x}{\varphi''(X_s) g(X_s)^2}\6s\;.
\end{equation}  
Taking the limit $h\to0_+$, we get 
\begin{equation}
 \label{diffsg7:4}
(\cL\varphi)(x) = \varphi'(x)f(x) + \frac12 \varphi''(x)g(x)^2\;. 
\end{equation} 
The cases $n\geqs2$ or $m\geqs2$ are treated similarly, using the  
multidimensional Ito formula. 
\end{proof}

\begin{example}[Generator of Brownian motion]
Let $W_t$ be an $m$-dimensional Brownian motion. This is a particular case of
diffusion, with $f=0$ and $g=\one$. Its generator is thus given by  
\begin{equation}
 \label{diffsg8}
\cL = \frac{1}{2}\sum_{i=1}^m \dpar{^2}{x_i^2} = \frac12\Delta\;. 
\end{equation} 
\end{example}


\subsection{Dynkin's formula}
\label{sec:Dynkin} 

Dynkin's formula is essentially a generalisation of the 
expression~\eqref{diffsg7:2} to stopping times. It will yield a first 
important class of links between SDEs and PDEs.

\begin{proposition}[Dynkin's formula]
Let $\set{X_t}_{t\geqs0}$ be a diffusion with generator $\cL$. Fix 
$x\in\R^n$, a stopping time $\tau$ such that $\expecin{x}{\tau}<\infty$, and
a compactly supported, twice continuously differentiable 
function $\varphi:\R^n\to\R$. Then 
\begin{equation}
 \label{diffD01}
\bigexpecin{x}{\varphi(X_\tau)} 
= \varphi(x) + \Biggexpecin{x}{\int_0^\tau (\cL\varphi)(X_s)\6s}\;. 
\end{equation}  
\end{proposition}
\begin{proof}
Consider the case $n=m=1$, $m$ being the dimension of Brownian motion.
Proceeding as in the proof of Proposition~\ref{prop_generator_Ito},
we obtain 
\begin{equation}
 \label{diffD01:1}
\bigexpecin{x}{\varphi(X_\tau)} = \varphi(x) 
+  \Biggexpecin{x}{\int_0^\tau (\cL\varphi)(X_s)\6s}
+  \Biggexpecin{x}{\int_0^\tau g(X_s)\varphi'(X_s)\6W_s}\;.
\end{equation}
It thus suffices to show that the expectation of the stochastic integral 
vanishes. For any function $h$ bounded by $M$ and any $N\in\N$, one has 
\begin{equation}
 \label{diffD01:2}
\Biggexpecin{x}{\int_0^{\tau\wedge N} h(X_s)\6W_s}
= \Biggexpecin{x}{\int_0^N \indexfct{s<\tau}h(X_s)\6W_s} = 0\;,
\end{equation} 
owing to the $\cF_s$-measurability of $\indexfct{s<\tau}$ and $h(X_s)$.
Moreover, 
\begin{align}
\nonumber
\Biggexpecin{x}{\Biggpar{\int_0^\tau h(X_s)\6W_s - \int_0^{\tau\wedge
N} h(X_s)\6W_s}^2}
&= \Biggexpecin{x}{\int_{\tau\wedge N}^\tau h(X_s)^2\6s} \\
&\leqs M^2 \bigexpecin{x}{\tau - \tau\wedge N}\;,
 \label{diffD01:3}
\end{align}
which goes to $0$ as $N\to\infty$, owing to the assumption 
$\expecin{x}{\tau}<\infty$, by Lebesgues' dominated convergence theorem. One 
can thus write  
\begin{equation}
 \label{diffD01:4}
0 = \lim_{N\to\infty}  \Biggexpecin{x}{\int_0^{\tau\wedge N} h(X_s)\6W_s}
= \Biggexpecin{x}{\int_0^{\tau} h(X_s)\6W_s}\;,
\end{equation} 
which finishes the proof, after plugging~\eqref{diffD01:4} 
into~\eqref{diffD01:1}. The proof of the general case is analogous. 
\end{proof}

Consider now the particular case where the stopping time $\tau$ is the 
first-exit time from an open bounded set $D\subset\R^n$. Assume the boundary 
value problem
\begin{align}
\nonumber
(\cL u)(x) &= \theta(x) \qquad x\in D \\
u(x) &= \psi(x) \qquad x\in\partial D
\label{diffD02}
\end{align}
admits a unique solution. This is the case if $D$, $\theta$ and $\psi$ 
are sufficiently regular. Replacing $\varphi$ by $u$ in Dynkin's formula, we 
get the relation
\begin{equation}
 \label{diffD03}
u(x) = \biggexpecin{x}{\psi(X_\tau) - \int_0^\tau \theta(X_s)\6s}\;. 
\end{equation} 
For $\psi=0$ and $\theta=-1$, $u(x)$ is equal to the expectation of $\tau$,
starting form $x$. For $\theta=0$ and $\psi$ the indicator of a subset $A$ of 
the boundary $\partial D$, $u(x)$ is the probability of leaving $D$ though $A$.
Hence, if one can solve the problem~\eqref{diffD02}, one obtains 
information on the first-exit time and location from $D$. Conversely,
simulating the expression~\eqref{diffD03} by a Monte-Carlo method, one
gets a numerical approximation of the solution of the boundary value 
problem~\eqref{diffD02}. 

\begin{example}[Mean exit time of Brownian motion from a ball]
\label{ex_MB_exit_ball}
Let $K=\setsuch{x\in\R^n}{\norm{x}<R}$ be the ball of radius $R$ centred at 
the origin. Given a point $x\in K$, let  
\begin{equation}
 \label{diffD04}
\tau_K = \inf\setsuch{t>0}{x+W_t\not\in K} 
\end{equation} 
and let 
\begin{equation}
 \label{diffD05}
\tau(N) = \tau_K\wedge N\;. 
\end{equation} 
The function $\varphi(x)=\norm{x}^2\indexfct{\norm{x}\leqs R}$ is 
compactly supported and satisfies $\Delta \varphi(x)=2n$ for all $x\in K$. One 
can furthermore extend it outsite $K$ in a smooth and compactly supported way. 
Plugging into Dynkin's formula, one gets 
\begin{align}
 \nonumber
\bigexpecin{x}{\norm{x+W_{\tau(N)}}^2} &= \norm{x}^2 + 
\biggexpecin{x}{\int_0^{\tau(N)} \frac12\Delta\varphi(W_s)\6s} \\
&= \norm{x}^2 + n \bigexpecin{x}{\tau(N)}\;.
 \label{diffD06}
\end{align} 
Since $\norm{x+W_{\tau(N)}}\leqs R$, 
letting $N$ go to infinity, one obtains by dominated convergence 
\begin{equation}
 \label{diffD07}
\bigexpecin{x}{\tau_K} = \frac{R^2 - \norm{x}^2}{n}\;. 
\end{equation} 
\end{example}

\begin{example}[Recurrence/transience of Brownian motion]
\label{ex_MB_recurrent_transient} 
Let again $K=\setsuch{x\in\R^n}{\norm{x}<R}$. We now consider
the case where $x\not\in K$, and we want to determine if
Brownian motion starting in $x$ hits $K$ almost surely, in which case it is 
called \defwd{recurrent}, or if it hits $K$ with a probability
strictly less than $1$, in which case it is called \defwd{transient}.
As for random walks, the answer depends on the dimension $n$ of space. 

We define 
\begin{equation}
 \label{diffD08}
\tau_K = \inf\setsuch{t>0}{x+W_t\in K}\;. 
\end{equation} 
For $N\in\N$, let $A_N$ be the ring 
\begin{equation}
 \label{diffD09}
A_N = \setsuch{x\in\R^n}{R<\norm{x}<2^N R}\;, 
\end{equation} 
and let $\tau$ be the first-exit time of $x+W_t$ from $A_N$. We thus have 
\begin{equation}
 \label{diffD10}
\tau = \tau_K \wedge \tau'\;, 
\qquad
\tau' = \inf\setsuch{t>0}{\norm{x+W_t}=2^NR}\;. 
\end{equation} 
Finally, let
\begin{equation}
 \label{diffD11}
p = \bigprobin{x}{\tau_K < \tau'} 
= \bigprobin{x}{\norm{x+W_\tau} = R}
= 1 - \bigprobin{x}{\norm{x+W_\tau} = 2^NR}\;.  
\end{equation} 
The spherically symmetric solutions of $\Delta\varphi=0$ are of the form
\begin{equation}
 \label{diffD12}
\varphi(x) = 
\begin{cases}
\abs{x} & \text{if $n=1$\;,} \\
-\log\norm{x} & \text{if $n=2$\;,} \\
\norm{x}^{2-n} & \text{if $n>2$\;.}
\end{cases}
\end{equation} 
For such a $\varphi$, Dynkin's formula yields 
\begin{equation}
 \label{diffD13}
\bigexpecin{x}{\varphi(x+W_\tau)} = \varphi(x)\;. 
\end{equation} 
On the other hand,  
\begin{equation}
 \label{diffD14}
\bigexpecin{x}{\varphi(x+W_\tau)} = \varphi(R) p + \varphi(2^N R) (1-p)\;. 
\end{equation} 
Solving with respect to $p$, one gets 
\begin{equation}
 \label{diffD15}
p = \frac{\varphi(x) - \varphi(2^N R)}{\varphi(R) - \varphi(2^N R)}\;. 
\end{equation} 
As $N\to\infty$, one has $\tau'\to\infty$, so that 
\begin{equation}
 \label{diffD16}
\bigprobin{x}{\tau_K < \infty} = \lim_{N\to\infty}
\frac{\varphi(x) - \varphi(2^N R)}{\varphi(R) - \varphi(2^N R)}\;. 
\end{equation} 
Consider now separately the cases $n=1$, $n=2$ and $n>2$. 
\begin{enumerate}
\item	For $n=1$, one has 
\begin{equation}
 \label{diffD17}
\bigprobin{x}{\tau_K < \infty} = \lim_{N\to\infty}
\frac{2^N R - \abs{x}}{2^N R - R} = 1\;, 
\end{equation} 
showing that Brownian motion is recurrent in dimension $1$. 
\item	For $n=2$, one has 
\begin{equation}
 \label{diffD18}
\bigprobin{x}{\tau_K < \infty} = \lim_{N\to\infty}
\frac{\log\norm{x} + N\log2 - \log R}{N\log2} = 1\;, 
\end{equation} 
showing that Brownian motion is also recurrent in dimension $2$. 
\item	For $n>2$, one has 
\begin{equation}
 \label{diffD19}
\bigprobin{x}{\tau_K < \infty} = \lim_{N\to\infty}
\frac{(2^N R)^{2-n} + \norm{x}^{2-n}}{(2^N R)^{2-n} + R^{2-n}} =
\biggpar{\frac{R}{\norm{x}}}^{n-2} < 1\;. 
\end{equation} 
Brownian motion is thus transient in dimension $n>2$. 
\end{enumerate}
\end{example}

\begin{exercise}[First-exit time of geometric Brownian motion]
\label{exo_diff2}
Consider the diffusion defined by the equation 
\begin{equation}
\6X_t = X_t \6W_t\;.
\end{equation} 

\begin{enumerate}
\item	Determine its generator $\cL$. 
\item	Find the general solution of the equation $\cL u=0$. 
\item	Determine $\probin{x}{\tau_a<\tau_b}$, where $\tau_a$ denotes the 
first-passage time of $X_t$ in $a$.

{\it Hint:} This amounts to computing $\expecin{\!x}{\psi(X_\tau)}$, where
$\tau$ is the first-exit time from $[a,b]$, and $\psi(a)=1$,
$\psi(b)=0$.
\end{enumerate}
\end{exercise}

\goodbreak

\begin{exercise}[First-exit time of geometric Brownian motion with linear drift]
\label{exo_diff3}
Consider more generally the diffusion defined by the equation 
\begin{equation}
\6X_t = r X_t \6t + X_t \6W_t\;, 
\qquad r\in\R\;.
\end{equation}

\begin{enumerate}
\item	Compute its generator $\cL$. 

\item	Show that if  $r\neq\frac12$, the general solution of the equation
$\cL u=0$ is given by 
\begin{equation}
u(x) = c_1 x^\gamma + c_2\;,
\end{equation} 
where $\gamma$ is a function of $r$ to be determined. 

\item	Assume $r<1/2$. Compute $\probin{x}{\tau_b<\tau_a}$ for $0<a<x<b$,
and then $\probin{x}{\tau_b<\tau_0}$ by letting $a$ go to $0$. Note that if 
$X_{t_0}=0$, then $X_t=0$ for all $t\geqs t_0$. Therefore, if 
$\tau_0<\tau_b$, then $X_t$ will never reach $b$. What is the probability that 
this happens?

\item	Assume now $r>1/2$. 
\begin{enumerate}
\item	Compute $\probin{x}{\tau_a<\tau_b}$ for $0<a<x<b$, and show that 
this probability goes to~$0$ as $a\to0_+$ for all $x\in]a,b[$. Conclude that 
almost surely, $X_t$ will never reach $0$ in this situation.

\item	Find $\alpha$ and $\beta$ such that $u(x)=\alpha\log x+\beta$
satisfies the problem 
\begin{equation}
\begin{cases}
(Lu)(x) = -1 & \text{if $0<x<b$\;,}\\
u(x) = 0 &\text{if $x=b$\;.}
\end{cases}
\end{equation} 
\item	Use this to compute $\expecin{\!x}{\tau_b}$. 
\end{enumerate}
\end{enumerate}
\end{exercise}


\subsection{Kolmogorov's equations}
\label{sec:Komogorov} 

A second class of links between SDEs and PDEs is given by Kolmogorov's 
equations, which are initial value problems. 

Observe that by taking the derivative of Dynkin's formula with respect to $t$, 
in the particuliar case $\tau=t$, one gets 
\begin{equation}
 \label{diffK01}
\dpar{}{t} (P_t\varphi)(x) = 
\dpar{}{t} \bigexpecin{x}{\varphi(X_t)}
= \bigexpecin{x}{(\cL\varphi)(X_t)} = (P_t\cL\varphi)(x)\;,
\end{equation} 
which can be written in compact form as 
\begin{equation}
 \label{diffK02}
\dtot{}{t} P_t = P_t \cL\;. 
\end{equation} 
We have seen in Remark~\ref{rem_generator} that one can also formally write 
$\dtot{}{t} P_t = \cL P_t$. Therefore, the operators $\cL$ et $P_t$ commute, at 
least formally. The next theorem makes this observation rigorous.

\begin{theorem}[Backward Kolmogorov equation]
Let $\varphi:\R^n\to\R$ be a compactly supported, twice continuously 
differentiable function. 
\begin{enumerate}
\item	The function 
\begin{equation}
 \label{diffK03}
u(t,x) = (P_t\varphi)(x) = \bigexpecin{x}{\varphi(X_t)} 
\end{equation} 
satisfies the initial value problem 
\begin{align}
\nonumber
\dpar ut(t,x) &= (\cL u)(t,x)\;, && t>0\;, \quad x\in\R^n\;, \\
u(0,x) &= \varphi(x)\;, && x\in\R^n\;.
\label{diffK04} 
\end{align} 
\item	If $w(t,x)$ is a bounded function, which is continuously differentiable
in $t$ and twice continuously differentiable in $x$, satisfying 
the initial value problem~\eqref{diffK04}, then $w(t,x)=(P_t\varphi)(x)$.
\end{enumerate}
\end{theorem}
\begin{proof} \hfill
\begin{enumerate}
\item	One has $u(0,x)=(P_0\varphi)(x)=\varphi(x)$ and 
\begin{align}
\nonumber
(\cL u)(t,x) 
&= \lim_{h\to0_+} \frac{(P_h\circ P_t\varphi)(x) - (P_t\varphi)(x)}{h} \\
\nonumber
&= \lim_{h\to0_+} \frac{(P_{t+h}\varphi)(x) - (P_t\varphi)(x)}{h} \\
&= \dpar{}{t} (P_t\varphi)(x) = \dpar{}{t} u(t,x)\;.
\label{diffK04:1}
\end{align}
\item	If $w(t,x)$ satisfies~\eqref{diffK04}, then one has 
\begin{equation}
 \label{diffK04:2}
\widetilde \cL w = 0 
\qquad 
\text{where } \quad \widetilde \cL w = -\dpar{w}{t} + \cL w\;. 
\end{equation} 
Fix $(s,x)\in\R_+\times\R^n$. The process $Y_t=(s-t,X^{0,x}_t)$ admits
$\widetilde \cL$ as generator. Let 
\begin{equation}
 \label{diffK04:3}
\tau_R = \inf\setsuch{t>0}{\norm{X_t}\geqs R}\;. 
\end{equation} 
Dynkin's formula shows that 
\begin{equation}
 \label{diffK04:4}
\bigexpecin{s,x}{w(Y_{t\wedge\tau_R})} 
= w(s,x) + \biggexpecin{s,x}{\int_0^{t\wedge\tau_R} (\widetilde \cL w)(Y_u)\6u}
= w(s,x)\;. 
\end{equation} 
Letting $R$ go to infinity, one obtains 
\begin{equation}
 \label{diffK04:5}
w(s,x) = \bigexpecin{s,x}{w(Y_t)} \qquad \forall t\geqs0\;.
\end{equation} 
In particular, taking $t=s$, one has 
\begin{equation}
 \label{diffK04:6}
w(s,x) = \bigexpecin{s,x}{w(Y_s)} 
= \bigexpec{w(0,X^{0,x}_s)}
= \bigexpec{\varphi(X^{0,x}_s)} 
= \bigexpecin{x}{\varphi(X_s)}\;,
\end{equation} 
as claimed. 
\end{enumerate}
\end{proof}

Note that in the case of Brownian motion, which has generator  
$\cL=\frac12\Delta$, Kolmogorov's backward equation~\eqref{diffK04} is nothing 
but the heat equation.

Since Kolmogorov's backward equation is linear, it is sufficient to solve it 
for a complete family of initial conditions $\varphi$ to determine its solution 
for all initial conditions.
A first important case occurs when one knows all eigenfunctions and eigenvalues 
of $\cL$. Then the general solution can be decomposed on a basis of 
eigenfunctions, with coefficients depending exponentially on time.

\begin{example}[Brownian motion]
Eigenfunctions of the generator $\cL = \frac12\Delta = \frac12\dtot{^2}{x^2}$ 
of 
one-dimensional Brownian motion are of the form $\e^{\icx kx}$. Decomposing the 
solution on this basis of eigenfunctions amounts to solving the heat equation 
by 
Fourier transform. One knows that the solution can be written as  
\begin{equation}
 \label{diffK05}
u(t,x) = \frac{1}{\sqrt{2\pi}} \int_\R \e^{-k^2t/2}\hat \varphi(k) \e^{\icx kx}
\6k\;, 
\end{equation}
where $\hat \varphi(k)$ is the Fourier transform of the
initial  condition. 
\end{example}

A second important case occurs when formally decomposing the initial condition
on a \lq\lq basis\rq\rq\ of Dirac distributions. In practice,
this amounts to using the notion of transition density.

\begin{definition}[Transition density]
\label{transition_density}
The diffusion $\set{X_t}_t$ is said to admit the \defwd{transition density}
$p_t(x,y)$, also written $p(y,t|x,0)$, if 
\begin{equation}
 \label{diffK06}
\bigexpecin{x}{\varphi(X_t)} = \int_{\R^n} \varphi(y)p_t(x,y)\6y  
\end{equation} 
for all bounded measurable functions $\varphi:\R^n\to\R$. 
\end{definition}

By linearity, if the transition density exists and is smooth, it satisfies  
Kolmogorov's backward equation (the generator $\cL$ acting on the variable 
$x$), with initial condition $p_0(x,y)=\delta(x-y)$. 

\begin{example}[Brownian motion and heat kernel]
In the case of one-dimensional Brownian motion, we have seen 
(c.f.~\eqref{pW2}) that the transition density is given by the heat kernel 
\begin{equation}
 \label{diffK07}
p(y,t|x,0) = \frac{1}{\sqrt{2\pi t}} \e^{-(x-y)^2/2t}\;. 
\end{equation} 
This is also the value of the integral~\eqref{diffK05} with 
$\hat\varphi(k)=\e^{-\icx k y}/\sqrt{2\pi}$, which is indeed the  
Fourier transform of $\varphi(x)=\delta(x-y)$. 
\end{example}

The~\emph{adjoint} of the generator $\cL$ is by definition the linear operator
$\cL^\dagger$ such that 
\begin{equation}
 \label{diffK08}
\pscal{\cL\phi}{\psi} = \pscal{\phi}{\cL^\dagger\psi}
\end{equation} 
for any choice of twice continuously differentiable functions 
$\phi,\psi:\R^n\to\R$, with $\phi$ compactly supported, where 
$\pscal{\cdot}{\cdot}$ denotes the usual inner product in $L^2$. Integrating
$\pscal{\cL\phi}{\psi}$ by parts twice, one obtains 
\begin{equation}
 \label{diffK09}
(\cL^\dagger\psi)(y) = \frac12\sum_{i,j=1}^n \dpar{^2}{y_i\partial y_j} 
\bigpar{(gg^T)_{ij}\psi}(y) - \sum_{i=1}^n \dpar{}{y_i} 
\bigpar{f_i\psi}(y)\;.
\end{equation} 

\begin{theorem}[Forward Kolmogorov equation]
If $X_t$ admits a smooth transition density $p_t(x,y)$, then it 
satisfies the equation 
\begin{equation}
 \label{diffK10}
\dpar{}{t} p_t(x,y) = \cL^\dagger_y p_t(x,y)\;, 
\end{equation} 
where the notation $\cL^\dagger_y$ means that $\cL^\dagger$ acts on the 
variable $y$. 
\end{theorem}
\begin{proof}
Dynkin's formula with $\tau=t$ implies 
\begin{align}
\nonumber
\int_{\R^n} \varphi(y) p_t(x,y)\6y
&= \bigexpecin{x}{\varphi(X_t)} \\
\nonumber
&= \varphi(x) + \int_0^t \bigexpecin{x}{(\cL\varphi)(X_s)} \6s \\
&= \varphi(x) + \int_0^t \int_{\R^n} (\cL\varphi)(y) p_s(x,y)\6y\;.
 \label{diffK10:1}
\end{align}
Taking the derivative with respect to time, and using~\eqref{diffK08}, we get
\begin{equation}
 \label{diffK10:2}
\dpar{}{t} \int_{\R^n} \varphi(y) p_t(x,y)\6y
= \int_{\R^n} (\cL\varphi)(y) p_t(x,y)\6y
= \int_{\R^n} \varphi(y) (\cL^\dagger_yp_t)(x,y)\6y\;,
\end{equation} 
which implies the result. 
\end{proof}

Assume the distribution of $X_0$ admits a density $\rho$ with respect to
Lebesgue measure. Then $X_t$ has the density 
\begin{equation}
 \label{diffK10:3}
\rho(t,y) = (Q_t\rho)(y) \defby \int_{\R^n} p_t(x,y)\rho(x)\6x\;. 
\end{equation} 
Applying Kolmogorov's forward equation~\eqref{diffK10}, one
obtains the \defwd{Fokker--Planck equation}
\begin{equation}
 \label{diffK10:4}
\dpar{}{t} \rho(t,y) = \cL^\dagger_y \rho(t,y)\;,
\end{equation} 
which can also be formally written
\begin{equation}
 \label{diffK10:5}
\dtot{}{t} Q_t = \cL^\dagger Q_t\;. 
\end{equation} 
The adjoint generator $\cL^\dagger$ is thus the generator of the adjoint 
semi-group $Q_t$. 

\begin{corollary}
\label{cor_stationary}
If $\rho_0(y)$ is the density of a probability measure satisfying
$\cL^\dagger\rho_0=0$, then $\rho_0$ is a stationnary measure of the 
diffusion. In other words, if the distribution of $X_0$ admits the density 
$\rho_0$, then $X_t$ admits the density $\rho_0$ for all $t\geqs0$.  
\end{corollary}

\begin{exercise}[Invariant measure of the Ornstein--Uhlenbeck process]
\label{exo_diff1}
Consider the diffusion defined by the equation 
\begin{equation}
\6X_t = -X_t\6t + \6W_t\;.
\end{equation}

\begin{enumerate}
\item	Give its generator $\cL$ and its adjoint $\cL^\dagger$.
\item	Let $\rho(x)=\pi^{-1/2}\e^{-x^2}$. Compute $\cL^\dagger\rho(x)$ and 
interpret the result. 
\end{enumerate}
\end{exercise}


\subsection{The Feynman--Kac formula}
\label{sec:Feynman-Kac} 

So far, we have encountered elliptic boundary value problems of the form $\cL 
u=\theta$, as well as parabolic evolution equations of the form $\sdpar ut=\cL 
u$. The Feynman--Kac formula will show that one can also link properties of 
diffusions with those of parabolic equations containing a term linear in $u$.
Adding a linear term to the generator can be interpreted as \lq\lq 
killing\rq\rq\ the diffusion at certain rate. The simplest case is that of a 
constant rate. Let $\zeta$ be a random variable of exponential distribution 
with parameter $\lambda$, independent of $W_t$. Set
\begin{equation}
 \label{FK01} 
\widetilde X_t = 
\begin{cases}
X_t & \text{if $t<\zeta$\;,} \\
\Delta & \text{if $t\geqs\zeta$\;,}
\end{cases}
\end{equation} 
where $\Delta$ is a \lq\lq cemetery state\rq\rq\ that has been added to
$\R^n$. One checks that owing to the exponential distribution of $\zeta$,
$\widetilde X_t$ is a Markov process on $\R^n\cup\set{\Delta}$. 
If $\varphi:\R^n\to\R$ is a bounded measurable test function, one has 
(setting $\varphi(\Delta)=0$)
\begin{equation}
 \label{FK02}
\bigexpecin{x}{\varphi(\widetilde X_t)}
=  \bigexpecin{x}{\varphi(X_t) \indexfct{t<\zeta}} 
= \prob{\zeta>t} \bigexpecin{x}{\varphi(X_t)} 
= \e^{-\lambda t} \bigexpecin{x}{\varphi(X_t)}\;.
\end{equation} 
It follows that
\begin{equation}
 \label{FK03}
\lim_{h\to0} \frac{\bigexpecin{x}{\varphi(\widetilde X_h)} - \varphi(x)}{h} 
= -\lambda \varphi(x) + (\cL\varphi)(x)\;,
\end{equation} 
which shows that the infinitesimal generator of $\widetilde X$ is the 
differential operator 
\begin{equation}
 \label{FK04}
\widetilde \cL = \cL - \lambda\;. 
\end{equation} 
More generally, if $q:\R^n\to\R$ is a continuous function bounded from below,
one can construct a random variable $\zeta$ such that  
\begin{equation}
 \label{FK05}
 \bigexpecin{x}{\varphi(\widetilde X_t)}
= \biggexpecin{x}{\varphi(X_t) \e^{-\int_0^t q(X_s)\6s}}\;.
\end{equation} 
In this case, the generator of $\widetilde X_t$ is 
\begin{equation}
 \label{FK06}
 \widetilde \cL = \cL - q\;, 
\end{equation} 
that is, $(\widetilde \cL\varphi)(x)=(\cL\varphi)(x) - q(x)\varphi(x)$. 

\begin{theorem}[Feynman--Kac formula]
\label{thm:Feynman-Kac} 
Let $\varphi:\R^n\to\R$ be a compactly supported, twice continuously 
differentiable function, and let $q:\R^n\to\R$ be a continuous function bounded 
from below. 
\begin{enumerate}
\item	The function 
\begin{equation}
 \label{FK07}
v(t,x) = \Bigexpecin{x}{\e^{-\int_0^t q(X_s)\6s}\varphi(X_t)} 
\end{equation} 
solves the initial value problem 
\begin{alignat}{2}
\nonumber
\dpar vt(t,x) &= (\cL v)(t,x) - q(x)v(t,x)\;, 
\qquad\qquad && t>0\;, \quad x\in\R^n\;, \\
v(0,x) &= \varphi(x)\;, && x\in\R^n\;.
\label{FK08} 
\end{alignat} 
\item	If $w(t,x)$ is continuously differentiable
in $t$ and twice continuously differentiable in $x$, bounded for $x$
in a compact set, and satisfies~\eqref{FK08}, then $w(t,x)$ is equal to the 
right-hand side of~\eqref{FK07}.
\end{enumerate}
\end{theorem}

\begin{proof}\hfill
\begin{enumerate}
\item	Set $Y_t=\varphi(X_t)$ and $Z_t=\e^{-\int_0^t q(X_s)\6s}$, and let 
$v(t,x)$ be given by~\eqref{FK07}. Then for $h>0$, 
\begin{align}
\nonumber
\frac{1}{h} \Bigbrak{\bigexpecin{x}{v(t,X_h)} - v(t,x)}
={}& \frac{1}{h} \Bigbrak{\Bigexpecin{x}{\bigexpecin{X_h}{Y_tZ_t}} -
\bigexpecin{x}{Y_tZ_t}} \\
\nonumber
={}& \frac{1}{h} \Bigbrak{\Bigexpecin{x}{\econdin{x}{Y_{t+h}\e^{-\int_0^t
q(X_{s+h})\6s}}{\cF_h} - Y_tZ_t}} \\
\nonumber
={}& \frac{1}{h} \Bigexpecin{x}{Y_{t+h}Z_{t+h}\e^{\int_0^h q(X_s)\6s}
-Y_tZ_t} \\
\nonumber
={}& \frac{1}{h} \Bigexpecin{x}{Y_{t+h}Z_{t+h}-Y_tZ_t} \\
{}& - \frac{1}{h} \Bigexpecin{x}{Y_{t+h}Z_{t+h}
\Bigbrak{\e^{\int_0^h q(X_s)\6s}-1}}\;.
\label{FK09:1} 
\end{align}
As $h$ goes to $0$, the first term in the last expression converges to $\sdpar 
vt(t,x)$, while the second one converges to $q(x)v(t,x)$. 

\item	If $w(t,x)$ satisfies~\eqref{FK08}, then  
\begin{equation}
 \label{FK09:2}
\widetilde \cL w = 0 
\qquad 
\text{where } \quad \widetilde \cL w = -\dpar{w}{t} + \cL w - qw\;. 
\end{equation} 
Fix $(s,x,z)\in\R_+\times\R^n\times\R^n$ and set $Z_t=z+\int_0^t
q(X_s)\6s$. The process $Y_t=(s-t,X^{0,x}_t,Z_t)$ is a diffusion with generator
\begin{equation}
 \label{FK09:3}
\widehat \cL = -\dpar{}{s} + \cL + q\dpar{}{z}\;. 
\end{equation} 
Let $\phi(s,x,z)=\e^{-z}w(s,x)$. Then $\widehat \cL\phi=0$, and Dynkin's 
formula shows that if $\tau_R$ is the first-exit time from a ball of radius $R$,
then 
\begin{equation}
 \label{FK09:4}
\bigexpecin{s,x,z}{\phi(Y_{t\wedge\tau_R})} = \phi(s,x,z)\;. 
\end{equation} 
It follows that 
\begin{align}
\nonumber
w(s,x) = \phi(s,x,0) 
&= \bigexpecin{s,x,0}{\phi(Y_{t\wedge\tau_R})} \\
\nonumber
&=
\Bigexpecin{x}{\phi
\bigpar{s-t\wedge\tau_R,X^{0,x}_{t\wedge\tau_R},Z_{t\wedge\tau_R}}} \\
&=
\Bigexpecin{x}{\e^{-\int_0^{t\wedge\tau_R} q(X_u)\6u}w(s-t\wedge\tau_R,X^{0,x}_{
t\wedge\tau_R})}\;,
 \label{FK09:5}
\end{align}
which converges to the expectation of $\e^{-\int_0^t 
q(X_u)\6u}w(s-t,X^{0,x}_t)$ 
as $R$ goes to infinity. In particular, for $t=s$ one obtains  
\begin{equation}
 \label{FK09:6}
w(s,x) = \Bigexpecin{x}{\e^{-\int_0^s q(X_u)\6u}w(0,X^{0,x}_s)}\;,
\end{equation} 
which is indeed equal to the function $v(t,x)$ defined in~\eqref{FK07}.
\qed
\end{enumerate}
\renewcommand{\qed}{}
\end{proof}

In combination with Dynkin's formula, the Feynman--Kac formula can be 
generalised to stopping times. If for instance $D\subset\R^n$ is a regular 
domain, and $\tau$ denotes the first-exit time from $D$, then under some 
regularity conditions on the functions $q, \varphi, \theta:\overline D\to\R$, 
the quantity 
\begin{equation}
 \label{FK10}
v(t,x) = \biggexpecin{x}{\e^{-\int_0^{t\wedge\tau}q(X_s)\6s}
\varphi(X_{t\wedge\tau})
- \int_0^{t\wedge\tau} \e^{-\int_0^s q(X_u)\6u}
\theta(X_{s})\6s} 
\end{equation} 
satisfies the initial value problem with boundary conditions 
\begin{alignat}{2}
\nonumber
\dpar vt(t,x) &= (\cL v)(t,x) - q(x)v(t,x) - \theta(x)\;, 
\qquad\qquad && t>0\;, \quad x\in D\;,
\\
\nonumber
v(0,x) &= \varphi(x)\;, && x\in D\;,\\
v(t,x) &= \varphi(x)\;, && x\in\partial D\;.
\label{FK11} 
\end{alignat} 
In particular, if $\tau$ is almost surely finite, taking the limit
$t\to\infty$, one obtains that
\begin{equation}
 \label{FK12}
v(x) = \biggexpecin{x}{\e^{-\int_0^{\tau}q(X_s)\6s}
\varphi(X_{\tau})
- \int_0^{\tau} \e^{-\int_0^s q(X_u)\6u}
\theta(X_{s})\6s} 
\end{equation} 
satisfies
\begin{alignat}{2}
\nonumber
 (\cL v)(x) &= q(x)v(x) + \theta(x)\;, 
\qquad\qquad && x\in D\;,\\
v(x) &= \varphi(x)\;, && x\in\partial D\;.
\label{FK13} 
\end{alignat} 
Note that in the case $q=0$, one recovers Relations~\eqref{diffD02} and 
\eqref{diffD03}.

\begin{example}
\label{ex_FeynmanKac}
Let $D=]-a,a[$ and $X_t=x+W_t$. Then $v(x)=\bigexpecin{x}{\e^{-\lambda\tau}}$
satisfies
\begin{align}
\nonumber
\frac12 v''(x) &= \lambda v(x)\;,\qquad x\in D\;,
\\
v(-a) &= v(a) = 1\;.
\label{FK14} 
\end{align} 
The general solution of the first equation is of the form 
$v(x)=c_1\e^{\sqrt{2\lambda}x} + c_2\e^{-\sqrt{2\lambda}x}$. The integration  
constants $c_1$ and $c_2$ are determined by the boundary conditions, and one 
finds
\begin{equation}
 \label{FK15}
 \bigexpecin{x}{\e^{-\lambda\tau}} = 
\frac{\cosh(\sqrt{2\lambda}\,x)}{\cosh(\sqrt{2\lambda}\,a)}\;.
\end{equation} 
Evaluating the derivative in $\lambda=0$, one obtains  
\begin{equation}
 \bigexpecin{x}{\tau}=a^2-x^2\;,
\end{equation} 
which is a particular case of~\eqref{diffD07}, but~\eqref{FK15} also determines 
all other moments of $\tau$, as well as its density. 

Solving the equation with boundary conditions $v(-a)=0$ and $v(a)=1$
one obtains
\begin{equation}
 \label{FK16}
 \bigexpecin{x}{\e^{-\lambda\tau}\indexfct{\tau_a<\tau_{-a}}} = 
\frac{\sinh(\sqrt{2\lambda}\,(x+a))}{\sinh(\sqrt{2\lambda}\cdot2a)}\;.
\end{equation} 
In particular, for $\lambda = 0$, we find 
\begin{equation}
 \bigprobin{x}{\tau_a<\tau_{-a}} = \frac{x+a}{2a}\;,
\end{equation} 
which can also be obtained directly from Dynkin's formula. However, taking 
derivatives at $\lambda = 0$, we also obtain 
\begin{align}
\nonumber
 \bigexpecin{x}{\tau\indexfct{\tau_a<\tau_{-a}}} &=
\frac{(a^2-x^2)(3a+x)}{6a}\;, 
\\
 \bigecondin{x}{\tau}{\tau_a<\tau_{-a}} &=
\frac{(a-x)(3a+x)}{3}\;.
 \label{FK17}
\end{align}  
\end{example}

\begin{remark}[Cover image]
The cover image shows a numerical solution of the heat equation, with constant 
temperatures (say $1$ and $0$) inside the Mandelbrot set, and outside the 
ellipse, after a time long enough for the solution to be close to a stationary 
state. The colour code represents the norm of the gradient of the solution 
$u(t,x)$. By the above results (assuming regularity of the Mandelbrot set 
does not pose any problem), $u(x) = \lim_{t\to\infty} u(t,x)$ is the solution 
of $\Delta u = 0$ in the domain, with boundary conditions $1$ on the Mandelbrot 
set, and $0$ on the ellipse. Therefore, $u(x)$ gives the probability, starting 
at $x$, to hit the Mandelbrot set before the ellipse. It also represents the 
electric potential in a capacitor formed by two conductors shaped like the 
boundary sets, while the colours represent the intensity of the electric field. 
One can observe a \lq\lq knife edge effect\rq\rq: the electric field is 
stronger near the sharp tips of the Mandelbrot set. A video of the convergence 
towards the equilibrium field can be found on the page 
\href{https://www.youtube.com/c/NilsBerglund}%
{\texttt{https://www.youtube.com/c/NilsBerglund}}.
\end{remark}

\begin{exercise}[The arcsine law]
\label{exo_diff5}

Let $\set{W_t}_{t\geqs0}$ be a standard Brownian motion in $\R$. Consider the 
process
\begin{equation}
 X_t = \frac{1}{t} \int_0^t \indexfct{W_s>0}\6s\;, 
\qquad
t>0\;.
\end{equation}
The aim of this exercise is to prove the \emph{arcsine law} :
\begin{equation}
\label{arcsinus} 
\bigprob{X_t < u} = \frac{2}{\pi} \Arcsin\bigpar{\sqrt{u}}\;, 
\qquad
0\leqs u\leqs 1\;.
\end{equation}

\begin{enumerate}
\item	What does the variable $X_t$ represent?
\item	Show that $X_t$ is equal in distribution to $X_1$ for all $t>0$. 
\item	Fix $\lambda>0$. 
For $t>0$ and $x\in\R$, one defines  
\begin{equation}
v(t,x) = \Bigexpec{\e^{-\lambda\int_0^t \indexfct{x+ W_s>0}\6s}}
\end{equation}
and its Laplace transform 
\begin{equation}
g_\rho(x) = \int_0^\infty v(t,x) \e^{-\rho t} \6t\;,
\qquad
\rho>0\;.
\end{equation}
Show that 
\begin{equation}
g_\rho(0) = \biggexpec{\frac{1}{\rho+\lambda X_1}}\;.
\end{equation}

\item	Compute $\dpar vt(t,x)$ using the Feynman--Kac formula. 

\item	Compute $g_\rho''(x)$. Conclude that $g_\rho(x)$ satisfies a 
second-order ODE with piecewise constant coefficients. Show that its general 
solution is given by 
\begin{equation}
 g_\rho(x) = A_\pm + B_\pm \e^{\gamma_\pm x} + C_\pm \e^{-\gamma_\pm x}
\end{equation}
with constants $A_\pm, B_\pm, C_\pm, \gamma_\pm$ depending on the 
sign of $x$. 

\item	Determine these constants by using the fact that $g_\rho$ should be 
bounded, continuous in $0$, and that $g_\rho'$ should be continuous in $0$. 
Conclude that $g_\rho(0)=1/\sqrt{\rho(\lambda+\rho)}$. 

\item	Prove~\eqref{arcsinus} by using the identity 
\begin{equation}
\frac{1}{\sqrt{1+\lambda}} = 
\sum_{n=0}^\infty (-\lambda)^n 
\frac1\pi \int_0^1 \frac{x^n}{\sqrt{x(1-x)}}\6x\;.
\end{equation}
\end{enumerate}
\end{exercise}



\chapter{Invariant measures for SDEs}
\label{chap:inv_meas} 

We consider in this chapter SDEs in $\R^n$ of the form 
\begin{equation}
\label{eq:SDE_inv} 
 \6X_t = f(X_t)\6t + g(X_t)\6W_t\;,
\end{equation} 
where the drift coefficient $f$ and the diffusion coefficient $g$ are such that 
there exists a unique strong solution, which is global in time. Then it is 
natural to ask the following questions: 

\begin{enumerate}
\item 	Does the diffusion~\eqref{eq:SDE_inv} admit an invariant probability 
measure? 
\item 	If so, is this measure unique? 
\item 	If so, does any initial distribution converge to the invariant measure?
\item 	If so, how fast does this convergence occur? For which distance? Can 
one obtain explicit bounds on the speed of convergence?
\end{enumerate}

Various methods have been derived to answer these questions, each one having 
its advantages and drawbacks. Some methods are easier to use and provide, for 
instance, convergence to an invariant distribution, but without any bound on 
the speed of convergence, while others may provide sharp bounds, but are 
limited to specific sets of initial distributions. In what follows, we are 
going to present a few selected examples of these methods, which have been 
chosen because they proved useful in particular applications. But one should 
keep in mind that there exist many more approaches. 

In what follows, it will by useful to employ the notation 
\begin{equation}
 P_t(x,A) = (P_t \indicator{A})(x) = \probin{x}{X_t \in A} 
\end{equation} 
for the Markov semigroup, where $A$ is any Borel set in $\R^n$. Given a 
probability measure $\mu$, we write 
\begin{equation}
 (\mu P_t)(A) 
 = \int_{\R^n} \mu(\6x) P_t(x,A) 
\end{equation} 
instead of $(Q_t\mu)(A)$ for the action of the adjoint semigroup, because it 
is reminiscent of matrix multiplication used for Markov chains on finite sets. 
A measure is invariant if it satisfies 
\begin{equation}
 (\mu P_t)(A) = \mu(A)
\end{equation} 
for all $t\geqs 0$ and all Borel sets $A\subset\R^n$. 

\begin{definition}[Feller property]
A semigroup $(P_t)_{t\geqs0}$ is said to heave the~\emph{Feller property} if 
$P_tf$ is bounded and continuous whenever $f$ is bounded and continuous. 
\end{definition}

A useful standard result in our situation is then the following. 

\begin{proposition}[Condition for Feller property]
Any diffusion $(X_t)_{t\geqs0}$ solving an SDE with globally Lipschitz 
coefficients has the Feller property. 
\end{proposition}

For a proof, see for instance~\cite[Lemma~8.1.4]{Oksendal}, 
or~\cite[Theorem~IX.2.5]{Revuz_Yor}. The global Lipschitz condition can often 
be relaxed to a local condition by working with appropriate stopping times 
(that is, by considering the diffusion up to its first exit from a sequence 
of balls of growing radius).

\begin{remark}[Strong Feller property]
The semigroup $P_t$ is said to have the~\emph{strong Feller property} if $P_tf$ 
is bounded and continuous whenever $f$ is bounded and measurable, but not 
necessarily continuous. A sufficient condition for a diffusion to satisfy the 
strong Feller property is the~\emph{ellipticity} condition on the drift 
coefficient
\begin{equation}
 \pscal{\xi}{g(x)g(x)^T \xi} \geqs c \norm{\xi^2} 
 \qquad \forall \xi\in\R^d\;.
\end{equation} 
This condition can be relaxed to~\emph{hypoellipticity} (H\"ormander 
condition). 
\end{remark}


\section{Existence of invariant probability measures}
\label{sec:inv_meas_gen} 


\subsection{Some basic examples}
\label{ssec:inv_meas_ex} 

Corollary~\ref{cor_stationary} shows that the density $\rho$ of an invariant 
probability measure should satisfy $\cL^\dagger\rho = 0$, where $\cL^\dagger$ 
is the adjoint generator. Cases where this equation can be solved are rare, but 
one important example is given by gradient SDEs. 

\begin{example}[Gradient system]
\label{ex:gradient_SDE} 
Consider the SDE 
\begin{equation}
\label{eq:SDE_gradient} 
 \6X_t = -\nabla V(X_t)\6t + \sqrt2 \6W_t\;,
\end{equation} 
where $V: \R^n \to \R$ is bounded below, and satisfies 
\begin{equation}
 \int_{\R^n} \e^{-V(x)} \6x < \infty\;.
\end{equation} 
The generator of the diffusion can be written in the two equivalent ways 
\begin{equation}
 \cL = \Delta - \nabla V \cdot \nabla 
 = \e^V \nabla \cdot \e^{-V} \nabla
\end{equation} 
(the factor $\sqrt{2}$ in~\eqref{eq:SDE_gradient} avoids a factor $\frac12$ in 
front of the Laplacian). Integrating by parts twice, we find 
\begin{equation}
 \pscal{f}{\cL g} 
 = - \int \e^{-V(x)} \nabla(f(x) \e^{V(x)}) \cdot \nabla g(x) \6x
 = \pscal{\cL^\dagger f}{g}\;,
\end{equation} 
with the adjoint generator given by 
\begin{equation}
 \cL^\dagger f = \nabla\cdot\bigpar{\e^{-V} \nabla (\e^{V}f)}\;.
\end{equation} 
In view of Corollary~\ref{cor_stationary}, this shows that 
\begin{equation}
 \rho(x) = \frac1{\cZ} \e^{-V(x)}\;, 
 \qquad \cZ = \int \e^{-V(x)} \6x
\end{equation} 
is the density of an invariant probability measure of the 
diffusion~\eqref{eq:SDE_gradient}, since it satisfies $\cL^\dagger \rho = 0$. 
\end{example}

Next, we discuss two very simple examples for which existence or uniqueness of 
an invariant probability measure fails.

\begin{example}[Brownian motion]
The transition density of Brownian motion in $\R^n$ is a Gaussian of variance 
$t$, as we have seen for instance in Section~\ref{ssec:BM_heat}. Therefore, for 
any fixed $x\in\R^n$, we have 
\begin{equation}
 \lim_{n\to\infty} p(t,x | 0,0) = 0\;,
\end{equation} 
which is not a normalisable measure. Therefore, Brownian motion in $\R^n$ does 
not admit an invariant probability measure (though it does admit invariant 
measures, which are simply all multiples of Lebesgue measure). 
\end{example}

\begin{example}[Non-irreducible SDE]
Consider the SDE in $\R^2$ 
\begin{align}
 \6X_t &= - X_t \6t + \6W_t \\
 \6Y_t &= (Y_t - Y_t)^3 \6t\;.
\end{align}
The diffusion admits three extremal invariant measures, given by 
\begin{align}
 \pi_{-}(\6x,\6y) &= \frac{1}{\sqrt{\pi}}\e^{-x^2}\6x \, \delta_{-1}(\6y)\;, \\
 \pi_{0}(\6x,\6y) &= \frac{1}{\sqrt{\pi}}\e^{-x^2}\6x \, \delta_{0}(\6y)\;, \\
 \pi_{+}(\6x,\6y) &= \frac{1}{\sqrt{\pi}}\e^{-x^2}\6x \, \delta_{1}(\6y)\;.
\end{align} 
This is because the $x$- and $y$-components do not interact with each other, 
and the process $X_t$ is an Ornstein--Uhlenbeck process with Gaussian invariant 
measure, while $Y_t$ is deterministic, with three invariant points located at 
$\pm1$ and $0$. Note that in addition, any convex combination of $\pi_-$, 
$\pi_0$ and $\pi_+$ is also an invariant probability measure ($\pi_-$, 
$\pi_0$ and $\pi_+$ are called extremal because they do not admit a 
nontrivial decomposition). 
\end{example}

The two last examples illustrate a general fact about invariant 
probability measures of Markov processes, namely that their existence requires 
two properties to hold:

\begin{enumerate}
\item 	There should be a mechanism preventing all the probability mass of 
going to infinity. More precisely, a positive recurrence property is required 
to hold, that is, the return time to some bounded set should have finite 
expectation. 

\item 	There should be a mechanism making the diffusion irreducible, that is, 
there should not be any non-trivial invariant sets.  
\end{enumerate}

These two conditions are analogous to those that one finds for Markov chains on 
a countable space. The main difference is that discrete-time Markov chains 
require an aperiodicity condition to hold in addition. However, this is 
specific to discrete time, and no such condition is necessary as soon as 
transition times between different points are sufficiently random. 


\subsection{The Krylov--Bogoliubov criterion}
\label{ssec:inv_KB} 

A general criterion for existence of invariant measures, going back to Krylov 
and Bogoliubov, is based on the notion of ergodic averages (or Cesaro means). 
Given an initial point $X_0\in\R^n$, consider the family of measures  
\begin{equation}
\label{eq:measure_KB} 
 \biggsetsuch{\frac{1}{T} \int_0^T P_t(X_0,\cdot)\6t}{T\geqs 1}\;.
\end{equation} 
In a similar way as for Markov chains, if these ergodic averages converge to a 
limiting probability measure, then this limit should be invariant. A 
convergence criterion is given by thightness. 

\begin{definition}[Tightness]
A family $\set{\mu_t}$ of probability measures on $\R^n$ is~\emph{tight} if for 
any $\delta>0$, there exists a compact set $K\subset\R^n$ such that 
$\mu_t(K)\geqs1-\delta$ for all $t$. 
\end{definition}

Then one has the following existence result, see for 
instance~\cite[Corollary~3.1.2]{DaPrato_Zabczyk_Ergodicity}.

\begin{proposition}[Existence of an invariant probability measure]
 If the family of measures~\eqref{eq:measure_KB} is tight, then there exists an 
invariant probability measure. 
\end{proposition}

While this criterion may be used to obtain abstract existence results for 
invariant probability measures, it is not so easy to apply because it requires 
some a priori knowledge on the semi--group $P_t$. Therefore, in what follows, 
we will rather discuss more practical criteria for analysing invariant measure.


\section{The Lyapunov function approach by Meyn and Tweedie}
\label{sec:Meyn_Tweedie} 

We present here some results of the approach based on Lyapunov functions, as
developed in~\cite{Meyn_Tweedie_1993b} for continuous-time Markov processes. 
This approach provides a relatively easy way of proving global existence of 
solution, existence of an invariant measure, and some convergence results, 
provided one can guess an appropriate Lyapunov function. 

\begin{definition}[Norm-like function]
A function $V: \R^n\to\R_+$ is called \emph{norm-like} if 
\begin{equation}
 \lim_{\norm{x}\to\infty} V(x) = +\infty\;.
\end{equation} 
This means that the level sets $\setsuch{x\in\R^n}{V(x) \leqs h}$ are 
precompact for any $h>0$.  
\end{definition}

In the case of ordinary differential equations, Lyapunov functions are 
norm-like functions that decrease along orbits of the dynamical system, at 
least when starting far away from the origin. The application to SDEs uses 
quite similar ideas, where the time derivative along orbits is replaced by the 
action of the generator. In the following, we will present several of these 
results, without giving detailed proofs. A detailed proof of a quite 
general existence and convergence result, due to Martin Hairer and Jonathan 
Mattingly, will be discussed in Section~\ref{ssec:Hairer_Mattingly}.


\subsection{Non-explosion and Harris recurrence criteria}
\label{ssec:MT_Harris} 

A first application of Lyapunov functions is a relatively easy criterion for 
existence of global in time solutions. 

\begin{theorem}[Non-explosion criterion]
Assume that there exist a norm-like function $V$ and constants $c, d>0$ such 
that 
\begin{equation}
\label{eq:MT_CD0} 
 \bigpar{\cL V}(x) \leqs c V(x) + d
\end{equation} 
for all $x\in \R^n$. Then 
\begin{enumerate}
\item 	The SDE~\eqref{eq:SDE_inv} admits global in time solutions for any 
starting point $x\in\R^n$.

\item 	There exists an almost surely finite random variable $D$ such that 
\begin{equation}
 V(X_t) \leqs D\e^{ct} \qquad \forall t\geqs 0\;.
\end{equation} 
The random variable $D$ satisfies the bound 
\begin{equation}
\label{eq:MT_thm1_1} 
 \bigprobin{x}{D \geqs a} \leqs \frac{V(x)}{a}
 \qquad 
 \forall a > 0\;, \quad \forall x\in\R^n\;.
\end{equation} 

\item	The expectation $\bigexpecin{x}{V(X_t)}$ is finite for all $x\in\R^n$ 
and all $t\geqs0$, and satisfies 
\begin{equation}
\label{eq:MT_thm1_2} 
 \bigexpecin{x}{V(X_t)} \leqs \e^{ct} V(x)\;.
\end{equation} 
\end{enumerate}
\end{theorem}

\begin{proof}[\Sketch]
Consider first the case $c = d = 0$. Then Ito's formula 
(cf.~Lemma~\ref{lem_Ito}) yields 
\begin{align}
 \bigexpecin{x}{V(X_t)} 
 &= V(x) + \biggexpecin{x}{\int_0^t (\cL V)(X_s) \6s} \\
 &\leqs V(x)\;.
\end{align}
This proves~\eqref{eq:MT_thm1_2}, as well as global existence of the solution. 
The bound~\eqref{eq:MT_thm1_1} follows from a slightly more sophisticated 
stochastic analysis argument, using the fact that $\e^{-ct}V(X_s)$ is a 
supermartingale. 

The case $c > 0$ and $d=0$ follows in a similar way from the Feynman--Kac 
formula (see Theorem~\ref{thm:Feynman-Kac}), while the other cases can be 
reduced to already treated cases by changing the Lyapunov function. 
\end{proof}

\begin{remark}[Stopping times]
To rule out the possibility of finite-time blow-up, the actual argument given 
in~\cite[Theorem~2.1]{Meyn_Tweedie_1993b} uses a more careful computation based 
on the process killed when leaving a large ball, whose radius is then sent to 
infinity. This requires in particular using Dynkin's formula instead of Ito's 
formula. 
\end{remark}

The following result gives a condition under which solutions remain bounded 
almost surely (a property called~\emph{non-evanescence} 
in~\cite{Meyn_Tweedie_1993b}), which is slightly stronger 
than~\eqref{eq:MT_CD0}. For a proof, 
see~\cite[Theorem~3.1]{Meyn_Tweedie_1993b}. 

\begin{theorem}[Non-evanescence condition]
\label{thm:MT_non-evanescence} 
Assume there exist a compact set $C\subset\R^n$, a constant $d>0$, and a 
norm-like function $V$ such that 
\begin{equation}
\label{eq:MT_CD1} 
 \bigpar{\cL V}(x) \leqs d \indicator{C}(x) 
\end{equation} 
for all $x\in\R^n$. Then 
\begin{equation}
 \biggprob{\lim_{t\to\infty} \norm{X_t} = \infty} = 0\;.
\end{equation} 
\end{theorem}

A discussed above, in order to obtain existence of an invariant measure, we will 
need some stronger form of recurrence condition. Recall that a discrete-space 
Markov chain is said to be~\emph{recurrent} if it almost surely returns to its 
starting point, and thus visits this point infinitely often. Such a property 
cannot hold for continuous-space processes, since sets of measure $0$ may never 
be hit. The relevant concept is given by~\emph{Harris recurrence}. 

\begin{definition}[Harris recurrence]
The diffusion $(X_t)_{t\geqs0}$ is called~\emph{Harris recurrent} if there 
exists a $\sigma$-finite measure $\mu$ such that whenever $\mu(A) > 0$, one has 
for all $x\in\R^n$ 
\begin{equation}
 \bigprobin{x}{\tau_A < \infty} = 1
 \qquad 
 \text{where $\tau_A = \inf\setsuch{t\geqs0}{X_t\in A}$\;.}
\end{equation} 
Equivalently, the diffusion is Harris recurrent if there 
exists a $\sigma$-finite measure $\nu$ such that whenever $\nu(A) > 0$, one has 
for all $x\in\R^n$ 
\begin{equation}
 \bigprobin{x}{\eta_A = \infty} = 1
 \qquad 
 \text{where $\eta_A = \int_0^\infty \indicator{X_t\in A}\6t$\;.}
\end{equation} 
\end{definition}

The equivalence of the two definitions is well-known for Markov chains on 
countable spaces, and a proof in the general case can be found for instance 
in~\cite[Theorem~1.1]{Meyn_Tweedie_1992_resolvent}. The interest of this 
definition is the following classical 
result~\cite{Azema_Duflo_Revuz69,Getoor_79}. 

\begin{proposition}[Existence of an invariant measure]
\label{prop:MT_existence_inv_meas} 
If $(X_t)_{t\geqs0}$ is Harris recurrent, then it admits an (essentially 
unique) invariant measure $\pi$. 
\end{proposition}

One way of showing Harris recurrence is to use the (somewhat tricky) 
notion of \emph{petite sets}. 
\begin{definition}[Petite set]
Let $a$ be a probability distribution on $\R_+$, and define a Markov kernel 
$K_a$ by 
\begin{equation}
 K_a(x,A) = \int_0^\infty P_t(x,A) a(\6t)
\end{equation}
for any Borel set $A\subset\R^n$. 
Let $\ph_a$ be a nontrivial measure on $\R^n$. Then a non-empty Borel set 
$C\in\R^n$ is called \emph{$\ph_a$-petite} if $K_a(x,A) \geqs \ph_a(A)$ for all 
$x\in C$.  
\end{definition}

The intuition behind this definition is as follows. If we choose, say, $a = 
\delta_1$, then $K_a(x,A) = P_1(x,A)$. The condition $K_a(x,A) \geqs \ph_a(A)$ 
then requires that the probability of being in the set $A$ at time $1$ is 
bounded below by a measure $\ph_a(A)$ independent of the starting point $x$ in 
the petite set (the measure $\ph_a$ need not be a probability mesure). This 
would be quite restrictive, but the definition allows to replace $P_1(x,A)$ by 
an average of $P_t(x,A)$ over all times $t\geqs0$, a much weaker requirement.
Petite sets allow us to give a first condition for a diffusion to be Harris 
recurrent. 

\begin{theorem}[Harris recurrence condition]
\label{thm:MT_Harris_rec}
If all compact subsets of $\R^n$ are petite and~\eqref{eq:MT_CD1} holds for a 
compact set $C\subset\R^n$, a constant $d>0$, and a norm-like function $V$, 
then the process $(X_t)_{t\geqs0}$ is Harris recurrent, and therefore admits an 
essentially unique invariant measure $\pi$. 
\end{theorem}

This result follows from Theorem~\ref{thm:MT_non-evanescence}, combined 
with~\cite[Theorem~5.1]{Meyn_Tweedie_1993a}, which gives an analogous statement 
in the discrete-time case.


\subsection{Positive Harris recurrence and existence of an invariant 
probability}
\label{ssec:MT_positive_rec} 

Combining Theorem~\ref{thm:MT_Harris_rec} with 
Proposition~\ref{prop:MT_existence_inv_meas}, we obtain a condition for the 
existence of an invariant measure. This measure need not have finite mass, 
preventing it from being normalisable to yield an invariant probability 
measure. This motivates the following definition. 

\begin{definition}[Positive Harris recurrence]
Let $(X_t)_{t\geqs0}$ be a Harris recurrent diffusion, with invariant measure 
$\pi$. If $\pi(\R^n) < \infty$, then $(X_t)_{t\geqs0}$ is said to 
be~\emph{positive Harris recurrent}. 
\end{definition}

\begin{remark}[Link with expected return times]
A diffusion satisfying $\bigexpecin{x}{\tau_A} < \infty$ for any $x\in\R^n$ and 
any Borel set $A$ such that $\mu(A) > 0$ for a $\sigma$-finite measure $\mu$ is 
positive Harris recurrent. Indeed, in this case, $\tau_A$ is almost surely 
finite, so that the chain is Harris recurrent. Furthermore, given $x\in\R^n$ and 
$A$ with $\mu(A)>0$, the probability measure $\pi_A(x,\cdot)$ given by
\begin{equation}
 \pi_A(x,B) = \frac{1}{\bigexpecin{x}{\tau_A}} 
 \Biggexpecin{x}{\int_0^{\tau_A} \indicator{X_t\in B} \6t}
\end{equation}
for any Borel set $B\subset\R^n$ can easily be shown to be invariant. By 
essential uniqueness, any invariant measure is thus normalisable.  
\end{remark}

A sufficient condition for the process $(X_t)_{t\geqs0}$ to be positive 
Harris recurrent will thus automatically be a sufficient condition for the 
existence of an invariant probability measure. A first such condition is 
provided by the following result, which 
is~\cite[Theorem~4.2]{Meyn_Tweedie_1993b}. 

\begin{theorem}[Positive Harris recurrence condition]
Assume there exist constants $c, d > 0$, a function $f: \R^n\to [1,\infty)$, a 
closed petite set $C\subset\R^n$, and a positive function $V$ such that
\begin{equation}
\label{eq:MT_CD2} 
 (\cL V)(x) \leqs - cf(x) + d\indicator{C}(x)  
\end{equation} 
for all $x\in\R^n$. Assume furthermore that $V$ is bounded on $C$. Then the 
process $(X_t)_{t\geqs0}$ is positive Harris recurrent, and therefore admits an 
invariant probability measure $\pi$. Furthermore, 
\begin{equation}
 \pscal{\pi}{f} 
 := \bigexpecin{\pi}{f}  
 = \int_{\R^n} f(x)\pi(\6x) < \infty\;.
\end{equation} 
\end{theorem}

While Condition~\eqref{eq:MT_CD2} is usually easy to check for a given $f$ and 
$V$, the requirement that $C$ be petite may be harder to verify. Fortunately, 
for Feller diffusions, there exists an alternative criterion for the existence 
of an invariant measure, which avoids having to check that $C$ is petite. The 
following result is~\cite[Theorem~4.5]{Meyn_Tweedie_1993b}. 

\begin{theorem}[Existence of invariant probability measures for Feller 
diffusions]
Assume that the diffusion $(X_t)_{t\geqs0}$ has the Feller property, and 
that~\eqref{eq:MT_CD2} holds for a compact $C\subset \R^n$. Then the diffusion 
admits an invariant probability measure $\pi$. Furthermore, any invariant 
probability $\pi$ satisfies $\pscal{\pi}{f} \leqs d/c$. 
\end{theorem}


\subsection{Convergence to the invariant probability measure}
\label{ssec:MT_convergence} 

Once the existence of an invariant probability measure $\pi$ is established, 
the next natural question is whether the distribution of $(X_t)_{t\geqs0}$ will 
converge to $\pi$, at least under some conditions on the law of $X_0$. There 
are many choices of norms quantifying such a convergence, and results exist for 
several of them. Here we consider the following weighted norm on signed 
measures. 

\begin{definition}[$f$-norm of a signed measure]
Let $\mu$ be a signed measure on $(\R^n,\cB)$, and let $f:\R^n\to[1,\infty)$ be 
a measurable function. Then we define the $f$-norm of $\mu$ by  
\begin{equation}
 \norm{\mu}_f 
 = \sup_{g: \abs{g} < f} \bigabs{\pscal{\mu}{g}}\;,
 \qquad
 \pscal{\mu}{g} 
 := \bigexpecin{\mu}{g} 
 = \int_{\R^n} g(x)\mu(\6x)\;,
\end{equation} 
where the supremum runs over all measurable functions such that $\abs{g(x)} 
\leqs f(x)$ for all $x\in\R^n$.
\end{definition}

\begin{definition}[Exponential ergodicity]
Given a measurable function $f:\R^n\to[0,\infty)$, a diffusion process 
$(X_t)_{t\geqs0}$ admitting an invariant probability measure $\pi$ is 
called~\defwd{$f$-exponentially ergodic} if there exist $\beta>0$ and a function 
$B:\R^n\to\R_+$ such that 
\begin{equation}
 \norm{P_t(x,\cdot) - \pi}_f \leqs B(x) \e^{-\beta t}
\end{equation}
for any $x\in\R^n$ and $t\geqs0$.  
\end{definition}

The following result, which is~\cite[Theorem~6.1]{Meyn_Tweedie_1993b}, provides 
a condition on Lyapunov functions that guarantees exponential ergodicity. 

\begin{theorem}[Condition for exponential ergodicity]
Assume there exist a norm-like function $V$, and constants $c>0$, $d\in\R$ such
that the diffusion $(X_t)_{t\geqs0}$ satisfies the condition 
\begin{equation}
\label{eq:MT_CD3} 
 (\cL V)(x) \leqs - cV(x) + d  
\end{equation} 
for all $x\in\R^n$. Assume further that all compact $K\subset\R^n$ are petite 
for some discrete-time Markov chain $(X_{n\Delta})_{n\geqs0}$. Then the 
diffusion is exponentially ergodic. More precisely, there exist constants 
$\beta, b > 0$ such that 
\begin{equation}
 \norm{P_t(x,\cdot) - \pi}_{1+V} \leqs b \bigpar{1 + V(x)} \e^{-\beta t}
\end{equation}
for any $x\in\R^n$ and $t\geqs0$.  
\end{theorem}

Table~\ref{tab:MT} summarises the main results on invariant measures stated in 
this section. 
Again, while condition~\eqref{eq:MT_CD3} is usually easy to check for a given 
Lyapunov function $V$, the requirement that all compact subsets be petite is 
often more difficult to verify. This is why we present in the next section an 
alternative approach, due to Martin Hairer and Jonathan Mattingly, that 
provides an exponential ergodicity criterion in a slightly different norm, and 
avoids any condition on sets being petite. 

\begin{table}
\begin{center}
\begin{tabular}{|l|l|}
\hline
\textbf{Criterion} & \textbf{Property} \\
\hline \hline 
\rule{0pt}{2.5ex}
$\bigpar{\cL V}(x) \leqs c V(x) + d$ 
& Non-explosion \\
\hline 
\rule{0pt}{2.5ex}
$\bigpar{\cL V}(x) \leqs d \indicator{C}(x)$ & \\
\quad $C$ compact & Non-evanescence \\ 
\quad $C$ compact, all compact $K$ petite 
& Harris recurrence, existence of invariant measure \\
\hline 
\rule{0pt}{2.5ex}
$\bigpar{\cL V}(x) \leqs -c f(x) + d \indicator{C}(x)$ & \\
\quad $C$ closed petite & Positive Harris recurrence \\
\quad or $C$ compact and diffusion Feller & and existence of invariant 
probability measure \\
\hline
\rule{0pt}{2.5ex}
$\bigpar{\cL V}(x) \leqs -c V(x) + d$ & \\
\quad all compact $K$ petite for $(X_{n\Delta})_{n\geqs0}$ & 
Exponential ergodicity \\
\hline 
\end{tabular}
\vspace{-3mm}
\end{center}
\caption[]{Summary of the main results of the Meyn--Tweedie theory for 
diffusions. Here $V$ is a norm-like function, $c$ and $d$ are positive 
constants, and $C$ and $K$ are subsets of $\R^n$.}
\label{tab:MT} 
\end{table}

\begin{exercise}[One-dimensional diffusion]
Consider a one-dimensional diffusion of the form 
$\6X_t = -V'(X_t)\6t + \sqrt{2}\6W_t$ (cf.~Example~\ref{ex:gradient_SDE}). 
Assume that $V(x)$ behaves like $x^\alpha$ for some $\alpha > 0$ for $x$ large.
Taking $V$ as a Lyapunov function, discuss what one can say on convergence to 
the invariant probability measure depending on the value of $\alpha$.  
\end{exercise}


\subsection{A simplified proof by Hairer and Mattingly}
\label{ssec:Hairer_Mattingly} 

We present here a convergence criterion from~\cite{Hairer_Mattingly_11}, which 
applies to discrete-time Markov processes. This is, however, not really a 
restriction, because if $P_t$ is the semi-group of a diffusion 
$(X_t)_{t\geqs0}$, then for any $\delta>0$, $P_\delta$ generates the embedded 
discrete-time Markov chain $(X_{\delta n})_{n\in\N}$ obtained by restricting 
$t$ to integer multiples of $\delta$. An invariant measure $\pi$ of the 
diffusion is clearly also invariant for the discrete-time Markov chain, and it 
is not difficult to convert a convergence result in discrete time to a 
convergence result in continuous time. To avoid confusion, we will denote the 
discrete semi-group by $\cP$, and its actions on bounded measurable functions 
$f$ and signed measures $\mu$ by 
\begin{align}
 \bigpar{\cP f}(x) &= \bigexpecin{x}{f(X_\delta)} 
 := \int_{\R^n} f(y) \cP(x,\6y)\;, \\
 \bigpar{\mu\cP}(A) &= \bigprobin{\mu}{X_\delta\in A} 
 := \int_{\R^n} \cP(x,A) \mu(\6x)\;.
\end{align}
The convergence result of~\cite{Hairer_Mattingly_11} requires two rather simple 
conditions on $\cP$. The first one is a discrete-time analogue 
of~\eqref{eq:MT_CD3}, which guarantees that the Markov process does not escape 
to infinity. 

\begin{assumption}[Geometric drift condition]
\label{ass:geom_drift} 
There exist a function $V:\R^n\to[0,\infty)$ and constants $d\geqs0$ and 
$\gamma\in(0,1)$ such that 
\begin{equation}
\label{eq:geom_drift} 
\bigpar{\cP V}(x) \leqs \gamma V(x) + d
\end{equation} 
for all $x\in\R^n$. 
\end{assumption}

Note that $\gamma$ is the discrete-time analogue of $\e^{-\delta c}$ in 
continuous time, so that~\eqref{eq:geom_drift} is indeed the discrete anlogue 
of~\eqref{eq:MT_CD3}. The second condition is a form of irreducibility 
condition, which is analogous to the conditions on sets being petite that we 
encountered above, but simpler to verify. It is also similar to what is 
known as~\emph{Doeblin condition} in the Markov chain literature. 

\begin{assumption}[Minorisation condition]
\label{ass:minorisation} 
Let $C = \setsuch{x\in\R^n}{V(x) < R}$ for some $R > 2d(1-\gamma)^{-1}$. Then 
there exists $\alpha\in(0,1)$ and a probability measure $\nu$ such that 
\begin{equation}
\label{eq:minorisation} 
 \inf_{x\in C} \cP(x,A) \geqs \alpha \nu(A)
\end{equation} 
holds for all Borel sets $A\subset\R^n$. 
\end{assumption}

Under these two conditions, the main result of~\cite{Hairer_Mattingly_11} is 
the following statement, which provides both existence of a unique invariant 
measure and convergence to this measure. Convergence takes place in the 
weighted supremum norm defined by 
\begin{equation}
 \norm{f}_{1+V} = \sup_{x\in\R^n} \frac{\abs{f(x)}}{1 + V(x)}\;.
\end{equation} 

\begin{theorem}[
Exponential ergodicity in discrete time]
\label{thm:Hairer_Mattingly} 
If Assumptions~\ref{ass:geom_drift} and~\ref{ass:minorisation} hold, then $\cP$ 
admits a unique invariant probability measure $\pi$. Furthermore, there exist 
constants $M > 0$ and $\gamma\in(0,1)$ such that 
\begin{equation}
\label{eq:expo_ergodicity_HM} 
 \norm{\cP^n f - \pscal{\pi}{f}}_{1+V}
 \leqs M \gamma^n \norm{f - \pscal{\pi}{f}}_{1+V}
\end{equation} 
holds for all measurable functions $f:\R^n\to\R$ such that $\norm{f}_{1+V} < 
\infty$.
\end{theorem}

Since the proof of this result is rather elementary (and also quite elegant), 
we will provide the details of it in the remainder of this section. Our 
presentation follows closely the article~\cite{Hairer_Mattingly_11}. We first 
recall the definition of the total variation distance between probability 
measures.

\begin{definition}[Total variation distance]
Let $\mu$ and $\nu$ be two probability measures on a measure space $(\cX,\cF)$. 
The~\emph{total variation distance} between $\mu$ and $\nu$ is defined as 
\begin{equation}
\label{eq:def_TVdist} 
 \normTV{\mu-\nu} = 2 \sup\Bigsetsuch{\abs{\mu(A)-\nu(A)}}{A\in\cF}\;.
\end{equation} 
\end{definition}

It is known that the total variation distance between $\mu$ and $\nu$ coincides 
with the $L^1$-distance: 

\begin{lemma}
For any probability measures $\mu$ and $\nu$ on $(\cX,\cF)$, one has 
\begin{equation}
 \normTV{\mu-\nu} = \int_{\cX} \abs{\mu - \nu}(\6x)\;.
\end{equation}
\end{lemma}
\begin{proof}
Consider the case where the measures have densities with respect to Lebesgue 
measure, and let $B = \bigsetsuch{x\in\R^n}{\mu(x) > \nu(x)}$.
Note that since $\mu$ and $\nu$ are probability measures, we have 
\begin{equation}
 0 \leqs \mu(B) - \nu(B) 
 = 1 - \mu(B^c) - \bigbrak{1 - \nu(B^c)} 
 = \nu(B^c) - \mu(B^c)\;,
\end{equation} 
which implies 
\begin{equation}
\label{eq:proof_TV} 
 \int_{\cX} \abs{\mu - \nu}(\6x)
 = (\mu - \nu)(B) + (\nu - \mu)(B^c) 
 = 2(\mu - \nu)(B)\;.
\end{equation} 
For any Borel set $A\in\cF$, we can write 
\begin{equation}
 \mu(A) - \nu(A) 
 \leqs \int_{A\cap B} (\mu - \nu)(\6x) 
 \leqs \int_B (\mu - \nu)(\6x)
 = \mu(B) - \nu(B)\;,
\end{equation} 
since $\mu - \nu$ is negative on $A \cap B^c$, and positive on $A^c\cap B$. A 
similar argument yields 
\begin{equation}
 \nu(A) - \mu(A) 
 \leqs \int_{B^c} (\nu - \mu)(\6x)
 = \nu(B^c) - \mu(B^c)
 = \mu(B) - \nu(B)\;.
\end{equation} 
We have thus shown the $\abs{\mu(A) - \nu(A)} \leqs (\mu-\nu)(B)$, where 
equality holds whenever $A = B$ or $A = B^c$. The result thus follows 
from~\eqref{eq:proof_TV}. 
\end{proof}

The main idea of the proof of Theorem~\ref{thm:Hairer_Mattingly} is to work with 
a whole family of equivalent norms. Instead of just $\norm{f}_{1+V}$, we 
thus consider the norms $\norm{f}_{1+\beta V}$ where $\beta>0$ is a scale 
parameter. We also consider the dual metric on probability measures given by
\begin{align}
 \rho_\beta(\mu,\nu) 
 &= \sup_{f: \norm{f}_{1+\beta V} \leqs 1} \int_{\R^n} f(x) (\mu - \nu)(\6x)\\
 &= \sup_{f: \norm{f}_{1+\beta V} \neq 0} 
 \frac{1}{\norm{f}_{1+\beta V}}\int_{\R^n} f(x) (\mu - \nu)(\6x)\;,
\label{eq:rho_beta} 
\end{align} 
which is in fact equivalent to the weighted total variation distance given by
\begin{equation}
 \rho_\beta(\mu,\nu)
 = \int_{\R^n} \bigpar{1+\beta V(x)}\abs{\mu-\nu}(\6x)\;.
\end{equation} 
In particular, for $\beta = 0$, we have $\rho_0(\mu,\nu) = \normTV{\mu-\nu}$.  
The supremum in~\eqref{eq:rho_beta} is attained for 
$f(x) = \brak{1 + \beta V(x)}\brak{\indicator{B}(x) - \indicator{B^c}(x)}$, 
with $B$ as in the proof of the above lemma. 

The key result for the distance $\rho_\beta$ is the following. 

\begin{proposition}[Contraction estimate in $\rho_\beta$ distance]
\label{prop:contract_rhobeta} 
If Assumptions~\ref{ass:geom_drift} and~\ref{ass:minorisation} hold, then there 
exist $\bar\alpha\in(0,1)$ and $\beta > 0$ such that 
\begin{equation}
 \rho_\beta(\mu\cP,\nu\cP) \leqs \bar\alpha \rho_\beta(\mu,\nu)
\end{equation} 
holds for all probability measures $\mu, \nu$ on $\R^n$. More precisely, for 
any $\alpha_0\in(0,\alpha)$ and any $\gamma_0\in(\gamma+2dR^{-1},1)$, one can 
choose 
\begin{equation}
 \beta = \frac{\alpha_0}{d}\;, \qquad 
 \bar\alpha = \Bigpar{1 - (\alpha-\alpha_0)} \vee 
\frac{2+R\beta\gamma_0}{2+R\beta}\;.
\end{equation}
\end{proposition}

To prove this result, we introduce an alternative definition of $\rho_\beta$. 
Consider first the function
\begin{equation}
 d_\beta(x,y)
 = 
 \begin{cases}
  0 & \text{if $x=y$\;,} \\
  2 + \beta V(x) + \beta V(y) & \text{if $x\neq y$\;,}
 \end{cases}
\end{equation} 
which can easily be checked to be a metric on $\R^n$. This metric induces the 
Lipschitz seminorm
\begin{equation}
\label{def:seminormf} 
 \normDgamma{f}_\beta = \sup_{x\neq y} \frac{\abs{f(x)-f(y)}}{d_\beta(x,y)}\;,
\end{equation} 
and a dual metric on probability measures given by
\begin{equation}
 \rho^*_\beta(\mu,\nu) 
 = \sup_{f: \normDgamma{f}_\beta \leqs 1} \int_{\R^n} f(x) (\mu - 
\nu)(\6x)\;.
\end{equation} 
Note that the supremum is taken on a different set of functions 
than in~\eqref{eq:rho_beta}. 

\begin{lemma}[Equivalence of norms]
\label{lem:equiv_norms_HM} 
We have 
\begin{equation}
 \normDgamma{f}_\beta = \inf_{c\in\R} \norm{f + c}_{1+\beta V}\;.
\end{equation} 
In particular, $\rho^*_\beta = \rho_\beta$. 
\end{lemma}

\begin{proof}
We first note that since $\abs{f(x)} \leqs \norm{f}_{1+\beta V}(1+\beta V(x))$ 
for all $x\in\R^n$, we have 
\begin{equation}
 \frac{\abs{f(x)-f(y)}}{2 + \beta V(x) + \beta V(y)}
 \leqs \frac{\abs{f(x)} + \abs{f(y)}}{2 + \beta V(x) + \beta V(y)}
 \leqs \norm{f}_{1+\beta V} 
\end{equation} 
for all $x, y\in\R^n$, so that $\normDgamma{f}_\beta 
\leqs \norm{f}_{1+\beta V}$. It follows from the 
definition~\eqref{def:seminormf} of $\normDgamma{f}_\beta$ that  
this seminorm is invariant under addition of a constant to $f$, so that 
\begin{equation}
\normDgamma{f}_\beta 
\leqs \inf_{c\in\R} \; \norm{f+c}_{1+\beta V}\;.
\end{equation} 
To prove the reverse inequality, it suffices by homogeneity of the norm to show 
that if holds for $f$ with $\normDgamma{f}_\beta= 1$. We set 
\begin{equation}
 c^* = \inf_{x\in\R^n} \biggpar{1 + \beta V(x) - f(x)}\;.
\end{equation} 
For any $x,y\in\R^n$, we have 
\begin{equation}
 f(x)
 \leqs \abs{f(y)} + \abs{f(x) - f(y)} 
 \leqs \abs{f(y)} + 2 + \beta V(x) + \beta V(y)\;,
\end{equation} 
which implies 
\begin{equation}
 1 + \beta V(x) - f(x) \geqs -1 - \beta V(y) - \abs{f(y)}\;.
\end{equation} 
Since $V(y)$ is finite at one point at least, $c^*$ is bounded 
below, and hence $\abs{c^*} < \infty$. Now we observe that on one hand, 
\begin{equation}
 f(x) + c^* \leqs f(x) + 1 + \beta V(x) - f(x) = 1 + \beta V(x)\;,
\end{equation} 
while on the other hand, 
\begin{align}
 f(x) + c^*
 &= \inf_{y\in\R^n} \biggpar{f(x) + 1 + \beta V(y) - f(y)} \\
 &\geqs \inf_{y\in\R^n} \biggpar{1 + \beta V(y) - \normDgamma{f}_\beta 
d_\beta(x,y)} \\
 &\geqs -1-\beta V(x)\;,
\end{align}
where we have used the fact that $\normDgamma{f}_\beta = 1$. 
Hence $\abs{f(x)+c^*} \leqs 1 + \beta V(x)$, and thus 
\begin{equation}
\inf_{c\in\R} \; \norm{f+c}_{1+\beta V}
 \leqs \norm{f+c^*}_{1+\beta V} 
 \leqs 1
 = \normDgamma{f}_\beta\;,
\end{equation} 
proving the reverse inequality. The equality of $\rho^*_\beta$ and $\rho_\beta$ 
follows from the fact that the unit balls $\setsuch{f}{\norm{f}_{1+\beta V}} 
\leqs 1$ and $\setsuch{f}{\normDgamma{f}_\beta\leqs 1}$ only differ by additive 
constants, and homogeneity of the norms.
\end{proof}

\begin{proof}[\textsc{Proof of Proposition~\ref{prop:contract_rhobeta}}]
We prove that under Assumptions~\ref{ass:geom_drift} 
and~\ref{ass:minorisation}, one has 
\begin{equation}
\label{eq:proof_MT_bound0} 
 \normDgamma{\cP f}_{\beta} \leqs \bar\alpha \normDgamma{f}_\beta\;.
\end{equation} 
Fix a test function $f$ with $\normDgamma{f}_\beta \leqs 1$. By 
Lemma~\ref{lem:equiv_norms_HM} we can assume, without loss of generality, that 
one also has $\norm{f}_{1+\beta V}\leqs 1$. By homogeneity, it then suffices to 
show that 
\begin{equation}
 \bigabs{\bigpar{\cP f}(x) - \bigpar{\cP f}(y)} 
 \leqs \bar\alpha d_\beta(x,y)\;.
\end{equation} 
Since the claim is true for $x=y$, we consider the case $x\neq y$. We treat 
separately the cases $V(x) + V(y) \geqs R$ and $V(x) + V(y) < R$.

\begin{itemize}
\item	If $V(x) + V(y) \geqs R$, we note that 
\begin{equation}
\label{eq:proof_MT_bound1} 
 \bigabs{\bigpar{\cP f}(x)}
 \leqs \norm{f}_{1+\beta V} \int_{\R^n} (1+\beta V(y)) \cP(x,\6y) 
 \leqs 1 + \beta \bigpar{\cP V}(x)\;.
\end{equation} 
Therefore, the geometric drift condition~\eqref{eq:geom_drift} yields 
\begin{align}
 \bigabs{\bigpar{\cP f}(x) - \bigpar{\cP f}(y)}
 &\leqs 2 + \beta \bigpar{\cP V}(x) + \beta \bigpar{\cP V}(y) \\
 &\leqs 2 + \beta\gamma V(x) + \beta\gamma V(y) + 2\beta d \\
 &\leqs 2 + \beta\gamma_0 V(x) + \beta\gamma_0 V(y)\;, 
\end{align}
where we have set $\gamma_0 = \gamma + 2dR^{-1}$ and used $V(x) + 
V(y) \geqs R$. We now set 
\begin{equation}
 \gamma_1 = \frac{2+\beta R\gamma_0}{2+\beta R}\;.
\end{equation} 
One readily checks that $2(1-\gamma_1) = \beta R(\gamma_1-\gamma_0)
\leqs \beta(\gamma_1-\gamma_0)(V(x)+V(y))$, so that 
\begin{equation}
\label{eq:proof_MT_bound2} 
 \bigabs{\bigpar{\cP f}(x) - \bigpar{\cP f}(y)}
 \leqs \gamma_1 \Bigpar{2 + \beta V(x) + \beta V(y)} 
 = \gamma_1 d_\beta(x,y)\;.
\end{equation} 

\item 	If $V(x) + V(y) < R$, then $x, y\in C$. We introduce the Markov kernel 
$\widetilde\cP$ defined by 
\begin{equation}
 \widetilde\cP(x,A) = \frac{1}{1-\alpha} \cP(x,A) 
 - \frac{\alpha}{1-\alpha} \nu(A)\;.
\end{equation} 
Note that $\widetilde\cP(x,\R^n) = 1$, while the minorisation 
condition~\eqref{eq:minorisation} implies that $\widetilde\cP(x,A)$ is always 
positive, as required. Then we have 
\begin{equation}
 \bigpar{\cP f}(x) = (1-\alpha) \bigpar{\widetilde\cP f}(x) 
 + \alpha \int_{\R^n} f(y) \nu(\6y)\;,
\end{equation} 
showing that 
\begin{align}
 \bigabs{\bigpar{\cP f}(x) - \bigpar{\cP f}(y)}
 &= (1-\alpha) 
 \bigabs{\bigpar{\widetilde\cP f}(x) - \bigpar{\widetilde\cP f}(y)} \\
 &\leqs (1-\alpha) \bigbrak{2 + \beta \bigpar{\widetilde\cP V}(x) + \beta 
\bigpar{\widetilde\cP V}(y)} \\
 &\leqs 2(1-\alpha) + \beta \bigpar{\cP V}(x) + \beta \bigpar{\cP V}(y) \\
 &\leqs 2(1-\alpha) + \gamma\beta \bigbrak{V(x)+V(y)} + 2\beta d\;.
\end{align}
Here, to obtain the second line, we have used a similar argument as 
in~\eqref{eq:proof_MT_bound1}, while the third line uses the fact that 
\begin{equation}
 \bigpar{\widetilde\cP V}(x) 
 \leqs \frac{1}{1-\alpha} \bigpar{\cP V}(x)
\end{equation} 
since $V$ is non-negative. It follows that setting 
\begin{equation}
 \beta  = \frac{\alpha_0}{d}\;, \qquad 
 \gamma_2 = \gamma\vee\Bigpar{1-(\alpha-\alpha_0)}
\end{equation} 
for some $\alpha_0\in(0,\alpha)$, one obtains 
\begin{equation}
\label{eq:proof_MT_bound3} 
 \bigabs{\bigpar{\cP f}(x) - \bigpar{\cP f}(y)} 
 \leqs \gamma_2 d_\beta(x,y)\;.
\end{equation}
\end{itemize}
It follows from~\eqref{eq:proof_MT_bound2} and~\eqref{eq:proof_MT_bound3} that
\begin{equation}
 \bigabs{\bigpar{\cP f}(x) - \bigpar{\cP f}(y)} 
 \leqs \bar\alpha d_\beta(x,y)\;, 
 \qquad 
 \bar\alpha = \gamma_1 \vee \gamma_2\;.
\end{equation} 
Since $\gamma_1\geqs\gamma$, this implies~\eqref{eq:proof_MT_bound0}.
The result then follows from the fact that $d_\beta$ is the norm dual to 
$\normDgamma{\cdot}_\beta$. Indeed, we have 
\begin{align}
\rho^*_\beta(\mu\cP, \nu\cP) 
&= \sup_{f: \normDgamma{f}_\beta \neq 0} 
\frac{1}{\normDgamma{f}_\beta} \pscal{\mu\cP - \nu\cP}{f} 
= \sup_{f: \normDgamma{f}_\beta \neq 0} 
\frac{1}{\normDgamma{f}_\beta} \pscal{\mu - \nu}{\cP f} \\
&\leqs \sup_{\tilde f: \normDgamma{\tilde f}_\beta \neq 0} 
\frac{\bar\alpha}{\normDgamma{\tilde f}_\beta} \pscal{\mu - \nu}{\tilde f} 
= \bar\alpha \rho^*_\beta(\mu,\nu)\;,
\end{align}
and we obtain the conclusion since $\rho_\beta = \rho^*_\beta$. 
\end{proof}

To conclude the proof of Theorem~\ref{thm:Hairer_Mattingly}, it remains to 
prove existence of the invariant measure, which can be done by a contraction 
argument.

\begin{proof}[\textsc{Proof of Theorem~\ref{thm:Hairer_Mattingly}}] We fix some 
$x\in\R^n$, and define for any $n\in\N$ the measure $\mu_n^x = \delta_x\cP^n$. 
Then by Proposition~\ref{prop:contract_rhobeta}, we have 
\begin{equation}
\label{eq:rhobeta_contraction} 
 \rho_\beta(\mu_{n+1}^x,\mu_n^x) \leqs \bar\alpha^n 
\rho_\beta(\mu_1^x,\delta_x)
\end{equation} 
for some $\bar\alpha\in(0,1)$ and $\beta > 0$. Therefore, $(\mu_n^x)_n$ is a 
Cauchy sequence. It is known that the total variation distance is complete for 
the space of measures with finite mass, implying that $\rho_\beta$ is complete 
for the space of probability measures integrating $V$. Therefore, there exists 
a probability measure $\pi = \mu_\infty$ such that 
$\rho_\beta(\mu_n^x,\mu_\infty)\to0$ as $n\to\infty$. This implies that 
$\mu_n^x$ converges to $\mu_\infty$ in total variation. Since $\cP$ is a 
contraction in total variation, it follows that $\mu_\infty\cP = 
\lim_{n\to\infty} \mu_n^x\cP = \lim_{n\to\infty} \mu_{n+1}^x = \mu_\infty$. 
In order to show~\eqref{eq:expo_ergodicity_HM}, we observe that 
\begin{equation}
 \bigpar{\cP^n f}(x) - \pscal{\pi}{f} 
 = \pscal{\delta_x - \pi}{\cP^n f} 
 = \pscal{\delta_x\cP^n - \pi}{f} 
 = \pscal{\mu_n^x - \pi}{f}\;.
\end{equation} 
Therefore, by the definition~\eqref{eq:rho_beta}  of $\rho_\beta$, we have for 
any $\beta\geqs0$, 
\begin{equation}
\label{eq:proof_HM_contraction} 
 \norm{\cP^n f - \pscal{\pi}{f}}_{1+\beta V} 
 = \sup_{x\in\R^n} \frac{\abs{\pscal{\mu_n^x - \pi}{f}}}{1 + \beta V(x)}
 \leqs \norm{f}_{1+\beta V} 
 \sup_{x\in\R^n} \frac{\rho_\beta(\mu_n^x,\pi)}{1 + \beta V(x)}\;.
\end{equation}
Now it follows from~\eqref{eq:rhobeta_contraction} and a telescopic sum 
argument that for $\beta$ as in Proposition~\ref{prop:contract_rhobeta}, one 
has $\rho_\beta(\mu_n^x,\pi) \leqs M(x)\bar\alpha^n$, where $M(x)$ is 
proportional to 
\begin{equation}
 \rho_\beta(\mu_1^x,\delta_x) 
 = \int_{\R^n} \bigpar{1+\beta V(y)} \abs{\mu_1^x - \delta_x}(\6y) 
 \leqs \pscal{\mu_1^x}{1+\beta V} + 1 + \beta V(x)\;.
\end{equation} 
Since $\pscal{\mu_1^x}{V} = (\cP V)(x) \leqs \gamma V(x) + d$ by the geometric 
drift condition~\eqref{eq:geom_drift}, the supremum 
in~\eqref{eq:proof_HM_contraction} has order $\bar\alpha^n$, proving the result 
for this particular $\beta$. However, all $\norm{\cdot}_{1+\beta V}$-norms are 
equivalent, so that the result holds in particular for $\beta = 1$.  
\end{proof}


\section{Garrett Birkhoff's approach}
\label{sec:Birkhoff} 

The aim of this section is to present a slightly different approach to 
estimating the rate of convergence to an invariant probability distribution, 
due to Garret Birkhoff~\cite{Birkhoff1957}. Compared to the approaches we have 
discussed so far, it has the following advantages:

\begin{enumerate}
\item 	The proof has a more transparent geometric interpretation, that helps 
understand the minorisation condition~\eqref{eq:minorisation} we have seen in 
the last section.

\item 	As Theorem~\ref{thm:Hairer_Mattingly}, the result provides explicit 
bounds on the rate of convergence to the invariant probability.

\item	The result also works for submarkovian processes, that is, processes in 
which the total probability decreases. In that case, it provides information on 
the principal eigenvalue of the process, as well as on the spectral gap to the 
next-to-leading eigenvalue.
\end{enumerate}

As in the last subsection, the approach applies to discrete-time Markov chains. 
For simplicity, we are going to assume that the transition kernel $\cP$ has a 
density $p(x,y)$ defined on $\cX\times\cX$ for a domain $\cX\subset\R^n$ (or 
possibly on a more general Banach space). The transition kernel thus acts on 
bounded measurable functions $f$ and on signed measures $\mu$ according to 
\begin{align}
\bigpar{\cP f}(x) &= \int_{\cX} p(x,y) f(y) \6y = \expecin{x}{f(X_1)}\;, \\
\bigpar{\mu \cP}(\6y) &= \int_{\cX} \mu(\6x) p(x,y) \6y 
= \probin{\mu}{X_1\in \6y}\;.
\label{eq:integral_semigroups} 
\end{align}

The main property that will guarantee convergence to an invariant probability 
distribution (or to a so-called quasistationary distribution in the 
submarkovian case) is the following. Note the similarity of the lower bound 
with the minorisation condition~\eqref{eq:minorisation}. 

\begin{definition}[Uniform positivity]
\label{def_uniformly_positive} 
The transformation $\cP$ is called \emph{uniformly positive} if there exist
strictly positive functions $s, m: \cX\to(0,+\infty)$ and a constant $L$ such 
that 
\begin{equation}
\label{eq:uniform_positivity} 
 s(x) m(y) \leqs p(x,y) \leqs L s(x) m(y)  
 \qquad
 \forall x, y\in \cX\;.
\end{equation} 
\end{definition}

The main results we are going to prove are as follows. The first one is an 
existence result for an invariant probability measure, or for its equivalent if 
the chain is submarkovian. It is thus a generalisation to integral operators of 
the well-known Perron--Frobenius theorem, which was first obtained by 
Jentzsch~\cite{Jentzsch1912}. 

\begin{theorem}[Perron--Frobenius--Jentzsch theorem]
\label{thm:Birkhoff} 
If $\cP$ is uniformly positive, there exist $\lambda_0>0$, a bounded 
measurable function $h_0: \cX\to\R_+$, and a probability measure $\pi_0$ on 
$\cX$ such that 
\begin{align}
 \bigpar{\cP h_0}(x) = \lambda_0 h_0(x)\;, \\
 \bigpar{\pi_0 \cP}(A) = \lambda_0 \pi_0(A)
 \label{eq:def_P_Birkhoff} 
\end{align} 
for all $x\in\cX$ and all Borel sets $A\subset\cX$. In particular, in the 
Markovian case $\cP(x,\cX) = 1$ for all $x\in\cX$, one has $\lambda_0 = 1$ and 
$h_0(x) = 1$ for all $x\in\cX$.
\end{theorem}

The number $\lambda_0$ is called the~\emph{principal eigenvalue} of the Markov 
process, the function $h_0$ is called the~\emph{principal eigenfunction}, while 
$\pi_0$ is called the~\emph{quasistationary distribution} (in the submarkovian 
case, when $\lambda_0 < 1$), and is equal to the stationary distribution in the 
Markovian case, when $\lambda_0 = 1$.
The main interest of the uniform positivity condition is the following result 
on speed of convergence towards the principal eigenfunction $h_0$. 

\begin{theorem}[Spectral gap estimate]
\label{thm:Birkhoff2} 
If $\cP$ is uniformly positive, then for any bounded measurable $f: 
\cX\to\R_+$, there exist finite constants $M_1(f), M_2(f)$ such that 
\begin{equation}
\label{eq:exp_ergodicity_Birkhoff} 
 \bigabs{\cP^n f(x) - \lambda_0^n M_1(f) h_0(x)} \leqs M_2(f)
\lambda_0^n\Biggpar{1-\frac1{L^2}}^n h_0(x)
\end{equation} 
for all $x\in\cX$. In particular,  
\begin{equation}
 M_1(f) = \frac{\pscal{\pi_0}{f}}{\pscal{\pi_0}{h_0}}\;,
\end{equation}
which reduces in the Markovian case to $M_1(f) = \pscal{\pi_0}{f}$. 
\end{theorem}

Note that in the Markovian case $\lambda_0 = 1$, the 
bound~\eqref{eq:exp_ergodicity_Birkhoff} is equivalent to the exponential 
ergodicity result~\eqref{eq:expo_ergodicity_HM} of 
Theorem~\ref{thm:Hairer_Mattingly}, but with an explicit value of the 
contraction constant $\gamma$ given by $1-L^{-2}$.  

In what follows, we are going to provide a detailed proof of 
Theorem~\ref{thm:Birkhoff}, starting with some simpler situations in order to 
build the intuition. 

\subsection{Two-dimensional case}

The simplest case occurs when $\cX$ is a discrete set of cardinality $2$. Then 
$\cP$ is a linear operator on $\cE=\R^2$, that is, a $2\times2$ matrix 
\begin{equation}
\cP=
 \begin{pmatrix}
 a & b \\ c & d
 \end{pmatrix}
\end{equation} 
with strictly positive entries. Therefore $\cP$ maps the cone 
$\cE^+=\R_+\times\R_+$ strictly into itself. Iterating $\cP$, the image of 
$\cE^+$ becomes thinner and thinner, and concentrates on the eigenvector of 
$\cP$ for the largest eigenvalue (\figref{fig:positive_cone}). However, unless 
the principal eigenvalue $\lambda_0$ of $\cP$ is $1$, iterates of a vector in 
$\cE^+$ will not converge to an eigenvector: they will shrink to $0$ if 
$\lambda_0<1$ and diverge if $\lambda_0>1$. 

To avoid this, one can identify all vectors $f, g$ such that $f=\lambda g$ for
some $\lambda>0$. In other words, this amounts to working on the projective 
line. Iterates of a projective line in $\cE^+$ will converge to the eigenspace 
associated with $\lambda_0$.

Birkhoff introduces Hilbert's \emph{projective metric} by defining, for
$f=(f_1,f_2)$ and $g=(g_1,g_2)\in \cE^+$, the distance  
\begin{equation}
 \theta(f,g) = \biggabs{\log\biggpar{\frac{f_2g_1}{f_1g_2}}}\;.
\end{equation} 
Note that this distance is infinite if $f$ or $g$ belongs to a coordinate axis; 
in fact, it induces a hyperbolic geometry. Also note that by definition, 
\begin{equation}
\label{3}
 \theta(\lambda f,\mu g) = \theta(f,g)
 \qquad \forall \lambda,\mu>0\;.
\end{equation} 

\begin{lemma}[Projective operator norm of $\cP$]
The operator norm of $\cP$ in the projective metric is given by  
\begin{equation}
\label{4}
\sup_{f,g\in \cE^+} 
 \frac{\theta(\cP f,\cP g)}{\theta(f,g)} = \tanh\Biggpar{\frac{\Delta}{4}}\;,
\end{equation} 
where $\Delta=\abs{\log(ad/bc)}$ is the diameter, in the projective norm, of
$\cP(\cE^+)$. 
\end{lemma}
\begin{proof}
We may assume, without limitation of generality, that $ad > bc$ (note that in 
the case $ad = bc$, the cone $\cE^+$ is projected to a half line by $\cP$, so 
that the operator norm of $\cP$ is zero). Let $f = (1,x)$ be a point in 
$\cE^+$. Then 
\begin{equation}
 \cP f = (a + bx, c + dx) = \lambda \bigpar{1, \ph(x)}\;,
\end{equation} 
where $\lambda = a+bx$, and $\ph$ denotes the homographic transformation 
\begin{equation}
 x \mapsto \ph(x) = \frac{c+dx}{a+bx}\;.
\end{equation} 
The projective distance between two infinitesimally close points $f=(1,x)$ and 
$(1,x+\6x)$ is given by $\6\theta(x) = \abs{\log(x/(x+\6x))} = \abs{\6x}/x$. 
Therefore, we obtain 
\begin{equation}
 \frac{\6\theta(\ph(x))}{\6\theta(x)}
 = \frac{x\ph'(x)}{\ph(x)} 
 = \frac{x(ad-bc)}{(a+bx)(c+dx)}\;.
\end{equation} 
It is straightforward to check that this expression is maximal in $x = 
\sqrt{ac/(bd)}$, where it has value 
\begin{equation}
 \frac{ad-bc}{ad + bc + 2\sqrt{abcd}}
 = \frac{\e^\Delta - 1}{1 + \e^\Delta + 2\e^{\Delta/2}}
 = \tanh\Biggpar{\frac{\Delta}{4}}\;. 
\end{equation} 
Since this value gives the smallest rate of contraction, the claim follows. 
\end{proof}

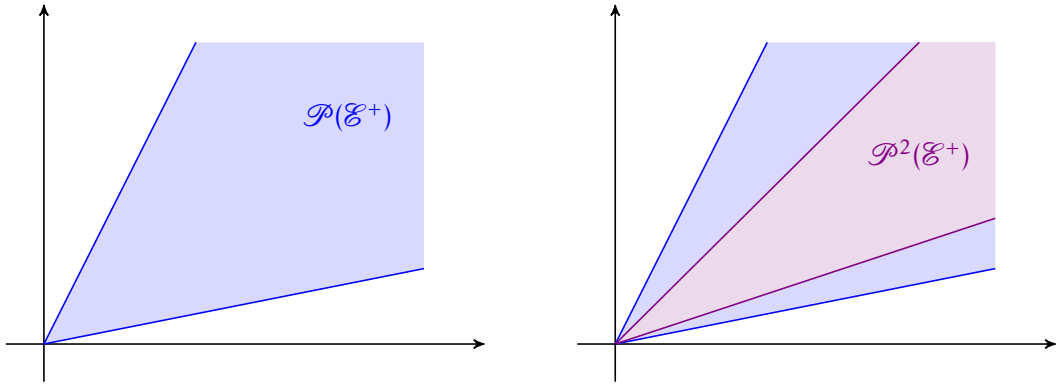
\begin{figure}
\begin{center}
\begin{tikzpicture}[-,auto,node distance=1.0cm, thick,
main node/.style={draw,circle,fill=white,minimum size=4pt,inner sep=0pt},
full node/.style={draw,circle,fill=black,minimum size=4pt,inner sep=0pt}]

\path[fill = blue!15] (0,0) -- (5,1) -- (5,4) -- (2,4) -- (0,0);

\draw[semithick,->,>=stealth'] (-0.5, 0.0) -- (5.8, 0.0);
\draw[semithick,->,>=stealth'] (0.0, -0.5) -- (0.0, 4.5);

\draw[semithick, blue] (0,0) -- (5,1);
\draw[semithick, blue] (0,0) -- (2,4);

\node[blue] at (4,3) {$\cP(\cE^+)$};
     
\end{tikzpicture}
\hspace{10mm}
\begin{tikzpicture}[-,auto,node distance=1.0cm, thick,
main node/.style={draw,circle,fill=white,minimum size=4pt,inner sep=0pt},
full node/.style={draw,circle,fill=black,minimum size=4pt,inner sep=0pt}]

\path[fill = blue!15] (0,0) -- (5,1) -- (5,4) -- (2,4) -- (0,0);
\path[fill = violet!15] (0,0) -- (5,{5/3}) -- (5,4) -- (4,4) -- (0,0);

\draw[semithick,->,>=stealth'] (-0.5, 0.0) -- (5.8, 0.0);
\draw[semithick,->,>=stealth'] (0.0, -0.5) -- (0.0, 4.5);

\draw[semithick, blue] (0,0) -- (5,1);
\draw[semithick, blue] (0,0) -- (2,4);

\draw[semithick, violet] (0,0) -- (5,{5/3});
\draw[semithick, violet] (0,0) -- (4,4);

\node[violet] at (4,2.5) {$\cP^2(\cE^+)$};
     
\end{tikzpicture}
\end{center}
\vspace{-4mm}
 \caption[]{Geometrical explanation for the convergence of iterates of a 
two-dimensional positive operator $\cP$. The positive cone 
$\cE^+=\R_+\times\R_+$ is mapped to a smaller cone $\cP(\cE^+)$, strictly 
contained in $\cE_+$. Each iterate $\cP^n(\cE^+)$ has a strictly smaller 
diameter, and the sequence of these iterates converges to a half line.}
 \label{fig:positive_cone}
\end{figure}

\subsection{General vector space}

Let now $\cE$ be a general vector space, of finite or infinite dimension. 
Again, let $\cE^+$ be the cone of elements whose components are all
non-negative. 

\begin{definition}[Projective metric]
\label{def_theta} 
Let $f,g\in\cE^+$. Consider the two-dimensional vector space $E$ spanned by $f$
and $g$. The intersection $C=E\cap\cE^+$ is a cone (it is invariant under
multiplication by positive constants). There exists a linear map $A$, mapping
$C$ to $\R_+\times\R_+$. We define 
\begin{equation}
 \theta(f,g;\cE^+) = \theta(Af,Ag)\;.
\end{equation} 
The definition does not depend on the choice of the map $A$ (this follows 
from \eqref{3}). $\theta$ is called the \emph{projective metric}
associated with $\cE^+$. 
\end{definition}

\begin{figure}
\begin{center}
\begin{tikzpicture}[-,auto,node distance=1.0cm, thick,
main node/.style={draw,circle,fill=white,minimum size=4pt,inner sep=0pt},
full node/.style={draw,circle,fill=black,minimum size=4pt,inner sep=0pt}]

\path[fill = blue!15] (0,0) -- (8,0) -- (8,6) -- (0,6) -- (0,0);

\draw[semithick,->,>=stealth'] (-0.5, 0.0) -- (8.8, 0.0);
\draw[semithick,->,>=stealth'] (0.0, -0.5) -- (0.0, 6.5);

\draw[thick, red] (0,5) -- (5,0);
\draw[semithick, violet] (0,0) -- (8,2);
\draw[semithick, violet] (0,0) -- (4,6);

\node[main node] at (0,5) {};
\node[main node] at (5,0) {};
\node[main node] at (4,1) {};
\node[main node] at (2,3) {};

\node[violet] at (4.2,1.3) {$f$};
\node[violet] at (2.3,3.0) {$g$};
\node[violet] at (5,-0.35) {$f_{\alpha^*}$};
\node[violet] at (-0.35,5) {$g_{\beta^*}$};
\node[blue] at (6,5) {$A(E\cap\cE^+)$};
     
\end{tikzpicture}
\end{center}
\vspace{-4mm}
 \caption[]{Construction of the map $A$ via the points $f_{\alpha^*}$ and 
$g_{\beta^*}$. The points $f, g, \dots$ should be viewed as equivalence classes 
represented by violet half-lines.}
 \label{fig:projective_norm}
\end{figure}
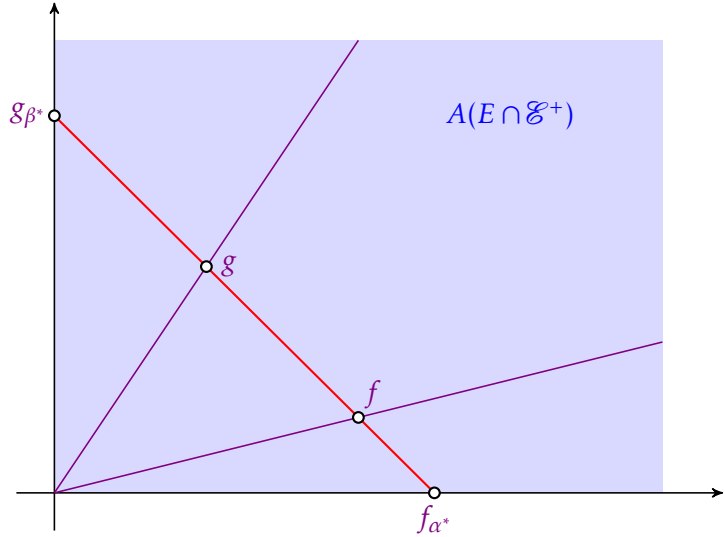

To understand this definition better, consider the
line $\set{f_\alpha=f-\alpha g, \alpha\in\R}$. If $\alpha\leqs0$, then
$f_\alpha$, being the sum of two positive elements, is in $\cE^+$. When
$\alpha>0$, however, the components of $f_\alpha$ decrease with increasing
$\alpha$, and change sign at some point. Let 
\begin{equation}
\label{7} 
 \alpha^* = \sup\bigsetsuch{\alpha>0}{f-\alpha g \in\cE^+}
\end{equation} 
(see \figref{fig:projective_norm}).
Similarly, we define 
\begin{equation}
\label{8} 
\beta^* = \sup\bigsetsuch{\beta>0}{g-\beta f \in\cE^+}\;.
\end{equation} 
Then the linear map of matrix (in the basis $(f,g)$) 
\begin{equation}
\label{9} 
A = 
 \begin{pmatrix}
 1 & \beta^* \\ \alpha^* & 1 
 \end{pmatrix}
\end{equation}
maps $f_{\alpha^*}$ to a multiple of $(1,0)$ and $f_{\beta^*}$ to a multiple of 
$(0,1)$. It thus satisfies the definition. 
Furthermore we have $Af=(1,\alpha^*)$ and $Ag=(\beta^*,1)$, so that 
\begin{equation}
\label{10} 
 \theta(f,g;\cE^+) = \theta(Af,Ag) = \bigabs{\log(\alpha^*\beta^*)}\;.
\end{equation} 

\begin{proposition}[Operator norm of $\cP$]
Let $\cP:\cE^+\to\cE^+$ be a linear map. If $\cP(\cE^+)$ has finite diameter 
$\Delta$,
then the operator norm of $\cP$ is given by 
\begin{equation}
\sup_{f,g\in \cE^+} 
 \frac{\theta(\cP f,\cP g;\cE^+)}{\theta(f,g;\cE^+)} = 
\tanh\Biggpar{\frac{\Delta}{4}}\;.
\end{equation} 
\end{proposition}
\begin{proof}
If $\theta(f,g;\cE^+)<\infty$, let $a, b$ be the endpoints of the intersection
of $\cE^+$ with the line through $f$ and $g$. By definition of $\Delta$, 
$\theta(\cP a,\cP b;\cE^+)\leqs\Delta$. Thus the operator norm is bounded by 
$\tanh(\Delta/4)$ as a consequence of~\eqref{4}. To show equality, one uses an
approximation argument for a sequence of $(f_n,g_n)$ of growing projective
distance. 
\end{proof}

\begin{theorem}[Convergence of iterates in projective space]
\label{thm_convergence} 
If $\cP(\cE^+)$ has finite diameter $\Delta$ and the cone $\cE^+$ is complete 
with respect to the distance $\theta$, then there is a unique ray $h$ in 
$\cE^+$ to which $\cP^n f$ converges for all $f\in\cE^+$.  
\end{theorem}

The proof is a standard contraction argument. A proof of completeness will be
given in Corollary~\ref{cor_completeness} below. 

\subsection{Integral transformations and Jentzsch's theorem}

We now return to the situation described at the beginning of this section, where 
$\cX$ is a Borel set of $\R^n$, and $\cE$ is the Banach space of continuous 
functions $f:\cX\to\R$, equipped with the supremum norm. Let $\cE^+$ denote the 
cone of positive functions $f:\cX\to\R_+$. Consider the integral operator 
$\cP:\cE^+\to\cE^+$ defined by~\eqref{eq:integral_semigroups}.

\begin{proposition}[Bound on the projective diameter]
If $\cP$ satisfies the uniform positivity 
condition~\eqref{eq:uniform_positivity}, then the diameter of $\cP(\cE^+)$ 
satisfies 
\begin{equation}
\label{eq:bound_Delta} 
 \Delta \leqs 2\log L\;.
\end{equation} 
\end{proposition}
\begin{proof}
Let $f,g\in\cE^+$. Without limiting the generality, we may assume
\begin{equation}
\int_{\cX} f(y)m(y)\6y = \int_{\cX} g(y)m(y)\6y = 1 
\end{equation} 
This implies 
\begin{equation}
 s(x) \leqs (\cP f)(x)\;,\; (\cP g)(x) \leqs L s(x)
 \qquad \forall x\in\cX\;.
\end{equation} 
It follows that $(\cP f) - \frac{1}{L}(\cP g) \geqs 0$ and $(\cP g) - 
\frac{1}{L}(\cP f)
\geqs 0$. Thus $\alpha^*, \beta^*$ defined in~\eqref{7} and~\eqref{8} are
greater or equal than $1/L$, that is $1/\alpha^*, 1/\beta^* \leqs L$ 
and the result follows from~\eqref{10}.  
\end{proof}

Applying Theorem~\ref{thm_convergence}, we recover Jentzsch's 
generalisation of the Perron--Frobenius theorem to integral 
operators~\cite{Jentzsch1912}:

\begin{theorem}[Perron--Frobenius theorem for integral operators]
If $\cP$ is uniformly positive, then there exists a strictly positive
$h_0\in\cE^+$ and $\lambda_0>0$ such that $\cP h_0=\lambda_0h_0$. Moreover, for 
any
$f\in\cE^+$, the sequence of lines spanned by $\cP^nf$ converges to the line
spanned by $h_0$. 
\end{theorem}

\begin{remark}[Dual picture]
\label{rem_dual} 
The dual map $\cP^*$ given by 
\begin{equation}
 \bigpar{\cP^*v}(y) := \bigpar{v\cP}(y) = \int_{\cX} v(x) p(x,y) \6x
\end{equation} 
satisfies $\pscal{\cP^*v}{f} = \pscal{x}{\cP f}$, where $\pscal{\cdot}{\cdot}$ 
denotes the usual inner product. Jentzsch's theorem shows the existence of a 
strictly positive function $p_0$ such that $\cP^*p_0=\lambda_0p_0$, with a 
similar convergence property. The eigenvalue $\lambda_0$ is the same, since 
\begin{equation}
 \lambda_0\pscal{p_0}{h_0} = \pscal{p_0}{\cP h_0} = \pscal{\cP^*p_0}{h_0} 
\end{equation} 
and $\pscal{p_0}{h_0}>0$. Furthermore, we have for any $f\in\cE^+$ that 
\begin{equation}
 \lambda_0^n \pscal{p_0}{f} = \pscal{(\cP^*)^np_0}{f} 
 = \pscal{p_0}{\cP^n f}\;,
\end{equation} 
which implies that 
\begin{equation}
 \lim_{n\to\infty} \frac{\cP^nf}{\lambda_0^n} = c(f) h_0
 \qquad
 \text{where } c(f) = \frac{\pscal{p_0}{f}}{\pscal{p_0}{h_0}}\;.
\end{equation} 
\end{remark}

\subsection{Banach lattices and spectral gap}

Birkhoff extends the theory to Banach lattices, that is, Banach spaces $\cE$
with a (partial) order in which every pair of elements $f,g$ admits an infimum
$f\wedge g$ and a supremum $f\vee g$. 
Examples of vector lattices include 
\begin{enumerate}
 \item the space of continuous functions $f:\cX\to\R$, equipped with the
supremum norm, with pointwise order given by 
\begin{equation}
 f\leqs g
 \qquad \Leftrightarrow\qquad
 f(x)\leqs g(x) \quad \forall x\in\cX\;,
\end{equation} 
and 
\begin{equation}
 (f\wedge g)(x) = f(x)\wedge g(x)
 \qquad \text{and} \qquad 
 (f\vee g)(x) = f(x)\vee g(x)\;;
\end{equation} 

\item the space of bounded measurable functions $f:\cX\to\R$, with the same
norm;

\item the space of finite signed measures $\mu$ on $\cX$, equipped with the
$L^1$-norm.
\end{enumerate}
The last two examples are associated with Markov kernels $\cP(x,A)$ and the
(dual) maps introduced in~\eqref{eq:integral_semigroups}.
As in Definition~\ref{def_uniformly_positive}, the Markov kernel is called 
uniformly positive if there exist a positive function $f$, a measure $\nu$ 
(absolutely continuous with respect to Lebesgue measure, with strictly positive 
density\footnote{For results on more general measures, 
see~\cite{Nummelin84,Orey1971}.}) and a constant $L$ such that 
\begin{equation}
 s(x) \nu(A) \leqs \cP(x,A) \leqs L s(x) \nu(A)
 \qquad
 \forall x\in\cX, \forall A\subset \cX\;.
\end{equation} 
A similar computation as above shows that $\cP(\cE^+)$ has projective diameter 
$\Delta\leqs 2\log L$. Then similar arguments as before show that $\cP$ admits a 
unique principal eigenvalue $\lambda_0$, a measure $\pi_0$ such that 
$\pi_0\cP=\lambda_0\pi_0$, called the \emph{quasistationary distribution}, and a 
positive function $h_0$ such that $\cP h_0=\lambda_0h_0$. 

We now examine the speed of convergence of iterates of a positive 
map $\cP$ for a general Banach lattice. The following proposition is a key 
result.

\begin{proposition}[Strong comparability]
\label{prop_comparable} 
Any $f,g\in\cE^+$ are strongly comparable, in the sense that there exist
strictly positive constants $\alpha, \beta, R$ such that 
\begin{align}
\alpha f &\leqs g \leqs R\alpha f\;, \\
\beta g  &\leqs f \leqs R\beta  g\;.
\label{4.4} 
\end{align}
The optimal constant is $R=\e^{\theta(f,g;\cE^+)}$. 
\end{proposition}
\begin{proof}
Let $A$ be the linear map of Definition~\ref{def_theta}, and write 
$Af=(f_1,f_2)$, $Ag=(g_1,g_2)$. Assume without limiting the generality that 
$f_1g_2\geqs f_2g_1$. Then 
\begin{equation}
 f_1(Ag) - g_1(Af) = (0, f_1g_2 - g_1f_2) \in\R_+\times\R_+
\end{equation} 
and thus $f_1g -g_1 f \in\cE^+$. Similarly, we have $g_2f-f_2g\in\cE^+$. 
This shows that 
\begin{equation}
 \frac{g_1}{f_1}f \leqs g \leqs \frac{g_2}{f_2}f 
 = \e^{\theta(f,g;\cE^+)} \frac{g_1}{f_1} f\;,
\end{equation} 
and thus~\eqref{4.4} holds with $\alpha=g_1/f_1$. The proof of the second
inequality is analogous. 
\end{proof}

A first consequence of this result is that we can prove completeness. 

\begin{corollary}[Completeness of the metric]
\label{cor_completeness} 
If $\norm{f}=\norm{g}=1$, then 
\begin{equation}
 \norm{f-g} \leqs \e^{\theta(f,g;\cE^+)} - 1\;.
\end{equation} 
As a consequence, in the metric defined by $\theta$, any $\theta$-connected
component of the unit sphere is a complete metric space. 
\end{corollary}
\begin{proof}
If $\norm{f}=\norm{g}=1$, then~\eqref{4.4} holds with $\alpha\leqs 1\leqs
R\alpha$ and $R=\e^{\theta(f,g;\cE^+)}$. Thus 
\begin{equation}
 \norm{f-g} = \norm{f\vee g - f\wedge g} \leqs 
 \norm{R\alpha f - \alpha f} = (R-1)\alpha\norm{f}
 \leqs R-1\;,
\end{equation}
as claimed. 
\end{proof}
It follows that Theorem~\ref{thm_convergence} indeed applies in this setting. 
Let us finally derive a spectral-gap estimate. 

\begin{proof}[\textsc{Proof of Theorem~\ref{thm:Birkhoff2}}]
Denote $\cP^nf$ by $f_n$. 
For any $n$ let $\alpha_n$ and $\beta_n$ be the optimal constants for which 
\begin{equation}
 \alpha_n h_0 \leqs \frac{f_n}{\lambda_0^n} \leqs \beta_n h_0\;.
\end{equation} 
Such constants exist and are positive for $n=1$ because $f_1, h_0$ belong to a
cone with diameter $\Delta$. Assuming by induction that the above inequality
holds for some $n$, and applying $\cP$, we obtain that it holds for $n+1$ with 
\begin{equation}
 \alpha_n \leqs \alpha_{n+1} \leqs \beta_{n+1} \leqs \beta_n\;.
\end{equation} 
Define 
\begin{align}
\nonumber
r_n &= f_n - \alpha_n \lambda_0^n h_0 \in \cE^+\;, \\
s_n &= \beta_n\lambda_0^n h_0 - f_n \in \cE^+\;.
\label{4.13} 
\end{align}
We have 
\begin{align}
\nonumber
r_n + s_n &= (\beta_n -\alpha_n) \lambda_0^n h_0 \;, \\
\cP r_n + \cP s_n &= (\beta_n -\alpha_n) \lambda_0^{n+1} h_0 \;. 
\end{align}
By Proposition~\ref{prop_comparable}, there exist positive constants $a_n, b_n$
and $R\leqs\e^\Delta$ such that 
\begin{align}
\nonumber
a_n h_0 &\leqs \cP r_n \leqs Ra_nh_0\;, \\
b_n h_0 &\leqs \cP s_n \leqs Rb_nh_0\;.
\end{align} 
On one hand it follows that 
\begin{equation}
\label{4.16} 
 (a_n+b_n)h_0 \leqs \cP r_n + \cP s_n = (\beta_n -\alpha_n) \lambda_0^{n+1} h_0
 \leqs R(a_n+b_n)h_0\;.
\end{equation}
On the other hand, we conclude by applying $\cP$ to~\eqref{4.13} that 
\begin{equation}
 (\alpha_n\lambda_0^{n+1}+a_n)h_0 \leqs \cP f_n = f_{n+1} \leqs
(\beta_n\lambda_0^{n+1}-b_n)h_0\;, 
\end{equation} 
This yields 
\begin{equation}
 \alpha_{n+1} \geqs \alpha_n + \frac{a_n}{\lambda_0^{n+1}}\;, 
 \qquad
 \beta_{n+1} \leqs \beta_n - \frac{b_n}{\lambda_0^{n+1}}\;.
\end{equation} 
Using~\eqref{4.16} it follows that 
\begin{equation}
 (\beta_{n+1}-\alpha_{n+1}) \leqs (\beta_n-\alpha_n) -
\frac{a_n+b_n}{\lambda_0^{n+1}}
\leqs \biggpar{1-\frac{1}{R}} (\beta_n-\alpha_n)\;.
\end{equation} 
This shows that the sequences $\alpha_n$ and $\beta_n$ converge to a common
limit $M_1(f)$, and thus that $f_n/\lambda_0^n$ converges to $M_1(f)h_0$ at
rate $(1-R^{-1})^n=(1-\e^{-\Delta})^n$. 

Finally, the uniform positivity condition~\eqref{eq:uniform_positivity} implies 
that $\e^{-\Delta}$ is bounded below by $1/L^2$, which concludes the proof. 
\end{proof}

\begin{remark}[Dual picture]
As in Remark~\ref{rem_dual}, we have 
\begin{equation}
 \lambda_0^n \pscal{\pi_0}{f} = \pscal{\pi_0\cP^n}{f} = \pscal{\pi_0}{\cP^nf} 
\end{equation} 
for all $n$, which shows that 
\begin{equation}
 M_1(f) = \frac{\pscal{\pi_0}{f}}{\pscal{\pi_0}{h_0}}\;.
\end{equation} 
\end{remark}

\subsection{From discrete time to continuous time}

We provide here a simple illustration of how the discrete-time results 
presented in this section (and in Section~\ref{ssec:Hairer_Mattingly}) can be 
applied to continuous-time SDEs. Consider the SDE 
\begin{equation}
\label{eq:SDE_example_Birkhoff} 
 \6X_t = f(X_t) \6t + \sigma \6W_t\;,
\end{equation} 
where $\sigma>0$ is a small parameter, and $f$ has a stable equilibrium point 
at the origin, that is 
\begin{equation}
 f(0) = 0\;,
\end{equation} 
and all the eigenvalues of the Jacobian matrix 
\begin{equation}
 A = \dpar{f}{x}(0) 
\end{equation} 
are strictly negative. The approaches we just introduced require upper and 
lower bounds on the transition density $p_t(x,y)$ for some $t>0$, say $t=1$. A 
general approach for obtaining such bounds is based on Malliavin calculus, but 
it is also possible to obtain the required information by less elaborate 
methods. The approach we outline here is a simplification of the method used 
in~\cite{BG_periodic2,BaudelBerglund}. 

A first point is that one can use Harnack inequalities for $\cL$-harmonic 
functions (see~\cite[Corollaries~9.24 and 9.25]{Gilbarg_Trudinger}) to show 
that the transition density at time $1$, $p_1(x,y)$, satisfies the following 
two regularity estimates on small balls. For $x\in\R^n$ and $r>0$, we let 
$\cB_r(x)$ denote the ball of radius $r$ centred in $x$. 

\begin{lemma}[Harnack-type bounds on the transition density]
\begin{enumerate}
\item 	Fix $x_0,y\in\R^n$. There exists a constant $C_0$, independent of 
$x_0$ and $\sigma$, such that 
\begin{equation}
\label{eq:Harnack1} 
  \sup_{x\in \cB_{\sigma^2}(x_0)} p_1(x,y) 
  \leqs C_0 \inf_{x\in \cB_{\sigma^2}(x_0)} p_1(x,y)\;.
\end{equation} 

\item 	Fix $x_0,y\in\R$ and $r_0>0$, and let $R_0 = r_0\sigma^2$. Then there 
exist constants $C_1\geqs1$ and $\alpha>0$, independent of $\sigma$, such 
that for any $R\leqs R_0$, one has  
\begin{equation}
\label{eq:Harnack2} 
 \osc_{\cB_R(x_0)} p_1 
 := \sup_{x\in \cB_R(x_0)} p_1(x,y)
  - \inf_{x\in \cB_R(x_0)} p_1(x,y) 
  \leqs C_1 \Biggpar{\frac{R}{R_0}}^\alpha 
  \osc_{\cB_{R_0}(x_0)} p_1\;. 
\end{equation} 
\end{enumerate}
\end{lemma}

Using~\eqref{eq:Harnack1}, one can then show 
(cf.~\cite[Lemma~5.7]{BG_periodic2}) that for $y$ in a compact set $D$, one has 
the rough a priori bound 
\begin{equation}
 \frac{\displaystyle\sup_{x\in D} p_1(x,y)}
 {\displaystyle\inf_{x\in D} p_1(x,y)}
 \leqs \e^{C/\sigma^2}
\end{equation} 
for a constant $C$, depending on $D$, but not on $\sigma$. Furthermore, 
combining~\eqref{eq:Harnack2} with~\eqref{eq:Harnack1}, one obtains 
(cf.~\cite[Lemma 5.8]{BG_periodic2}) that for $x_0, y$ in $D$ and any $\eta > 
0$, there exists $r = r(y,\eta) > 0$, independent of $\sigma$, such that 
\begin{equation}
 \sup_{x\in\cB_{r\sigma^2}(x_0)} p_1(x,y) 
 \leqs 
 (1+\eta) \inf_{x\in\cB_{r\sigma^2}(x_0)} p_1(x,y)\;.
\end{equation} 
This result can now be extended to larger balls by using a coupling argument. 
Let $X^x_t$ denote the solution of~\eqref{eq:SDE_example_Birkhoff} with 
initial condition $x$, conditionned to stay in $D$ up to time $t$. For $x_1, 
x_2,y\in D$, let 
\begin{equation}
 N(x_1,x_2) 
 = \inf\Bigsetsuch{n\geqs 1}{\abs{X^{x_2}_n - X^{x_1}_n} < r(\eta,y)\sigma^2}\;.
\end{equation} 
If $p^D_n$ denotes the transition density at time $n$ of the process 
conditioned to stay in $D$, one obtains~\cite[Proposition~5.9]{BG_periodic2} 
that for all $n\geqs2$, one has 
\begin{equation}
\label{eq:bound_proof_ratio} 
 \frac{\displaystyle\sup_{x\in D} p^D_n(x,y)}
 {\displaystyle\inf_{x\in D} p^D_n(x,y)}
 \leqs 1 + \eta 
 + \sup_{x_1,x_2\in D} \Bigprob{N(x_1,x_2) > n-1} \e^{C/\sigma^2}
\end{equation} 
for a constant $C$ independent of $\sigma$ and $y\in D$. Controlling the tails 
of $N(x_1,x_2)$ thus amounts to a coupling argument, with an error of size 
$r\sigma^2$. To do this, we observe that the difference 
$Y_t = X^{x_1}_t - X^{x_2}_t$ satisfies the equation 
\begin{equation}
 \6Y_t = A Y_t \6t + b(t,Y_t) \6t\;,
\end{equation} 
where $b(t,y) = \Order{y^2}$ for small $y$. Using the integral representation 
\begin{equation}
 Y_t = (x_2 - x_1) \e^{At} + \int_0^t \int_0^t \e^{A(t-s)} b(s,Y_s) \6s\;,
\end{equation} 
one can prove an estimate of the form 
\begin{equation}
 \Bigprob{\norm{X^{x_1}_1 - X^{x_2}_1} \geqs \rho\norm{x_2 - x_1}} 
 \leqs \e^{-\kappa/\sigma^2}
\end{equation} 
uniformly over $x_1, x_2\in D$, for some $\rho\in(0,1)$ and $\kappa>0$. Using 
the Markov property\footnote{We simplified the argument somewhat, because one 
has to account for the difference between the initial process, and the process 
conditioned on staying in $D$. See~\cite[Proposition~6.13]{BG_periodic2} for 
the precise argument.}, one arrives at
\begin{equation}
 \Bigprob{N(x_1,x_2) > n} \leqs \e^{-n\kappa/\sigma^2}\;,
\end{equation} 
so that by choosing $n$ large enough, one can make the right-hand side 
of~\eqref{eq:bound_proof_ratio} smaller than $1+2\eta$. We thus obtain a 
bound on the variation of the map $x\mapsto p(x,y)$ inside $D$, yielding a 
uniform positivity property of the form~\eqref{eq:uniform_positivity} with 
$m(y)=1$ and $L$ close to $1$.



\chapter{Large deviations for stochastic differential equations}
\label{chap:ld} 

The theory of large deviations is a very powerful tool to analyse sequences 
$(X_n)_{n\geqs0}$ or $(X_\eps)_{\eps>0}$ of random variables that obey some 
scaling behaviour as $n\to\infty$ or $\eps\to0$. Standard textbooks on that 
theory are~\cite{DZ} and~\cite{DS}. In the particular case of scaled Brownian 
motion $\sqrt{\eps}W_t$, Schilder's theorem provides a large-deviation 
principle, which can be transferred to SDEs using the so-called contraction 
principle. This is the starting point of the theory developed by Freidlin and 
Wentzell in the monograph~\cite{FW}. 

In this chapter, we will first provide a general introduction to 
large-deviation principles, before specialising it to the case of SDEs, and 
giving some applications to the exit problem of a diffusion from a domain. 


\section{Large-deviation principles}
\label{sec:ldp} 


\subsection{A simple example}
\label{ssec:ldp_ex} 

Let $S_n$ be the number of heads obtained when throwing a fair coin $n$ times. 
Then $S_n$ follows a binomial distribution with parameters $(n,\frac12)$, and 
thus has expectation $\frac{n}{2}$ and variance $\frac{n}{4}$. Therefore, the 
law of large numbers shows that 
\begin{equation}
 \label{lgn9}
\lim_{n\to\infty}\Biggprob{\biggabs{\frac{S_n}n - \frac12} \geqs \delta} 
= 0 
\end{equation} 
for any $\delta>0$, that is, $\frac{S_n}{n}$ converges to $\frac12$ in 
probability. The strong law of large numbers states that this convergence also 
holds almost surely. The central limit theorem states that the variable
\begin{equation}
 \widehat S_n = \frac{S_n - \frac{n}{2}}{\sqrt{\frac{n}{4}}}
\end{equation} 
converges in distribution to a standard normal law as $n\to\infty$. This 
indicates that for large $n$, $S_n$ is likely to belong to an interval of order 
$\sqrt{n}$ around $\frac{n}{2}$. 

One may be interested in more precise estimates on the probability of $S_n$ 
being very far from its expected value. One way of estimating this probability 
is via the Bienaym\'e--Chebyshev inequality, which yields  
\begin{equation}
\label{lgn10}
\Biggprob{\biggabs{\frac{S_n}n - \frac12} \geqs \delta}
\leqs \frac1{\delta^2} \variance\biggpar{\frac{S_n}n} = \frac1{4n\delta^2}\;.
\end{equation}
However, while this upper bound is rigorous, it may be too pessimistic for 
certain values of $n$ and $\delta$. 

\begin{example}
\label{ex_lgn1}
For $n = 1000$ and $\delta=0.1$, \eqref{lgn10} yields 
\begin{equation}
\label{lgn11}
\Biggprob{\biggabs{\frac{S_{1000}}{1000} - \frac12} \geqs \delta}
= \Bigprob{S_{1000}\notin[400,600]}
\leqs \frac1{40} = 0.025\;.
\end{equation}
However, one would expect this probability to be much smaller. For instance, 
one might be tempted to infer from the central limit theorem the bound 
\begin{equation}
 \Biggprob{\biggabs{\frac{S_n}{n} - \frac12} \geqs \delta}
 = \Biggprob{\biggabs{S_{n} - \frac n2} \geqs n\delta} 
 \simeq \Biggprob{\abs{Z} \geqs \sqrt{\frac n4}n\delta}
 \leqs \Bigprob{\abs{Z} \geqs 1581}\;, 
\end{equation} 
where $Z$ follows a standard normal law. However, nothing guarantees that 
we are allowed to apply the central-limit theorem in this $n$-dependent way! 
\end{example}

The Bienaym\'e--Chebychev inequality can be improved by using the so-called 
Chernoff bound. For any $t>0$, one has 
\begin{equation}
\label{lgn12}
\biggprob{S_n - \frac n2 \geqs n\delta}
= \Bigprob{\e^{t(S_n - n/2)} \geqs \e^{t n\delta}}
\leqs \frac1{\e^{t n\delta}} \bigexpec{\e^{t(S_n - n/2)}}\;.
\end{equation}
We can write 
\begin{equation}
 S_n = \sum_{i=1}^n X_i
\end{equation} 
where the $X_i$ are independent, identically distributed, taking values $0$ or 
$1$ with probability $\frac12$. Setting $Z_i = \e^{t(X_i-1/2)}$we obtain 
\begin{equation}
\label{lgn13}
\expec{Z_i} = \e^{-t/2} \prob{X_i=0} + \e^{t/2} \prob{X_i=1}
= \cosh\biggpar{\frac{t}{2}}\;.
\end{equation}
Since $\e^{t(S_n - n/2)} = \prod_{i=1}^n
Z_i$, independence of the $Z_i$ implies 
\begin{equation}
\label{lgn14}
\bigexpec{\e^{t(S_n - n/2)}} = \biggexpec{\prod_{i=1}^n Z_i} 
= \prod_{i=1}^n \expec{Z_i} = \cosh^n\biggpar{\frac{t}{2}}\;. 
\end{equation}
This last expression can be written as 
$\exp\brak{n \log\cosh\bigpar{\frac{t}{2}}}$. Inequality \eqref{lgn12} 
thus becomes 
\begin{equation}
\label{lgn15}
\biggprob{S_n - \frac n2 \geqs n\delta} 
\leqs \exp\Biggset{-n\biggbrak{t\delta - 
\log\cosh\biggpar{\frac{t}{2}}}}\;. 
\end{equation}
Consider the function 
\begin{equation}
 f: t\mapsto t\delta - 
\log\cosh\biggpar{\frac{t}{2}}
\end{equation} 
which is maximal for $\tanh\bigpar{\frac{t}{2}}=2\delta$, that is, for 
$t=t^*=\log\frac{1+2\delta}{1-2\delta}$. Let us set   
\begin{align}
\nonumber
\cI(\delta) = f\bigpar{t^*} 
&= \delta \log\frac{1+2\delta}{1-2\delta} - 
\log\Biggpar{\sqrt{\frac{1-2\delta}{1+2\delta}}\,} \\
&= \biggpar{\frac12+\delta} \log(1+2\delta) + \biggpar{\frac12-\delta}
\log(1-2\delta)\;.
\label{lgn16}
\end{align}
Using~\eqref{lgn15} for $t = t^*$, and the same estimate for 
$\prob{S_n - \frac n2 \leqs -n\delta}$, we get 
\begin{equation}
\label{lgn17}
\biggprob{\biggabs{\frac{S_n}n - \frac12} \geqs \delta}
\leqs 2 \e^{-n \cI(\delta)}\;.
\end{equation} 
This is a particular case of large-deviation estimate, and $\cI(\delta)$ is 
called the~\emph{rate function}, see~\figref{fig:LDP}. The fact 
that~\eqref{lgn16} is reminiscent of entropy is no coincidence. 

\begin{figure}
\begin{center}
\scalebox{1.0}{
\begin{tikzpicture}[>=stealth',main node/.style={circle,minimum
size=0.25cm,fill=blue!20,draw},x=8cm,y=8cm]

\draw[->,thick] (-0.7,0) -> (0.7,0);
\draw[->,thick] (0,-0.05) -> (0,0.8);

\foreach \x in {-5,...,5}
\draw[semithick] (0.1*\x,-0.015) -- (0.1*\x,0.015);

\draw[semithick] (0.5,-0.015) -- node[below=0.15cm] {{\small $0.5$}}
(0.5,0.015);

\draw[semithick] (-0.5,-0.015) -- node[below=0.15cm] {{\small
$-0.5\;\;$}} (-0.5,0.015);

\foreach \y in {1,...,7}
\draw[semithick] (-0.015,0.1*\y) -- (0.015,0.1*\y);

\draw[semithick] (-0.015,0.5) -- node[right=0.15cm] {{\small $0.5$}}
(0.015,0.5);

\newcommand*{\imax}{ln(2)}

\draw[dashed,semithick,dash pattern=on 5pt off 5pt] (-0.5,0) -- (-0.5,{\imax})
-- (0.5,{\imax}) -- (0.5,0); 

\draw[blue,very thick,-,smooth,domain=-0.5:0.5,samples=100,/pgf/fpu,
/pgf/fpu/output format=fixed] plot (\x, {
(0.5 + \x )*ln(1 + 2* \x ) + (0.5 - \x )*ln(1- 2* \x ) 
}) -- (0.5,{\imax});

\node[] at (0.65,-0.05) {$t$};
\node[] at (0.06,0.75) {$I(t)$};
\end{tikzpicture}
}
\vspace{-3mm}
\end{center}
\caption[]{Plot of the rate function $\delta \mapsto \cI(\delta)$.}
\label{fig:LDP}
\end{figure}
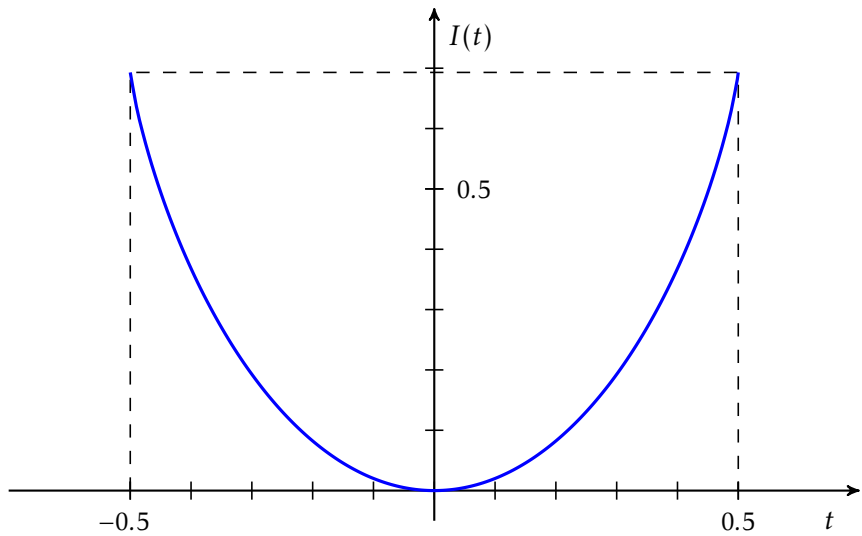

\begin{example}
Continuing with Example~\ref{ex_lgn1}, since $I(0.1)\simeq 0.02$,
we get the much smaller bound
\begin{equation}
\label{lgn18}
\Bigprob{S_{1000}\notin[400,600]} \leqs 2 \e^{-1000 \,I(0.1)} \simeq 3.6
\cdot 10^{-9}\;.
\end{equation} 
\end{example}

In fact, the estimate~\eqref{lgn17} is related to a particular case of 
Cram\'er's theorem. 

\begin{theorem}[Cram\'er's theorem]
\label{eq:thm_Cramer} 
Let $(X_n)_{n\geqs0}$ be a sequence of independent, identically distributed 
random variables, such that the logarithmic moment generating function 
$ \Lambda(t) = \log \bigexpec{\e^{t X_1}}$
is finite for all $t\in\R$. Then 
\begin{equation}
 \lim_{n\to\infty} \frac{1}{n} \log\Biggprob{\sum_{i=1}^n X_i \geqs nx}
 = -\Lambda^*(x)
\end{equation} 
for all $x > \expec{X_1}$, where 
\begin{equation}
 \Lambda^*(x) = \sup_{t\in\R} \bigpar{tx - \Lambda(t)}
\end{equation} 
is the Legendre transform of $\Lambda$.   
\end{theorem}

\begin{exercise}
Compute $\Lambda^*(x)$ for a Gaussian random variable, and for an exponential 
random variable. 
\end{exercise}


\subsection{Definition of large-deviation principles}
\label{ssec:ldp_def} 

\begin{definition}[Large-deviation principle]
\label{def:ldp} 
Let $(S,d)$ be a separable metric space. We say that a family 
$\set{\mu_\eps}_{\eps>0}$ of probability measures on $S$ satisfies a large 
deviation principle with rate $\eps$ and rate function $\cI$ if
\begin{enumerate}
\item 	 $\cI: S \to [0, +\infty]$ is lower semicontinuous, has compact 
sublevel sets and is not identical to $+\infty$.

\item 	For every closed set $C\subset S$ one has 
\begin{equation}
\label{eq:LPD_upper_bound} 
 \limsup_{\eps\to 0} \eps\log \mu_\eps(C) 
 \leqs -\inf_{s\in C} \cI(s)\;.
\end{equation} 

\item 	For every open set $O\subset S$ one has 
\begin{equation}
\label{eq:LDP_lower_bound} 
 \liminf_{\eps\to 0} \eps\log \mu_\eps(O) 
 \geqs -\inf_{s\in O} \cI(s)\;.
\end{equation} 
\end{enumerate}
\end{definition}

Roughly speaking, the large-deviation principle says that for any  
sufficiently nice set $\Gamma\subset S$, we have 
\begin{equation}
 \mu_\eps(\Gamma)
 \simeq \e^{-\inf_\Gamma \cI/\eps}
\end{equation} 
in the sense of logarithmic equivalence. 

\begin{example}[Cram\'er's theorem as a large-deviation principle]
Setting $n = \intpart{\eps^{-1}}$, Theorem~\ref{eq:thm_Cramer} shows that 
$S_n = \sum_{i=1}^n X_i$ satisfies a large-deviation principle with rate 
function 
\begin{equation}
 \cI(s) = \Lambda^*(s)\;.
\end{equation} 
\end{example}


\subsection{The contraction principle}
\label{ssec:ldp_contraction} 

The contraction principle is a powerful method to generate new large-deviation 
principles from already known ones. The following version of this principle 
follows directly from~\cite[Lemma~2.1.4]{DS}, see 
also~\cite[Lemma~3.3]{HairerWeber}. 

\begin{lemma}[Contraction principle]
\label{lem:contraction_principle} 
Let $\set{\mu_\eps}_{\eps>0}$ be a family of probability measures on a 
separable metric space $(S,d)$, and let $\cI:S\to[0,\infty]$ satisfy the first 
condition of Definition~\ref{def:ldp}. 
Let $(S',d')$ be another separable metric space, and let 
$\setsuch{\Psi_\eps}{\eps\geqs0}$ be a family of maps from $S$ to $S'$ which are 
continuous on a neighbourhood of $\setsuch{s\in S}{\cI(s) < \infty}$. 
We assume that
\begin{enumerate}
 \item 	The probability measures $\set{\mu_\eps}_{\eps>0}$ satisfy a 
large-deviation principle on $S$ with rate function $\cI$.

\item 	For every $c\in\R$, there exists a neighbourhood $O_c$ of $\setsuch{s 
\in S}{\cI(s) \leqs c}$ such that the maps $\Psi_\eps$ converge uniformly on 
$O_c$ to $\Psi_0$.
\end{enumerate}
Then the image measures $\mu_\eps\circ\Psi_\eps^{-1}$ satisfy a 
large-deviation principle on $S'$ with rate function
\begin{equation}
 \cI'(s') = \inf\bigsetsuch{\cI(s)}{s\in S, \Psi_0(s) = s'}\;,
\end{equation}
with the convention that the infimum equals $+\infty$ if the set is empty.
\end{lemma}


\section{Large deviations for SDEs}
\label{sec:ldp_SDE} 


\subsection{Schilder's theorem}
\label{ssec:ldp_Schilder} 

We consider in this section scaled Brownian motion in $\R$ or $\R^n$ defined 
for any $\eps>0$ by 
\begin{equation}
 W^\eps_t = \sqrt{\eps} W_t\;.
\end{equation} 
Our aim is to obtain a large-deviation principle on sample-paths of $W^\eps_t$, 
meaning that it should apply to the random functions $\set{W^\eps_t}_{0\leqs 
t\leqs T}$ for some fixed time horizon $T>0$. Let us first consider some 
particular cases of subsets of sample paths. 

\begin{example}[Brownian motion at a fixed time]
\label{ex:LPD_BM1} 
Fix a time $T>0$ and an open set $A\in\R^n$. Since 
\begin{equation}
 \bigprob{W^\eps_T \in A} 
 = \int_A \frac{1}{\sqrt{(2\pi\eps T)^n}} \e^{-x^2/(2\eps T)} \6x\;,
\end{equation} 
applying the Laplace method shows that 
\begin{equation}
\label{eq:LDP_ex1} 
 \lim_{\eps\to0} \eps \log \bigprob{W^\eps_T \in A}
 = -\frac 12 \inf_{x\in A} \frac{\norm{x}^2}{T}\;.
\end{equation} 
\end{example}

\begin{example}[$1d$ scaled Brownian motion staying below level $h$]
\label{ex:LPD_BM2} 
The reflection principle implies 
\begin{equation}
 \Biggprob{\sup_{0 \leqs t \leqs T} W^\eps_t > h}
 = 2 \bigprob{W^\eps_T > h} 
 = 2 \Biggprob{W_T > \frac{h}{\sqrt{\eps}}}
 = 2 \int_{h/\sqrt{\eps}}^\infty \frac{\e^{-x^2/2}}{\sqrt{2\pi\eps}}\6x\;,
\end{equation} 
which yields
\begin{equation}
\label{eq:LDP_ex2} 
 \lim_{\eps\to0} \eps \log \Biggprob{\sup_{0 \leqs t \leqs T} W^\eps_t > h}
 = -\frac{h^2}{2T}\;.
\end{equation} 
\end{example}

In both cases, we would like to write the right-hand side 
of~\eqref{eq:LDP_ex1} and~\eqref{eq:LDP_ex2} as the infimum of 
some rate function over all sample paths fulfilling the property defining the 
event. It turns out that such a rate function is provided by the following 
theorem.  

\begin{theorem}[Schilder's theorem]
\label{thm:Schilder}
Let $\cC_0$ be the space of continuous functions $\ph:[0,T]\to\R^n$ such that 
$\ph(0)=0$. Scaled Brownian motion on $[0,T]$ satisfies a large-deviation 
principle on $\cC_0$ with rate function 
\begin{equation}
 \cI_{[0,T]}(\ph) = 
 \begin{cases}
  \displaystyle\frac12 \norm{\ph}_{H^1}^2 
  = \frac12 \int_0^T \norm{\dot\ph(s)}^2 \6s 
  & \text{if $\ph\in H^1$ and $\ph(0) = 0$\;,} \\[14pt]
  +\infty 
  & \text{otherwise\;.}
 \end{cases}
\end{equation} 
\end{theorem}

Before giving a proof of this result, let us look at some of its consequences. 
Let $\psi:[0,T]\to\R^n$ be a path vanishing at times $0$ and $T$. Then the 
Gateaux derivative of the rate function in the direction $\psi$ is given by 
\begin{equation}
\label{eq:BM_Gateaux} 
 \dtot{}{t} \cI_{[0,T]}(\ph + t \psi) \Bigr\vert_{t=0}
 = \int_0^T \pscal{\dot\ph(s)}{\dot\psi(s)} \6s 
 = -\int_0^T \pscal{\ddot\ph(s)}{\psi(s)} \6s\;.
\end{equation} 
Stationary points of the rate function are thus straight lines of the form 
$\ph(s) = sv$ for some $v\in\R^n$. This is nothing but the Euler--Lagrange 
variational principle applied to the Lagrangian of a free particle. In 
Example~\ref{ex:LPD_BM1}, we can apply the large-deviation principle to the set 
of paths $\ph$ such that $\ph(T)\in A$. The infimum of the rate function is 
achieved by the path 
\begin{equation}
 \ph(s) = \frac{s}{T} x_0\;, \qquad 
 x_0 = \arginf_{x\in A} \norm{x}\;, 
\end{equation} 
and one has indeed $\cI_{[0,T]}(\ph) = \frac{1}{2T}\norm{x_0}^2$. A similar 
argument applies to Example~\ref{ex:LPD_BM2}, with $\ph(s) = \frac sT h$.  

Let us now outline a proof of Schilder's theorem. Its main ingredient is the 
Cameron--Martin--Girsanov formula, which states that any drifted Brownian 
motion is equivalent to standard Brownian motion under a change of measure. 

\begin{lemma}[Cameron--Martin--Girsanov formula]
Let $(W_t)_{t\geqs0}$ be a standard $1$-dimensional Brownian motion on 
$(\Omega,\cF_t,\fP)$. Then for any $h\in L^2$, the process defined by   
\begin{equation}
 \widehat W_t = W_t - \int_0^t h(s)\6s
\end{equation} 
is a standard Brownian motion on $(\Omega,\cF_t,\Q)$, where $\Q$ is defined via 
its Radon--Nikodym derivative
\begin{equation}
\label{eq:Girsanov_RDderivative} 
 \dtot{\Q}{\fP} \Biggr|_{\cF_t} 
 = \exp\Biggset{\int_0^t h(s)\6W_s - \frac12 \int_0^t h(s)^2\6s}\;.
\end{equation} 
\end{lemma}
\begin{proof}[\Sketch]
For $h\in L^2$ and $\gamma > 0$, the processes 
\begin{align}
 X_t &= \exp\Biggset{\int_0^t h(s)\6W_s - \frac12 \int_0^t h(s)^2\6s}\;, \\
 Y_t &= \exp\Biggset{\int_0^t \bigpar{\gamma + h(s)}\6W_s 
 - \frac12 \int_0^t \bigpar{\gamma + h(s)}^2\6s} 
 = X_t \exp\Biggset{\gamma \widehat W_t - \frac12\gamma^2 t}
\end{align}
are exponential martingales (similar to exponential Brownian motion) with 
respect to $\fP$. This means that $\bigecond{X_t}{\cF_s} = X_s$ and 
$\bigecond{Y_t}{\cF_s} = Y_s$  whenever $t>s$. 
For any $\cF_s$-measurable random variable $Z$, one has 
\begin{align}
 \Bigexpeclaw{\Q}{Z \e^{\gamma(\widehat W_t - \widehat W_s)}}
 &= \Bigexpeclaw{\fP}{Z X_t\e^{\gamma(\widehat W_t - \widehat W_s)}} \\
 &= \Bigexpeclaw{\fP}{Z \e^{-\gamma \widehat W_s} Y_t\e^{\frac12\gamma^2 t}} \\
 &= \Bigexpeclaw{\fP}{Z \e^{-\gamma \widehat W_s + \frac12\gamma^2 t} 
 \bigecondlaw{\fP}{Y_t}{\cF_s}} \\
 &= \Bigexpeclaw{\fP}{Z \e^{-\gamma \widehat W_s + \frac12\gamma^2 t} Y_s} \\
 &= \Bigexpeclaw{\fP}{Z X_s\e^{\frac12\gamma^2(t-s)}} \\
 &= \bigexpeclaw{\Q}{Z}\e^{\frac12\gamma^2(t-s)}\;,
\end{align}
where the third line follows from the tower property of conditional 
expectations, and the fourth line from the martingale property. 
Since $\widehat W_t - \widehat W_s$ is $\Q$-independent of $\cF_s$, the 
increments if $\widehat W_t$ are independent. The above computation shows that 
these increments are Gaussian of variance $t-s$, proving that indeed $\widehat 
W_t$ is a Brownian motion with respect to the measure $\Q$.
\end{proof}

\begin{proof}[\textsc{Proof of Theorem~\ref{thm:Schilder}}]
Assume for simplicity we are in dimension $n=1$. 
We start by proving the lower bound~\eqref{eq:LDP_lower_bound}. We thus have to 
prove that for any open set $O \subset \cC_0$, one has 
\begin{equation}
 \liminf_{\eps\to0} \eps\log\bigprob{W^\eps \in O}
 \geqs - \inf_{\ph\in O} \cI_{[0,T]}(\ph)\;.
\end{equation} 
For this, it is in fact sufficient to prove that 
\begin{equation}
\label{eq:proof_LDP_lb} 
 \liminf_{\eps\to0} \eps\log\bigprob{W^\eps \in \cB_\delta(\ph)}
 \geqs - \cI_{[0,T]}(\ph)
\end{equation} 
holds for all $\ph\in\cC_0$ such that $\cI_{[0,T]}(\ph) < \infty$ and all 
$\delta>0$, where $\cB_\delta(\ph)$ denotes the ball of radius $\delta$ in the 
sup norm, centred in $\ph$. 

If we set $\widehat W_t = W_t - \frac{1}{\sqrt{\eps}}\ph_t$ and define $\6\Q$ 
by~\eqref{eq:Girsanov_RDderivative}, we have 
\begin{align}
 \bigprob{\norm{W^\eps - \ph}_\infty < \delta} 
 &= \biggprob{\norm{\widehat W}_\infty < \frac{\delta}{\sqrt{\eps}}} 
 = \int_{\widehat W \in \cB_{\delta/\sqrt{\eps}}(0)} \6\fP  \\
 &= \int_{\widehat W \in \cB_{\delta/\sqrt{\eps}}(0)} 
 \exp\biggset{-\frac{1}{\sqrt{\eps}} \int_0^T \dot\ph_s\6\widehat W_s 
 + \frac{1}{2\eps} \int_0^T \ph_s^2 \6s} \6\Q \\
 &= \e^{-\cI_{[0,T]}(\ph)/\eps} 
 \int_{\widehat W \in \cB_{\delta/\sqrt{\eps}}(0)} 
 \exp\biggset{-\frac{1}{\sqrt{\eps}} \int_0^T \dot\ph_s\6\widehat W_s} \6\Q\;.
\end{align}
Jensen's inequality yields 
\begin{multline}
 \frac{1}{\Q\bigset{\widehat W \in \cB_{\delta/\sqrt{\eps}}(0)}}
 \int_{\widehat W \in \cB_{\delta/\sqrt{\eps}}(0)} 
 \exp\biggset{-\frac{1}{\sqrt{\eps}} \int_0^T \dot\ph_s\6\widehat W_s} \6\Q \\
 \geqs{} \exp \Biggset{-\frac{1}{\sqrt{\eps} \Q\bigset{\widehat W \in 
\cB_{\delta/\sqrt{\eps}}(0)}} \int_{\widehat W \in \cB_{\delta/\sqrt{\eps}}(0)} 
\int_0^T \dot\ph(s)\6\widehat W_s \6\Q}\;.
\end{multline} 
It thus follows from the Cameron--Martin--Girsanov formula that 
\begin{align}
 \bigprob{&\norm{W^\eps - \ph}_\infty < \delta} \\
 &\geqs \e^{-\cI_{[0,T]}(\ph)/\eps}  
 \fP\bigset{W \in \cB_{\delta/\sqrt{\eps}}(0)}
 \exp \Biggset{-\frac{1}{\sqrt{\eps} 
 \fP\bigset{W \in \cB_{\delta/\sqrt{\eps}}(0)}} 
 \int_{W \in \cB_{\delta/\sqrt{\eps}}(0)} 
\int_0^T \dot\ph(s)\6W_s \6\fP} \\
 &= \e^{-\cI_{[0,T]}(\ph)/\eps}  
 \fP\bigset{W \in \cB_{\delta/\sqrt{\eps}}(0)}\;.
\end{align}
Noting that the probability in the last expression goes to $1$ as $\eps$ 
decreases to $0$, we obtain indeed the required lower 
bound~\eqref{eq:proof_LDP_lb}. 

Regarding the upper bound~\eqref{eq:LPD_upper_bound}, we have to show that for 
any closed set $C \subset\cC_0$, one has 
\begin{equation}
 \limsup_{\eps\to0} \eps\log\bigprob{W^\eps \in O}
 \geqs - \inf_{\ph\in C} \cI_{[0,T]}(\ph)\;.
\end{equation} 
To do this, given $m\in\N$ we introduce a polygonal approximation $W^{m,\eps}$ 
of the Brownian path $W^\eps$, joining the $m+1$ points $(0,0)$, $(T/m, 
W^\eps_{T/m})$, \dots, $(T,W^\eps_T)$. For any $\delta > 0$, one can find a set 
$C^\delta$ whose boundary is at distance $\delta$ from the boundary of $C$, and 
such that 
\begin{equation}
\label{eq:proof_Schilder_ub} 
 \bigprob{W^\eps \in C}
 \leqs \bigprob{W^{m,\eps} \in C^\delta}
 + \bigprob{\norm{W^\eps - W^{m,\eps}}_\infty \geqs \delta}\;.
\end{equation} 
We first show that the second term on the right-hand side is negligible. 
Indeed, by the stationary increments property of Brownian motion, we have 
\begin{align}
 \bigprob{\norm{W^\eps - W^{m,\eps}}_\infty \geqs \delta}
 &\leqs m \biggprob{\sup_{0\leqs s\leqs T/m} \norm{W^\eps_s - 
W^{m,\eps}_s}_\infty \geqs \delta} \\
 &\leqs m \biggprob{\sup_{0\leqs s\leqs T/m} \norm{W^\eps_s} \geqs 
\frac{\delta}{2}} \\
 &= m \biggprob{\sup_{0\leqs s\leqs T/m} \norm{W_s} \geqs 
\frac{\delta}{2\sqrt{\eps}}} \\
 &\leqs 2m \e^{-m\delta^2/(8\eps T)}\;,
\end{align}
where the last bound is a standard bound for Brownian motion, that can be 
deduced from Doob's submartingale inequality. It follows that for any 
$\delta>0$, one has 
\begin{equation}
 \limsup_{m\to\infty} \limsup_{\eps\to0}
 \bigprob{\norm{W^\eps - W^{m,\eps}}_\infty \geqs \delta} 
 = -\infty\;.
\end{equation}
We now deal with the first term on the right-hand side 
of~\eqref{eq:proof_Schilder_ub}. Note that 
\begin{equation}
 \bigprob{W^{m,\eps} \in C^\delta}
 \leqs \biggprob{\cI_{[0,T]}(W^{m,\eps}) \geqs \inf_{\ph\in C^\delta} 
\cI_{[0,T]}(\ph)}\;.
\end{equation} 
Since $W^{m,\eps}$ is a random polygon, we have 
\begin{equation}
 \cI_{[0,T]}(W^{m,\eps})
 = \frac12 \sum_{k=1}^m \frac{T}{m}
 \Bignorm{\frac{m}{T} \bigpar{W^\eps_{kT/m} - W^\eps_{(k-1)T/m}}}^2
 = \frac{\eps}{2} \sum_{k=1}^m \xi_k^2\;,
\end{equation} 
where the $\xi_k$ are independent, with standard normal distribution. 
Therefore, Markov's inequality yields for any $\gamma < \frac12$ and 
$\ell\geqs0$ the bound 
\begin{equation}
 \Bigprob{\cI_{[0,T]}(W^{m,\eps}) \geqs \ell} 
 \leqs \e^{-2\gamma\ell/\eps} \Bigpar{\expec{\e^{\gamma\xi_1^2}}}^m 
 \leqs \frac{\e^{-2\gamma\ell/\eps}}{(1-2\gamma)^{m/2}}\;.
\end{equation} 
Since $\gamma < \frac12$ is arbitrary, it follows that 
\begin{equation}
 \limsup_{m\to\infty}\, \limsup_{\eps\to0}\,
 \eps\log \biggprob{\cI_{[0,T]}(W^{m,\eps}) \geqs \inf_{\ph\in C^\delta} 
\cI_{[0,T]}(\ph)} 
 \leqs -\inf_{\ph\in C^\delta} \cI_{[0,T]}(\ph)\;, 
\end{equation} 
and the result follows by taking the limit $\delta\searrow0$. 
\end{proof}

\begin{remark}
The proof we have given here is essentially the one found 
in~\cite[Chapter~3, Theorem~2.2]{FW}. There exist other proofs, however, for 
instance a proof based on the Wiener isometry can be found 
in~\cite{HairerWeber}. 
\end{remark}


\subsection{The Freidlin--Wentzell large-deviation principle}
\label{ssec:ldp_FW} 

We consider from now on SDEs of the form 
\begin{equation}
\label{eq:SDE_LDP} 
 \6X^\eps_t = f(X^\eps_t) \6t + \sqrt{\eps} g(X^\eps_t) \6W_t\;,
\end{equation} 
where $W_t$ denotes $k$-dimensional Brownian motion, $X^\eps_t$ is 
$n$-dimensional, $f:\R^n\to\R^n$ is a drift coefficient, and 
$g:\R^n\to\R^{n\times k}$ is a matrix-valued diffusion coefficient. We assume 
that $f$ and $g$ satisfy the Lipshitz and bounded-growth conditions seen in 
Section~\ref{ssec:Ito_SDEs}. Assume furthermore an ellipticity condition on the 
diffusion matrix $D(x) = g(x)g(x)^T$, requiring that 
\begin{equation}
\label{eq:ellipticity} 
 \pscal{\xi}{D(x) \xi} \geqs c \norm{\xi^2} 
 \qquad \forall \xi\in\R^d
\end{equation} 
holds for some $c>0$. Note that this is only possible if $k\geqs n$, since 
otherwise, $D(x)$ cannot have full rank. 

Recall that the solution of~\eqref{eq:SDE_LDP} with initial condition 
$X^\eps_0 = x_0$ is defined as the fixed point of the integral equation 
\begin{equation}
\label{eq:FP_SDE_LDP} 
 X^\eps_t = x_0 + \int_0^t f(X^\eps_s)\6s + \int_0^t g(X^\eps_s)\6W^\eps_s\;,
\end{equation} 
where $W^\eps_s = \sqrt{\eps}W_s$ denotes scaled Brownian motion. We view the 
right-hand side as a map $\Psi$ from realisations of 
$(W^\eps_t)_{t\in[0,T]}$ to sample paths $(X^\eps_t)_{t\in[0,T]}$. 

Consider first the case where $g$ is the identity matrix. Then the contraction 
principle, cf.~Lemma~\ref{lem:contraction_principle}, suggests that $X^\eps$ 
satisfies a large-deviation principle with rate function 
\begin{align}
\cJ_{[0,T]}(\ph)
&= \inf\bigsetsuch{\cI_{[0,T]}(\psi)}{\Psi(\psi) = \ph} \\
&= \inf\Biggsetsuch{\frac12 \int_0^T \norm{\dot\psi(s)}^2\6s}
{x_0 + \int_0^t f(\ph(s))\6s + \int_0^t \dot\psi(s)\6s = \ph(t)
\quad\forall t\in[0,T]}\;.
\end{align}
Taking the time-derivative of the condition suggests that one should take 
$\dot\psi(t) = \dot\ph(t) - f(\ph(t))$, yielding 
\begin{equation}
 \cJ_{[0,T]}(\ph) = \frac12 \int_0^T \norm{\dot\ph(s) - f(\ph(s))}^2 \6s\;.
\end{equation} 
Consider next the case where $f=0$, while $g$ is a general matrix satisfying the 
ellipticity condition~\eqref{eq:ellipticity}. Then~\eqref{eq:FP_SDE_LDP} 
suggests choosing $\psi$ in such a way that $g(\ph) \dot\psi = \dot\ph$. 
However, since $g$ is not a square matrix in general, some care is required 
when solving the variational problem.

\begin{lemma}
The minimum of $\norm{\dot\psi}^2$ under the constraint $g(\ph) \dot\psi = 
\dot\ph$ is achieved for $\dot\psi = g(\ph)^T D(\ph)^{-1} \dot\ph$, and has the 
value $\dot\ph^T D^{-1}\dot\ph = \pscal{\dot\ph}{D^{-1}\dot\ph}$. 
\end{lemma}
\begin{proof}
The problem amounts to minimising $p(y) = \norm{y}^2$ over $y\in\R^k$ under the 
constraint $q(y) = g(\ph)y - \dot\ph = 0$. By the Lagrange multiplier theorem, 
there exists a vector $\lambda\in\R^n$ such that 
\begin{equation}
 2y_j 
 = \dpar{p}{y_j}(y) 
 = \sum_{i=1}^n \lambda_i \dpar{q_i}{y_j}(y)
 = \sum_{i=1}^n \lambda_i g_{ij}(\ph)\;, 
 \qquad 
 j\in\set{1,\dots,k}\;.
\end{equation} 
This amounts to setting $2y = g(\ph)^T \lambda$, and the constraint $q(y)=0$ 
yields $g(\ph)g(\ph)^T \lambda = 2\dot\ph$, or $\lambda = 2D(\ph)^{-1}\dot\ph$. 
Therefore, $y = g(\ph)^T D(\ph)^{-1} \dot\ph$, and taking the norm gives the 
result. 
\end{proof}

It follows that for $f=0$, one has 
\begin{equation}
 \cJ_{[0,T]}(\ph) 
 = \frac12 \int_0^T \pscal{\dot\ph(s)}{D(\ph(s))^{-1}\dot\ph(s)} \6s\;. 
\end{equation}
Combining these two special cases yields the following result. Its proof is 
somewhat more complicated than a direct application of the contraction 
principle, because of the continuity properties that need to be verified. 
However, one can work with Euler approximations of the solutions, obtained by a 
time-discretisation. 

\begin{theorem}[Large-deviation principle for SDEs]
\label{thm:WF} 
The processes $(X^\eps_t)_{\eps>0}$ satisfy a large-deviation principle with 
rate function 
 \begin{equation}
 \cJ_{[0,T]}(\ph) 
 = 
 \begin{cases}
  \displaystyle
  \frac12 \int_0^T \pscal{\dot\ph(s) - 
f(\ph(s))}{D(\dot\ph)^{-1}\bigpar{\dot\ph(s) 
  - f(\ph(s))}} \6s 
  & \text{if $\ph\in H^1$ and $\ph(0) = x_0$\;,} \\[14pt]
  +\infty 
  & \text{otherwise\;.}
 \end{cases}
\end{equation}
\end{theorem}

\begin{remark}
If $D(x)$ is only positive semi-definite, the large-deviation principle remains 
valid with 
\begin{equation}
 \cJ_{[0,T]}(\ph) 
 = 
 \inf\Biggsetsuch{\cI(\ph)}{\ph\in H^1, 
 \ph(t) = x_0 + \int_0^t f(\ph(s))\6s + \int_0^t D(\ph(s))^{1/2} \dot\ph(s)\6s 
\quad \forall t\in[0,T]}\;.
\end{equation} 
\end{remark}

\begin{example}[Interpretation of the rate function]
Fix a continuous path $\ph_0:[0,T]\to\R^n$, and define for $\delta>0$ the set 
\begin{equation}
 \Gamma_\delta = \biggsetsuch{\ph\in\cC([0,T],\R^n)}
 {\sup_{t\in[0,T]}\norm{\ph(t) - \ph_0(t)} > \delta}\;,
\end{equation} 
which one can check is open in the topology induced by the supremum norm. 
Then we have $\bigprob{\Gamma_\delta} = \bigprob{\tau(\delta) \leqs T}$, 
where $\tau(\delta)$ denotes the first-exit time from a $\delta$-neighbourhood 
of the path $\ph_0$. The large-deviation principle yields 
\begin{equation}
 -\inf_{\ph\in\Gamma_\delta} \cJ_{[0,T]}(\ph)
 \leqs \liminf_{\eps\to0} \eps\log\bigprob{\tau(\delta) \leqs T} 
 \leqs \limsup_{\eps\to0} \eps\log\bigprob{\tau(\delta) \leqs T} 
 \leqs -\inf_{\ph\in\overline\Gamma_\delta} \cJ_{[0,T]}(\ph)\;.
\end{equation}
Taking the limit $\delta\to0$, we obtain 
\begin{equation}
 \cJ_{[0,T]}(\ph_0) = 
 -\lim_{\delta\to0} \lim_{\eps\to0} 
 \eps\log\bigprob{\tau(\delta) \leqs T}\;.
\end{equation} 
Therefore, $\cJ_{[0,T]}(\ph_0)$ can be interpreted as the cost of following the 
path $\ph_0$. 
\end{example}

Let us also mention Varadhan's lemma, which applies to general random variables 
satisfying a large-deviation principle. The lemma allows to obtain results on 
\lq\lq tilted\rq\rq\ processes, in which unlikely outcomes are made likely by 
changing the probability measure. 

\begin{theorem}[Varadhan's lemma]
Let $\cC$ denote the set of continuous functions $\ph:[0,T]\to\R$, and let 
$\phi:\cC\to\R$ be a continuous map. Assume the tail condition 
\begin{equation}
 \lim_{L\to\infty} \limsup_{\eps\to0} 
 \eps \log \int_{\set{\phi(X^\eps)\geqs L}} 
 \e^{\phi(X^\eps)/\eps} \6\fP 
 = -\infty\;.
\end{equation}
Then 
\begin{equation}
 \lim_{\eps\to 0} \eps \log \int \e^{\phi(X^\eps)/\eps} \6\fP 
 = \sup_{\ph\in\cC} \Bigbrak{\phi(\ph) - \cJ_{[0,T]}(\ph)}\;.
\end{equation} 
\end{theorem}


\section{Application to the stochastic exit problem}
\label{sec:exit} 

One of the many applications of the large-deviation principle for SDEs is the 
stochastic exit problem, that we have already encountered in 
Section~\ref{sec:diffusions}. In particular, the large-deviation principle gives 
information on the probability to leave a given set in finite time, on the 
expected time required to leave the set, and on the most probable path the 
process takes to do so. The approach is in fact quite similar to the geometric 
optics limit of the wave equation, where waves are replaced by light rays moving 
in straight lines. It can also be viewed as a version of  Richard Feynman's 
path-integral approach to quantum mechanics. 


\subsection{The Ornstein--Uhlenbeck process}
\label{ssec:ldp_OU} 

As a relatively simple warm-up example, let us consider an Ornstein--Uhlenbeck 
process in $\R$, defined by the SDE
\begin{equation}
\label{eq:OU} 
 \6X^\eps_t = -X^\eps_t \6t + \sqrt{\eps}\6W_t\;.
\end{equation} 
Let us fix an initial value $x_0 \geqs 0$ and a level $h > x_0$, and ask the 
question of how likely it is that $X^\eps_t$ reaches the level $h$ at or before 
time $T$, when $\eps$ is very small. For that, we define the set 
\begin{equation}
 \Gamma = \biggsetsuch{\ph\in\cC([0,T],\R)}{\ph(0) = x_0, \sup_{0\leqs t\leqs 
T} \ph(t) > h}\;.
\end{equation}
One checks that $\Gamma$ is open in the topology induced by the supremum norm, 
and that its closure is obtained by replacing the condition $\ph(t) > h$ by 
$\ph(t) \geqs h$. 

We have to compute the infimum over $\Gamma$ of the rate function 
\begin{equation}
 \cJ_{[0,T]}(\ph) 
 = \frac12 \int_0^T \bigbrak{\dot\ph(s) + \ph(s)}^2 \6s\;.
\end{equation} 
Proceeding as in~\eqref{eq:BM_Gateaux}, we find that the Gateaux derivative of 
the rate function in the direction $\psi$, where $\psi(0) = \psi(T) = 0$, is 
given by 
\begin{equation}
 \dtot{}{t} \cJ_{[0,T]}(\ph + t \psi) \Bigr\vert_{t=0}
 = \int_0^T \brak{\dot\ph(s) + \ph(s)}\brak{\dot\psi(s) + \psi(s)} \6s
 = \int_0^T \brak{-\ddot\ph(s) + \ph(s)}\psi(s) \6s\;.
\end{equation} 
Stationary points of the rate function thus solve the linear differential 
equation 
\begin{equation}
\label{eq:ODE_OU_LDP} 
 \ddot\ph(s) = \ph(s)\;,
\end{equation} 
the general solution of which has the form $\ph(s) = A\e^s + B\e^{-s}$. 
We still have to deal with the boundary conditions. For that, we fix any time 
$t\in(0,T)$, and define $\ph_t$ by 
\begin{equation}
 \ph_t(s) = 
 \begin{cases}
 \dfrac{h-x_0\e^{-t}}{2\sinh(t)} \e^s 
 + \dfrac{x_0\e^{t}-h}{2\sinh(t)} \e^{-s} 
 & \text{for $0 \leqs s \leqs t$\;,} \\
 h\e^{-(s-t)} & \text{for $t < s \leqs T$\;.}
 \end{cases}
\end{equation} 
\begin{figure}
\begin{center}
\scalebox{1.0}{
\begin{tikzpicture}[>=stealth',main 
node/.style={draw,semithick,circle,fill=white,minimum size=4pt,inner 
sep=0pt},x=2cm,y=1.5cm]

\draw[->,thick] (-0.5,0) -> (5.5,0);
\draw[->,thick] (0,-0.5) -> (0,2.7);

\draw[semithick,dashed] (0,2) -- (5,2); 
\draw[semithick,dashed] (1,0) -- (1,2); 
\draw[semithick,dashed] (3,0) -- (3,2); 
\draw[semithick,dashed] (5,0) -- (5,2); 

\newcommand*{\Aa}{(2-exp(-1))/(2*sinh(1))}
\newcommand*{\Ba}{(exp(1)-2)/(2*sinh(1))}

\newcommand*{\Ab}{(2-exp(-3))/(2*sinh(3))}
\newcommand*{\Bb}{(exp(3)-2)/(2*sinh(3))}

\newcommand*{\Ac}{(2-exp(-5))/(2*sinh(5))}
\newcommand*{\Bc}{(exp(5)-2)/(2*sinh(5))}

\draw[blue,very thick,-,smooth,domain=0:1,samples=25,/pgf/fpu,
/pgf/fpu/output format=fixed] plot (\x, {\Aa *exp(\x) + \Ba *exp(-\x)});

\draw[blue,very thick,-,smooth,domain=1:5,samples=75,/pgf/fpu,
/pgf/fpu/output format=fixed] plot (\x, {2*exp(-(\x-1))});

\draw[violet,very thick,-,smooth,domain=0:3,samples=75,/pgf/fpu,
/pgf/fpu/output format=fixed] plot (\x, {\Ab *exp(\x) + \Bb *exp(-\x)});

\draw[violet,very thick,-,smooth,domain=3:5,samples=25,/pgf/fpu,
/pgf/fpu/output format=fixed] plot (\x, {2*exp(-(\x-3))});

\draw[purple,very thick,-,smooth,domain=0:5,samples=100,/pgf/fpu,
/pgf/fpu/output format=fixed] plot (\x, {\Ac *exp(\x) + \Bc *exp(-\x)});

\node[main node] at (0,2) {};
\node[main node] at (0,1) {};
\node[main node] at (1,0) {};
\node[main node] at (3,0) {};
\node[main node] at (5,0) {};

\node[] at (0.15,2.5) {$x$};

\node[] at (-0.2,1) {$x_0$};
\node[] at (-0.15,2) {$h$};

\node[] at (5.4,-0.2) {$s$};

\node[] at (1,-0.25) {$t_1$};
\node[] at (3,-0.25) {$t_2$};
\node[] at (5,-0.25) {$T$};

\node[blue] at (1.7,1.4) {$\ph_{t_1}(s)$};
\node[violet] at (3.7,1.4) {$\ph_{t_2}(s)$};
\node[purple] at (4.7,1) {$\ph_{T}(s)$};

\end{tikzpicture}
}
\vspace{-3mm}
\end{center}
\caption[]{Several paths used to minimise the rate function for the 
Ornstein--Uhlenbeck process. The infimum is attained on the path $\ph_T$.}
\label{fig:OU_LDP}
\end{figure}
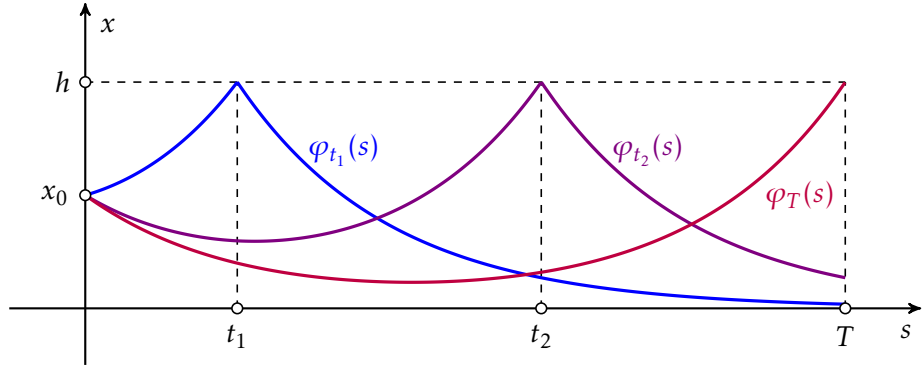
One readily checks (see \figref{fig:OU_LDP}) that 
\begin{itemize}
\item 	$\ph_t$ is continuous, satisfies the equation~\eqref{eq:ODE_OU_LDP} on 
$[0,t)$, the deterministic limit $\dot\ph = -\ph$ of~\eqref{eq:OU} on 
$(t,T]$, and $\ph_t(0) = x_0$ while $\ph_t(t) = h$;
\item 	if $\cosh(t) \leqs h/x_0$, then $\ph_t(s)$ is increasing in $s$ on 
$[0,t]$;
\item 	if $\cosh(t) > h/x_0$, then $\ph_t(s)$ is first decreasing and then 
increasing in $s$ on $[0,t]$, with a positive minimal value.
\end{itemize}
In particular, $\ph_t$ is an element of $\overline\Gamma$, reaching its maximal 
value $h$ at time $t$. Furthermore, 
\begin{align}
 \cJ_{[0,T]}(\ph_t) 
 &= \cJ_{[0,t]}(\ph_t) + \cJ_{[t,T]}(\ph_t)
 = \frac12 \int_0^t \bigbrak{\dot\ph_t(s) + \ph_t(s)}^2 \6s + 0 \\
 &= \frac{(h\e^{t/2} - x_0\e^{-t/2})^2}{2\sinh(t)}
 = h^2 + x_0^2 + \frac{h^2-x_0^2}{\tanh(t)} + \frac{hx_0}{\sinh(t)}\;.
 \label{eq:OU_I0T} 
\end{align} 
Clearly, any modification of $\ph_t$ on the interval $(t,T]$ can only increase 
the value of the rate function. The same is true for modifications on 
the interval $[0,t)$. One way of seeing this is to compute the second variation 
of the rate function, which is given by the bilinear form 
\begin{equation}
 (\psi_1, \psi_2) \mapsto 
 \int_0^T \brak{-\ddot\psi_1(s) + \psi_1(s)}\psi_2(s) \6s\;. 
\end{equation} 
The associated quadratic form is positive definite, because the eigenvalues of 
the second derivative (with Dirichlet boundary conditions on $[0,T]$) are 
negative. Since~\eqref{eq:OU_I0T} is a decreasing function of $t$, we have thus 
obtained 
\begin{equation}
 \inf_{\ph\in\overline\Gamma} \cJ_{[0,T]}(\gamma) 
 = \cJ_{[0,T]}(\gamma_T) 
 = \frac{(h\e^{T/2} - x_0\e^{-T/2})^2}{2\sinh(T)}\;.
\end{equation} 
Making the same argument with $h_1 > h$, and letting $h_1$ decrease to $h$ 
shows that this infimum coincides with the infimum over the open set $\Gamma$ 
itself. By Theorem~\ref{thm:WF}, we conclude that the first-passage time 
$\tau = \inf\setsuch{t\in[0,T]}{X^\eps_t > h}$ satisfies 
\begin{equation}
 \lim_{\eps\to0} \eps\log\bigprob{\tau < T}
 = - \frac{(h\e^{T/2} - x_0\e^{-T/2})^2}{2\sinh(T)}\;.
\end{equation} 
In particular, we have 
\begin{equation}
\label{eq:OU_limit} 
\lim_{T\to\infty} \lim_{\eps\to0} \eps\log\bigprob{\tau < T}
= -h^2\;,
\end{equation} 
showing that for large $T$, the probability of reaching level $h$ before time 
$T$ behaves roughly like $\e^{-h^2/\eps}$, independently of $x_0\in[0,h)$. 


\subsection{Lagrangian and Hamiltonian formulations}
\label{ssec:ldp_Lagrange_Hamilton} 

The minimisation problems we encountered for Brownian motion and for the 
Ornstein--Uhlen\-beck process are in fact well known in analytical mechanics. 
Indeed, the rate function is a Lagrangian action, of the form 
\begin{equation}
 \cJ_{[0,T]}(\ph) = \int_0^T L(\ph(s),\dot\ph(s))\6s\;,
\end{equation} 
where $L$ is a Lagrangian given by 
\begin{equation}
\label{eq:Lagrangian_LDP} 
 L(\ph,\dot\ph) = \frac12 \pscal{\dot\ph - f(\ph)}
{D(\dot\ph)^{-1}\bigpar{\dot\ph - f(\ph)}}\;.
\end{equation} 
Let us disregard for a moment the question of boundary conditions. The 
stationary points of the action are known to satisfy the Euler--Lagrange 
equations 
\begin{equation}
 \dtot{}{t} \biggpar{\dpar{L}{\dot\ph_i}}
 = \dpar{L}{\ph_i}\;, 
 \qquad
 i = 1,\dots, n\;.
\end{equation} 

\begin{exercise}
Write the Euler--Lagrange equations corresponding to Brownian motion, and 
those corresponding to the Ornstein--Uhlenbeck process. Check that their 
solutions are indeed of the form discussed in the previous subsections.  
\end{exercise}

It is often more convenient to use the Hamiltonian formalism. This is done by 
introducing the conjugate momenta 
\begin{equation}
 \psi_i = \dpar{L}{\dot\ph_i}
\end{equation} 
and the Hamiltonian 
\begin{equation}
 H(\ph,\psi) = \sum_{i=1}^n \psi_i\dot\ph_i - L(\ph,\dot\ph)\;,
\end{equation} 
where $\dot\ph$ has to be expressed in terms of the $\psi_i$ on the right-hand 
side. The stationary points of the action are then solutions of the Hamilton 
equations 
\begin{equation}
\label{eq:Hamilton} 
 \dtot{\ph_i}{t} = \dpar{H}{\psi_i}\;, \qquad 
 \dtot{\psi_i}{t} = -\dpar{H}{\ph_i}\;.
\end{equation} 
For the Lagrangian~\eqref{eq:Lagrangian_LDP}, the Hamiltonian reads 
\begin{equation}
\label{eq:Hamiltonian_LDP} 
 H(\ph,\psi) = \frac12 \pscal{\psi}{D(\ph)\psi} + \pscal{\psi}{f(\ph)}\;,
\end{equation} 
and the rate function takes the value 
\begin{equation}
 \cJ_{[0,T]}(\ph,\psi) = \frac12 \int_0^T 
\pscal{\psi(s)}{D(\ph(s))\psi(s)}\6s\;.
\end{equation} 
One should note that while the first term in the 
Hamiltonian~\eqref{eq:Hamiltonian_LDP} can be interpreted as a kinetic energy, 
the second term plays a different role from the usual potential energy in a 
particle system. 

The Hamilton equations~\eqref{eq:Hamilton} take the form 
\begin{equation}
 \dtot{\ph}{t} = D(\ph)\psi + f(\ph)\;, \qquad 
 \dtot{\psi}{t} = \frac12\nabla_\ph \pscal{\psi}{D(\ph)\psi}
 -\nabla_\ph \pscal{\psi}{f(\ph)}\;.
\end{equation} 

\begin{figure}
\begin{center}
\scalebox{1.0}{
\begin{tikzpicture}[>=stealth',main 
node/.style={draw,semithick,circle,fill=white,minimum size=4pt,inner 
sep=0pt},x=4cm,y=1.5cm]

\draw[->,thick] (-0.5,0) -> (2.4,0);
\draw[->,thick] (0,-1) -> (0,4.3);

\draw[thick,blue,->] (0,0) -> (0.5,1); 
\draw[thick,blue] (-0.5,-1) -- (2,4); 

\draw[thick,blue] (-0.2,0) -- (2,0);
\draw[thick,blue,->] (2,0) -> (0.5,0);
\draw[thick,blue,->] (-0.5,0) -> (-0.25,0);
\draw[thick,blue,->] (0,0) -> (-0.25,-0.5);

\draw[semithick, dashed] (2,0) -- (2,4);

\newcommand*{\psiplus}{2+sqrt(2)}
\newcommand*{\psiminus}{2-sqrt(2)}

\draw[blue,thick,-,smooth,domain={\psiminus}:{\psiplus},samples=25,/pgf/fpu,
/pgf/fpu/output format=fixed] plot ({1/\x + 0.5*\x}, \x);

\newcommand*{\psiplusA}{2+sqrt(3)}
\newcommand*{\psiminusA}{2-sqrt(3)}

\draw[blue,thick,-,smooth,domain={\psiminusA}:{\psiplusA},samples=25,/pgf/fpu,
/pgf/fpu/output format=fixed] plot ({0.5/\x + 0.5*\x}, \x);

\draw[blue,thick,-,smooth,domain=-0.5:{15/8},samples=25,/pgf/fpu,
/pgf/fpu/output format=fixed] plot (\x,{\x + sqrt(\x *\x + 1)});

\draw[blue,thick,-,smooth,domain=-0.5:1.5,samples=25,/pgf/fpu,
/pgf/fpu/output format=fixed] plot (\x,{\x + sqrt(\x *\x + 4)});

\newcommand*{\amp}{1/(2*sinh(5))}

\draw[purple,very thick,-,smooth,domain=0:5,samples=100,/pgf/fpu,
/pgf/fpu/output format=fixed] 
plot ({\amp *((2-exp(-5))*exp(\x) + (exp(5)-2)*exp(-\x))}, 
{2*\amp *(2-exp(-5))*exp(\x)});

\node[main node] at (1,0) {};
\node[main node] at (2,0) {};

\node[main node,purple,fill=white] at (2, {2*\amp *(2-exp(-5))*exp(5)}) {};

\node[] at (2.3,-0.2) {$\ph$};
\node[] at (-0.07,4) {$\psi$};

\node[] at (1,-0.2) {$x_0$};
\node[] at (2,-0.2) {$h$};

\end{tikzpicture}
}
\vspace{-3mm}
\end{center}
\caption[]{Hamiltonian flow corresponding to the Ornstein--Uhlenbeck process. 
The purple curve is the optimal trajectory allowing to reach level $h$ starting 
from $x_0$ in a given time. It corresponds to $\ph_T$ in \figref{fig:OU_LDP}.}
\label{fig:OU_Hamilton}
\end{figure}
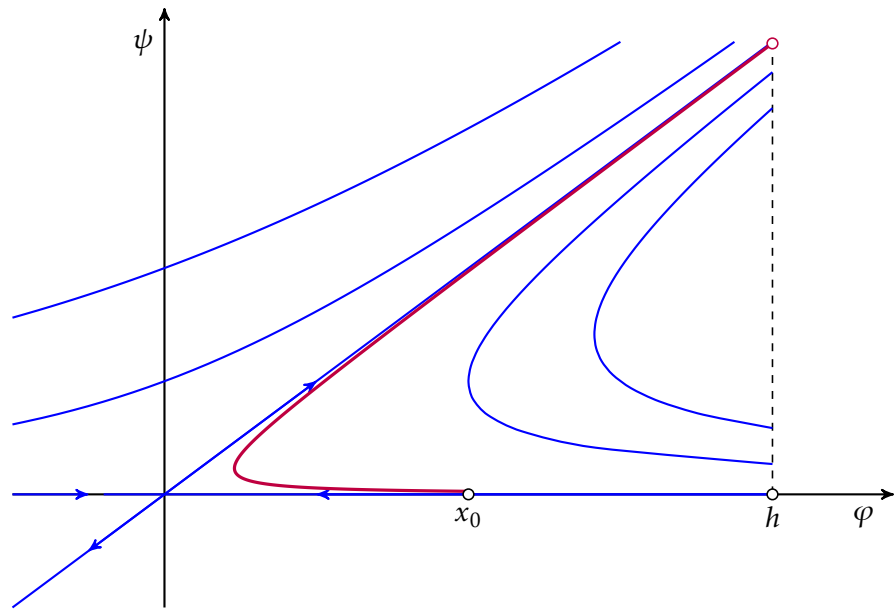

\begin{example}[Ornstein--Uhlenbeck process]
For the Ornstein--Uhlenbeck process~\eqref{eq:OU}, the Hamiltonian has the form 
\begin{equation}
 H(\ph,\psi) = \frac12 \psi^2 - \psi\phi\;,
\end{equation} 
and the Hamilton equations read 
\begin{equation}
  \dtot{\ph}{t} = \psi - \ph\;, \qquad 
 \dtot{\psi}{t} = \psi\;.
\end{equation} 
In this one-dimensional situation, one can take advantage of the fact that the 
Hamiltonian is a constant of motion, so that solutions of the Hamilton 
equations belong to level curves of $H$, see~\figref{fig:OU_Hamilton} (this 
remains true for more general processes). Note that as the time $T$ allowed to 
reach a given level $h$ becomes large, the optimal path becomes closer and 
closer to the stable and unstable manifolds of the origin, on which the 
Hamiltonian has value $0$. This also remains true in more generality. 
\end{example}


\subsection{The probability of exiting a domain in a time independent of $\eps$}
\label{ssec:ldp_exit1} 

Consider now a general diffusion of the form 
\begin{equation}
\label{eq:SDE_LDP_exit} 
 \6X^\eps_t = f(X^\eps_t) \6t + \sqrt{\eps} g(X^\eps_t) \6W_t\;, 
\end{equation} 
with the usual assumptions on $f$ and $g$ guaranteeing existence and uniqueness 
of a strong solution, and let $\cD\subset\R^n$ be a bounded, open set with 
smooth boundary. Then the large-deviation principle has the following 
consequence. 

\begin{proposition}
Fix a point $x_0\in\cD$, and define for any $y\notin\cD$ and $t>0$
\begin{equation}
 V(x_0,y; t)
 = \inf\Bigsetsuch{\cJ_{[0,T]}(\ph)}{\ph\in\cC([0,t],\R^n), \ph(0)=x_0, 
\ph(t)=y}\;.
\end{equation}
Let $\tau^\eps$ denote the first-exit time from $\cD$ of the solution 
of~\eqref{eq:SDE_LDP_exit} starting at $x_0$. Then 
\begin{equation}
\label{eq:first_exit_V} 
 \lim_{\eps\to0} \eps\log\bigprobin{x_0}{\tau^\eps \leqs T} 
 = -\inf\Bigsetsuch{V(x_0,y; t)}{t\in[0,T], y\not\in\cD}
\end{equation} 
holds for any $T>0$. 
\end{proposition}

The proof proceeds in the same way as in the examples we have seen so far, the 
main point being to show that the upper and lower bound in the large-deviation 
principle coincide. 

\begin{figure}
\begin{center}
\vspace{-14mm}
\begin{tikzpicture}[>=stealth',main node/.style={draw,circle,fill=white,minimum
size=3pt,inner sep=0pt},scale=1.1
]



\path[fill=blue!10] (-4,0) .. controls (-4,5) and (4,3) .. (4,1) -- 
(-4,0) .. controls (-4,-4) and (4,-2) .. (4,1);

\draw[thick,blue] (-4,0) .. controls (-4,5) and (4,3) .. (4,1);
\draw[thick,blue] (-4,0) .. controls (-4,-4) and (4,-2) .. (4,1);


\pgfmathsetseed{1}	
\draw[red] plot [domain=-0.5:4.5,samples=80,smooth]
(\x+0.12*rand,{0.5*\x + 0.5*sin(90*\x) +0.12*rand});


\draw[semithick,blue,->] (2,2.72) -- (2.3,3.73);
\draw[semithick,->] (2,2.72) -- (0.5,2);

\node[blue] at (2.6,3.35) {$n(y)$};
\node[] at (1.5,2.1) {$f(y)$};


\node[main node,thick] at (-0.45,-0.6) {}; 
\node[] at (-0.75,-0.6) {$x_0$};

\node[main node,thick] at (0,0.6) {}; 
\node[] at (-0.3,0.7) {$x^*$};

\node[main node,thick] at (3.75,1.65) {}; 
\node[] at (4.2,1.6) {$X^\eps_{\tau^\eps}$};

\node[main node,thick] at (2,2.72) {}; 
\node[] at (2.3,2.8) {$y$};

\node[blue] at (3.0,-1.3) {$\partial\cD$};

\node[blue] at (-2.5,2.0) {$\cD$};

\end{tikzpicture}
\vspace{-16mm}
\end{center}
\caption[]{Exit of a diffusion $(X^\eps_t)_{t\geqs0}$ from a domain $\cD$. The 
vector $n(y)$ is the unit outward normal vector at $y\in\partial\cD$, and the 
vector field $f(y)$ should satisfy $\pscal{f(y)}{n(y)} < 0$ at any 
$y\in\partial\cD$.}
\label{fig:exit_problem}
\end{figure}

The main difficulty at this stage remains solving the variational problem. In 
general this can be quite intricate, but one can say more under some additional 
assumptions. We assume from now on (see \figref{fig:exit_problem}) that 
\begin{itemize}
\item 	the deterministic equation $\dot x = f(x)$ has a unique equilibrium 
point in $\cD$, say at $x^* = 0$;
\item 	the closure $\overline\cD$ is contained in the basin of attraction of 
$x^*$ for the deterministic evolution, meaning that for any initial 
condition $x_0\in\overline\cD$, the solution of $\dot x = f(x)$ 
starting from $x_0$ converges to $x^*$ as time goes to infinity;
\item 	the vector field $f$ points inward on the boundary $\cD$ of the domain, 
meaning that if $n(y)$ denotes the unit outward normal vector at $y$, then  
$\pscal{f(y)}{n(y)} < 0$ for any $y\in\partial\cD$ (the boundary 
$\partial\cD$ is called non-characteristic). 
\end{itemize}

\begin{definition}[Quasipotential]
\label{def:quasipotential} 
The quasipotential is the function defined for any $y\notin\cD$ by 
\begin{equation}
\label{eq:quasipotential} 
 V(x^*,y) = \inf_{t>0} V(x^*,y;t)\;.
\end{equation} 
\end{definition}

\begin{example}[Gradient system]
\label{ex:gradient_SDE_LDP} 
Consider a gradient SDE of the form 
\begin{equation}
\label{eq:SDE_gradient_LDP} 
 \6X_t = -\nabla U(X_t)\6t + \sqrt{\eps} \6W_t\;,
\end{equation} 
where $U$ is a potential that admits a unique minimum in $\cD$ at $x^*$. 
For any $T>0$, we have 
\begin{align}
\cJ_{[0,T]}(\ph) 
&= \frac12 \int_0^T \norm{\dot\ph(s) + \nabla U(\ph(s))}^2 \6s \\
&= \frac12 \int_0^T \norm{\dot\ph(s) - \nabla U(\ph(s))}^2 \6s 
+ 2 \int_0^T \pscal{\dot\ph(s)}{\nabla U(\ph(s))}\6s \\
&= \frac12 \int_0^T \norm{\dot\ph(s) - \nabla U(\ph(s))}^2 \6s 
+ 2 \bigbrak{U(\ph(T)) - U(\ph(0))}\;.
\end{align}
Taking $\ph(0) = x^*$ and $\ph(T) = y \notin \cD$, we see that $V(x^*,y; T)$ is 
bounded below by $2[U(y) - U(x^*)]$. Our assumptions imply that we can make the 
integral on the right-hand side arbitrarily small as $T\to\infty$, by choosing 
for $\ph$ a path connecting $y$ and $x^*$ that follows the reversed 
deterministic dynamics $\dot x = \nabla U(x)$. It follows that the 
quasipotential is simply given in this case by 
\begin{equation}
\label{eq:quasipotential_reversible} 
 V(x^*,y) = 2\bigbrak{U(y) - U(x^*)}\;.
\end{equation} 
\end{example}

\begin{proposition}
\label{prop:taueps_V} 
Under the above assumptions, we have for any initial 
condition $x_0\in\cD$
\begin{equation}
 \lim_{T\to\infty}
 \lim_{\eps\to0} \eps\log\bigprobin{x_0}{\tau^\eps \leqs T} 
 = - \overline V\;,
\end{equation} 
where $\overline{V}$ is the infimum of the quasipotential on the boundary: 
\begin{equation}
 \overline V = \inf_{y\in\partial\cD} V(x^*,y)\;. 
\end{equation}
\end{proposition}
\begin{proof}
We will prove a slightly more quantitative result, which will be useful later 
on. 
\begin{enumerate}
\item 	The first step is to show that it is costly to remain outside a 
neighbourhood of $x^*$ for long. Indeed, let $\cB_\delta(x^*)$ be the ball of 
radius $\delta>0$ around $x^*$. One can show that there exist constants $c, T_0 
>0$ such that, for any path remaining in $\cD\setminus\cB_\delta(x^*)$ during a 
time interval $[0,T]$, one has $\cJ_{[0,T]}(\ph) \geqs c(T-T_0)$. This is due to 
the fact that $\norm{f}$ is bounded away from zero in 
$\cD\setminus\cB_\delta(x^*)$, and that the only way to keep the rate function 
small is to follow the deterministic dynamics that leads towards $x^*$ 
(cf.~\cite[Chapter~4, Lemma~2.2]{FW}). Therefore, we have for any $\delta>0$ 
\begin{equation}
\label{eq:proof_exit_annulus} 
 \lim_{T\to\infty} \limsup_{\eps\to0} 
 \eps\log \biggbrak{\sup_{x_0\in\cD} 
\Bigprobin{x_0}{X^\eps_t\in\cD\setminus\cB_\delta(x^*) \; \forall t\in[0,T]}}
= -\infty\;.
\end{equation} 
This justifies restricting the argument to starting points $x_0$ in the ball 
$\cB_\delta(x^*)$ of radius $\delta$ centred in $x^*$. 

\item 	The second step consists in showing that for any $\eta > 0$, there 
exists $\delta_0 > 0$ such that, for any $\delta < \delta_0$, there is a $T_0 > 
0$ such that for any $T\geqs T_0$, one has
\begin{equation}
\label{eq:proof_exit_lowerbound} 
 \liminf_{\eps\to0} \eps 
 \log \biggbrak{\inf_{x_0\in\cB_\delta(x^*)} 
 \bigprobin{x_0}{\tau^\eps \leqs T}}
 > -(\overline V + \eta)\;.
\end{equation} 
In view of the lower bound~\eqref{eq:LDP_lower_bound} in the large-deviation 
principle, it is sufficient to construct, for any $x_0\in\cB_\delta(x^*)$, a 
particular continuous path $\ph^*:[0,T]\to\R^n$ such that $\ph^*(0) = x_0$, 
$\ph^*(T)\in\partial\cD$, and $\cJ_{[0,T]}(\ph^*) \leqs \overline V + 
\eta$. To construct this path, let $y^*$ be a point on $\partial\cD$ where the 
quasipotential reaches its minimum (which exists by a compactness argument), 
and let $\ph_\infty$ be a path minimising the quasipotential in potentially 
infinite time. This path intersects the boundary of $\cB_\delta(x^*)$ at a point 
$x_1$ (\figref{fig:exit_path}). We may assume that $\ph_\infty$ is parametrised 
in such a way that $\ph_\infty(0) = y^*$ and $\ph_{\infty}(-T_1) = x_1$ for 
some $T_1 = T_1(\delta) > 0$. Then define for $T > T_1$ 
\begin{equation}
\ph^*(t) = 
\begin{cases}
x_0 + \dfrac{t}{T-T_1} (x_1 - x_0)
& \text{for $0 \leqs t \leqs T-T_1$\;,} \\
\ph_\infty(t-T)
& \text{for $T-T_1 \leqs t \leqs T$\;.}
\end{cases}
\end{equation}
The contribution of the first part of $\ph^*$ to the rate function goes to $0$ 
with $\delta$, while the contribution of the second part is smaller than 
$\overline V$. Therefore, for any $\eta>0$, one has indeed 
$\cJ_{[0,T]}(\ph^*) = \overline V + \eta$ if $\delta$ is small enough, as 
required. 

\begin{figure}
\begin{center}
\vspace{-18mm}
\begin{tikzpicture}[>=stealth',main node/.style={draw,circle,fill=white,minimum
size=3pt,inner sep=0pt},scale=1.1
]



\path[fill=blue!10] (-4,0) .. controls (-4,5) and (4,3) .. (4,1) -- 
(-4,0) .. controls (-4,-4) and (4,-2) .. (4,1);

\draw[thick,blue] (-4,0) .. controls (-4,5) and (4,3) .. (4,1);
\draw[thick,blue] (-4,0) .. controls (-4,-4) and (4,-2) .. (4,1);

\draw[semithick,violet,fill=violet!10] (0,0.6) circle (0.8);


\draw[semithick,orange] (0.4,0.2) -- (0.65,1.1);

\pgfmathsetseed{1}	
\draw[purple, thick] plot [domain=0:3.3,samples=80,smooth]
(\x,{0.6 + 0.5*\x + 0.2*sin(90*\x)});


\node[main node,thick] at (0.4,0.2) {}; 
\node[] at (0.1,0.2) {$x_0$};

\node[main node,thick] at (0,0.6) {}; 
\node[] at (-0.25,0.7) {$x^*$};

\node[main node,thick] at (3.28,{0.6 + 0.5*3.28 + 0.2*sin(90*3.28)}) {}; 
\node[] at (3.5,2.2) {$y^*$};

\node[main node,orange,thick,fill=white] at (0.64,1.08) {}; 
\node[orange] at (0.65,1.4) {$x_1$};

\node[blue] at (3.0,-1.3) {$\partial\cD$};
\node[blue] at (-2.5,2.0) {$\cD$};

\node[purple] at (1.8,1.3) {$\ph_\infty$};
\node[violet] at (-1,-0.3) {$\cB_\delta(x^*)$};

\node[orange]  at (1,1) {$\ph^*$};

\end{tikzpicture}
\vspace{-20mm}
\end{center}
\caption[]{Construction of the path $\ph^*$ in the proof of 
Proposition~\ref{prop:taueps_V}. The path is obtained by taking a 
path $\ph_\infty$ minimising the rate function, removing its part inside a ball 
$\cB_\delta(x^*)$, and adding a linear path connecting it to $x_0$.}
\label{fig:exit_path}
\end{figure}
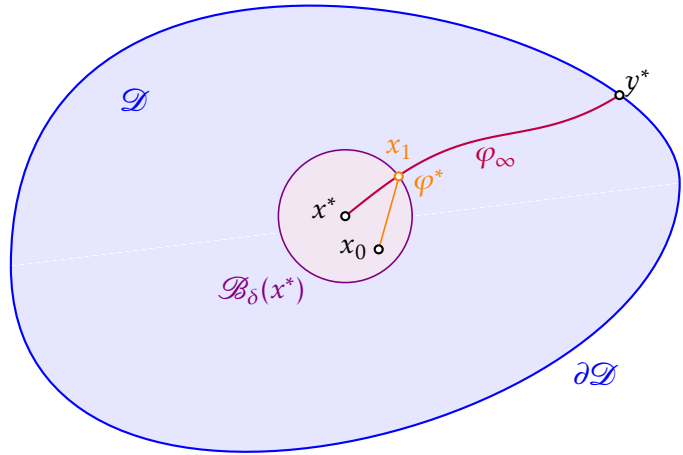

\item 	Next we argue that for any $\eta > 0$, there exists $\delta_0 > 0$ such 
that, for any $\delta < \delta_0$, there is a $T_0 > 0$ such that  
\begin{equation}
\label{eq:proof_exit_upperbound} 
 \limsup_{\eps\to0} \eps 
 \log \biggbrak{\sup_{x_0\in\cB_\delta(x^*)} 
 \bigprobin{x_0}{\tau^\eps \leqs T}}
 < -(\overline V - \eta)
\end{equation} 
holds for any $T\geqs T_0$. 
Indeed, assume that this is not the case. Then by the upper 
bound \eqref{eq:LPD_upper_bound} of the large-deviation principle, there would 
exist $x_0\in\cB_\delta(x^*)$ and a continuous path $\ph:[0,T]\to\R^n$ such 
that $\ph(0) = x_0$ and $\ph(t)\notin\cD$ for some $t\in[0,T]$, satisfying 
$\cJ_{[0,T]}(\ph) \leqs \overline V - \eta$. Adding a linear piece from $x^*$ 
to $x_0$ to this path, one would obtain a path with rate function smaller than 
$\overline V - \frac{\eta}{2}$ is $\delta$ is small enough, contradicting the 
definition of $\overline V$.
\end{enumerate}
Since $\eta>0$ was arbitrary, combining these three steps yields the claimed 
result.
%
\end{proof}

\begin{remark}
For the Ornstein--Uhlenbeck process, we can take $V(x^*,y)$ given 
by~\eqref{eq:quasipotential_reversible} with $U(x) = \frac12 
x^2$, so that we recover the result~\eqref{eq:OU_limit}. 
\end{remark}


\subsection{Mean exit times and exit locations}
\label{ssec:ldp_exit2} 

So far, we have described the exit problem on timescales independent of $\eps$. 
This is not quite satisfactory, however, since really interesting new phenomena 
appear in these systems for times that are exponentially large in $\eps$. One 
of the classical results in this context is the following one. 

\begin{theorem}[Exit from the neighbourhood of a stable equilibrium point]
\label{thm:exit} 
Let $\cD\subset\R^n$ be a bounded, open set with smooth boundary. 
Assume that the deterministic equation $\dot x = f(x)$ has a unique equilibrium 
point $x^* \in \cD$, that the closure $\overline\cD$ is contained in the 
deterministic basin of attraction of $x^*$, and that $\pscal{f(y)}{n(y)} < 0$ 
for any $y\in\partial\cD$, where $n(y)$ denotes the unit normal vector to the 
boundary at $y$. Then the following properties hold for any initial condition 
$x_0\in\cD$.
\begin{enumerate}
\item 	For any $\eta>0$, one has 
\begin{equation}
\lim_{\eps\to 0} \Bigprobin{x_0}{\e^{(\bar V - \eta)/\eps} < \tau^\eps 
< \e^{(\bar V + \eta)/\eps}}
= 1\;.
\end{equation} 
\item 	The mean first-exit time satisfies 
\begin{equation}
\label{eq:exit_domain_expect} 
 \lim_{\eps\to0} \eps\log\bigexpecin{x_0}{\tau^\eps} = \overline{V}\;.
\end{equation} 
\item 	For any closed subset $N\subset\partial\cD$ satisfying $\inf_{y\in 
N}V(x^*,y) > \overline V$, one has 
\begin{equation}
 \lim_{\eps\to 0} \bigprobin{x_0}{X^\eps_{\tau^\eps} \in N} = 0\;.
\end{equation} 
\end{enumerate}
\end{theorem}

Let us briefly comment on the meaning of these assertions.
\begin{enumerate}
\item 	The first claim means that $\tau^\eps$ concentrates around 
$\e^{\overline{V}/\eps}$, albeit in a rather weak sense, since there remain 
exponentially large windows on either side of this value.
\item 	The second claim says that the mean first-exit time behaves like 
$\e^{\overline{V}/\eps}$ in the sense of logarithmic equivalence, a property 
known as Arrhenius' law. 
\item 	The third claim says that the diffusion is likely to exit $\cD$ near 
those points on its boundary where the quasipotential is the smallest. 
\end{enumerate}

\begin{proof}[\Sketch]
We will give a proof of~\eqref{eq:exit_domain_expect}, which is the 
most intuitive part of the result. It is based on the idea that the 
probabilities to exit during time intervals $[mT,(m+1)T]$ follow 
approximately a  geometric law with probability of success 
$\probin{x^*}{\tau^\eps < T}$, which has expectation $(\probin{x^*}{\tau^\eps < 
T})^{-1}$. 

\begin{figure}
\begin{center}
\vspace{-18mm}
\begin{tikzpicture}[>=stealth',main node/.style={draw,circle,fill=white,minimum
size=3pt,inner sep=0pt},scale=1.1
]


\path[fill=blue!10] (-4,0) .. controls (-4,5) and (4,3) .. (4,1) -- 
(-4,0) .. controls (-4,-4) and (4,-2) .. (4,1);

\draw[thick,blue] (-4,0) .. controls (-4,5) and (4,3) .. (4,1);
\draw[thick,blue] (-4,0) .. controls (-4,-4) and (4,-2) .. (4,1);

\draw[semithick,violet,fill=violet!10] (0,0.6) circle (1);
\draw[semithick,teal,fill=teal!10] (0,0.6) circle (0.56);


\pgfmathsetseed{1}	
\draw[red, thick] (0.4,0.2) -- plot [domain=0:2.25,samples=80,smooth]
({0.4 + (\x + 0.8*sin(200*\x))*cos(180*\x) +0.12*rand},{0.2 + (\x + 
0.8*sin(200*\x))*sin(180*\x)+0.12*rand});


\node[main node,thick] at (0.4,0.2) {}; 
\node[] at (0.15,0.3) {$x_0$};

\node[main node,thick] at (0,0.6) {}; 
\node[] at (-0.1,0.75) {$x^*$};

\node[main node,thick] at (2.7,2.4) {}; 
\node[] at (3,2.6) {$X^\eps_{\tau^\eps}$};

\node[main node,thick] at (1,0.6) {}; 
\node[] at (1.4,0.55) {$X^\eps_{\sigma_0}$};

\draw[semithick] (-0.85,1.75) -- (-0.25,1.1);
\node[main node,thick] at (-0.25,1.1) {}; 
\node[] at (-1.15,1.85) {$X^\eps_{\tau_1}$};

\node[main node,thick] at (0.3,-0.35) {}; 
\node[] at (0.7,-0.5) {$X^\eps_{\sigma_1}$};

\node[blue] at (3.0,-1.3) {$\partial\cD$};
\node[blue] at (-2.5,2.0) {$\cD$};

\node[teal] at (-1.9,0.3) {$\cB_{\delta/2}(x^*)$};
\draw[semithick,teal] (-1.2,0.3) -- (-0.55,0.5);
\node[violet] at (-1,-0.5) {$\cB_\delta(x^*)$};


\end{tikzpicture}
\vspace{-20mm}
\end{center}
\caption[]{Construction of the Markov chain $(Z_m)_{m\geqs0} = 
(X^\eps_{\tau_m})_{m\geqs0}$ used in the proof of Theorem~\ref{thm:exit}.}
\label{fig:exit_expect}
\end{figure}

Fix some $T>0$. For any $x_0\in\cD$ and any $m\in\N_0$, we have
\begin{align}
\bigprobin{x_0}{\tau^\eps > (m+1)T}
&= \Bigexpecin{x_0}{\indexfct{\tau^\eps > mT} 
\bigprobin{X^\eps_{mT}}{\tau^\eps > T}} \\
&\leqs \biggpar{\sup_{x_1\in\cD} \bigprobin{x_1}{\tau^\eps > T}} \; 
\bigprobin{x_0}{\tau^\eps > mT}\;.
\end{align}
It thus follows by induction that for all $m\geqs0$, 
\begin{equation}
 \bigprobin{x_0}{\tau^\eps > mT} 
 \leqs p^m\;, \qquad 
 \text{where }p = \sup_{x_1\in\cD} \bigprobin{x_1}{\tau^\eps > T}\;.
\end{equation} 
Therefore, integration by parts shows that 
\begin{align}
\bigexpecin{x_0}{\tau^\eps}
= \int_0^\infty \bigprobin{x_0}{\tau^\eps > t}\6t 
\leqs T \sum_{m=0}^\infty \bigprobin{x_0}{\tau^\eps > mT} 
\leqs T \sum_{m=0}^\infty p^m
= \frac{T}{1-p}\;.
\end{align}
Note that $p$ is the complement of the probability described in 
Proposition~\ref{prop:taueps_V}. In particular, \eqref{eq:proof_exit_lowerbound} 
shows that for any $\eta>0$, there exists $\eps_0>0$ such that for any $\eps < 
\eps_0$, one has 
\begin{equation}
 \eps\log(1 - p) \geqs -(\overline V + \eta)
\end{equation} 
if $\delta$ is small enough. Since $\eta > 0$ is arbitrary, this proves  
the upper bound for~\eqref{eq:exit_domain_expect}. 

This argument does not work directly for the lower bound on the expectation, 
because $\bigprobin{x_1}{\tau^\eps > T}$ is not bounded below uniformly in 
$\cD$. Instead, we introduce a discrete-time Markov chain defined as follows. 
Set $\tau_0 = 0$, and define an increasing sequence of stopping times $\tau_0 < 
\sigma_1 < \tau_1 < \dots$ by 
\begin{align}
 \sigma_m &= \inf\bigsetsuch{t > \tau_m}{X^\eps_t \in 
\partial\cB_\delta(x^*)}\;, \\
 \tau_m &= \inf\bigsetsuch{t > \sigma_{m-1}}{X^\eps_t \in 
\partial\cB_{\delta/2}(x^*)\cup\partial\cD}
\end{align}
(\figref{fig:exit_expect}). 
The process $(Z_m)_{m\geqs0}$ defined by $Z_m = X^\eps_{\tau_m}$ is a 
discrete-time Markov chain on the set 
$\partial\cB_{\delta/2}(x^*)\cup\partial\cD$, and we have $\tau^\eps = 
\tau_\nu$, where $\nu = \inf\setsuch{m\geqs1}{Z_m\in\partial\cD}$. 
Now we note that for $\norm{x_0 -x^*} \leqs \frac{\delta}{2}$, one has 
\begin{align}
\bigexpecin{x_0}{\tau^\eps}
&= \biggexpecin{x_0}{\sum_{m=1}^{\nu-1} (\tau_{m+1} - \tau_m)} 
\geqs \biggexpecin{x_0}{\sum_{m=1}^{\nu-1} (\tau_{m+1} - \sigma_m)} \\
&\geqs \bigexpecin{x_0}{\nu}\,
\inf_{x\in\partial\cB_\delta(x^*)}\bigexpecin{x}{\tau_1} \;.
\end{align}
The infimum is bounded below by a constant $T_1$ independent of $\eps$ thanks 
to~\eqref{eq:proof_exit_annulus}. To estimate $\bigexpecin{x_0}{\nu}$, we note 
that for any $T>0$, 
\begin{align}
\bigprobin{x_0}{\nu = 1} 
&= \bigprobin{x_0}{\tau_1 = \tau^\eps} \\
&= \bigprobin{x_0}{\tau_1 = \tau^\eps, \tau_1 < T}
+ \bigprobin{x_0}{\tau_1 = \tau^\eps, \tau_1 \geqs T}\;.
\end{align}
By~\eqref{eq:proof_exit_upperbound}, for any $\eta > 0$ the first term on the 
right-hand side can be bounded above by $\e^{-(\overline V-\eta)/\eps}$ if 
$\delta$ is small enough and $T$ is large enough. The second term can be made 
smaller than the first one by taking $T$ large enough, again thanks 
to~\eqref{eq:proof_exit_annulus}. It follows that 
\begin{equation}
 \bigexpecin{x_0}{\tau^\eps} 
 \geqs T_1 \sum_{m=0}^\infty \biggpar{1 - 
\sup_{x_0\in\cB_{\delta/2}(x^*)} \bigprobin{x_0}{\nu = 1}}^m 
 \leqs T_1 \e^{(\overline V-\eta)/\eps}\;,
\end{equation} 
proving the lower bound for starting points in $\cB_{\delta/2}(x^*)$. 
The result can be extended to general starting points in $\cD$ by using the 
lower bound 
\begin{equation}
 \bigexpecin{x_0}{\tau^\eps} \geqs 
\bigexpecin{x_0}{\indexfct{\tau_\eps > 
\tau_1}\bigexpecin{X^\eps_{\tau_1}}{\tau^\eps}}
\geqs \bigprobin{x_0}{\tau_\eps > \tau_1} 
\inf_{x_1\in\partial\cB_\delta(x^*)} \bigexpecin{x_1}{\tau^\eps}
\end{equation}
and the fact that $\bigprobin{x_0}{\tau_\eps > \tau_1}$ has order $1$. 
\end{proof}

\begin{figure}
\begin{center}
\begin{tikzpicture}[>=stealth',main node/.style={draw,circle,fill=white,minimum
size=2.5pt,inner sep=0pt},scale=0.65,
]


\path[ultra thin,fill=teal!20] (-1,4) .. controls (-1,3) and (-0.3,1) .. (0,0)
.. controls (0.3,-1) and (1,-3) .. (1,-4) -- (-7,-4) -- (-7,4);
\path[fill=green!15] (-1,4) .. controls (-1,3) and (-0.3,1) .. (0,0)
.. controls (0.3,-1) and (1,-3) .. (1,-4) -- (9,-4) -- (9,4);



\draw[thick] (-1,4) .. controls (-1,3) and (-0.3,1) .. (0,0)
\arrowonline{0.5};
\draw[thick] (1,-4) .. controls (1,-3) and (0.3,-1) .. (0,0)
\arrowonlineback{0.5};


\draw[green!50!black,thick] plot[smooth cycle,tension=1]
  coordinates{(-5,0) (-2,0) (-2.2,-2) (-5.2,-2)};


  

\draw[semithick] (0,0) .. controls (-1,-3) and (-4,-2.85) .. (-5,-2.25) 
\arrowonlinebacksemi{0.4};


\draw[semithick] (0,0) .. controls (1,3) and (3,3) .. (5,2.5)
\arrowonlinesemi{0.4};


\draw[semithick] (1.1,-4) .. controls (0.8,-2.1) and (2,2.5) .. (5,2.5)
\arrowonlinesemi{0.4};


\draw[semithick,->] (-3.5,-1) .. controls (-4,0) and (-2.8,0.4) .. (-2,-0.2);
\draw[semithick,->] (-3.5,-1) .. controls (-3,-2) and (-4.4,-2.4) ..
(-5.2,-1.8);




\node[main node] at (0,0) {};
\node[main node] at (-3.5,-1) {};

\node[main node,green!50!black,thick,fill=white] at (5,2.5) {};


\node[] at (-4.3,0.7) {$\cP$};
\node[] at (5.5,2.7) {$x^*$};
\node[] at (0.5,-0.1) {$z^*$};

\end{tikzpicture}
\vspace{-2mm}
\end{center}
\caption[]{Example of a two-dimensional vector field with two attractors (an
equilibrium point $x^*$ and a stable periodic orbit $\cP$). Their basins of 
attraction are shaded in different colours. When weak noise is added to the 
system, solutions spend long times in the neighourhood of attractors, and make
occasional transitions between attractors.  
}
\label{fig:attractors}
\end{figure}
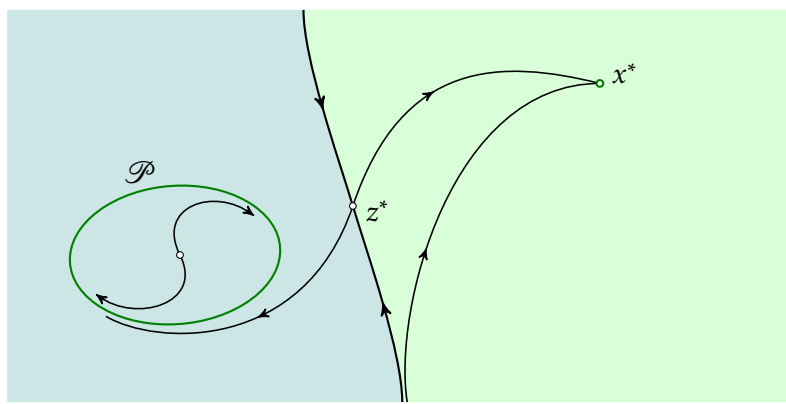

One major limitation of Theorem~\ref{thm:exit} is the condition that orbits 
starting on $\partial\cD$ are attracted by the equilibrium point $x^*$ (the 
condition on the vector field pointing inward on the boundary of $\cD$ is 
actually not essential for this result to hold). It is of great interest to 
extend this result to situations with multiple attractors, as illustrated in 
\figref{fig:attractors}. Theorem~\ref{thm:exit} cannot be applied, for 
instance, to describe transitions from the stable equilibrium $x^*$ to the 
stable periodic orbit $\cP$, or the other way, because this would mean 
leaving a set having the saddle point $z^*$ on its boundary. 

There exists a theory, discussed in~\cite[Chapter~6]{FW}, dealing with 
such situations in great generality, by approximating the dynamics with a 
suitable Markov chain. This theory is quite involved, and here we will only 
outline the argument in a relatively simple bistable system, as shown 
in~\figref{fig:attractors}. In particular, one would like to show that the 
first-hitting time $\tau^\eps$ of a neighbourhood of $\cP$ satisfies 
\begin{align}
\label{eq:expec_exit_bistable} 
 \lim_{\eps\to0} \eps\log\bigexpecin{x^*}{\tau^\eps} 
 &= \overline{V} \\
 &= \inf_{T>0} \, \inf\Bigsetsuch{\cJ_{[0,T]}(\ph)}{\ph\in\cC([0,t],\R^n), 
\ph(0)=x^*, \ph(t)=z^*}\;.
\end{align} 
An estimate that turns out to be useful to show such a result is the following 
\lq\lq three-set argument\rq\rq\ from~\cite[Corrollary~5.8]{BG12a}.
For sets $A, B, C\subset \R^n$, we denote by 
\begin{equation}
 \tau^+_A = \inf\setsuch{t>0}{X^\eps_t\in A}
\end{equation} 
the first-hitting time of $A$, and we write
\begin{equation}
 \bigprobin{A}{\tau^+_B < \tau^+_C} 
 = \sup_{x\in A} \bigprobin{x}{\tau^+_B < \tau^+_C}\;, \qquad 
 \bigexpecin{A}{\tau^+_B} 
 = \sup_{x\in A} \bigexpecin{x}{\tau^+_B}\;.  
\end{equation} 
Then we have the following simple estimate. 

\begin{lemma}[Three-set argument]
Let $A,B,C\subset\R^n$ be pairwise disjoint sets with smooth boundary. 
If $\bigprobin{A}{\tau^+_C < \tau^+_B} < 1$, then 
\begin{equation}
 \bigexpecin{A}{\tau^+_B} 
 \leqs \frac{\bigexpecin{A}{\tau^+_{B\cup C}} + 
 \bigprobin{A}{\tau^+_C < \tau^+_B} \bigexpecin{\partial C}{\tau^+_{A\cup B}}}
 {1 - \bigprobin{A}{\tau^+_C < \tau^+_B}}\;.
\end{equation} 
\end{lemma}
\begin{proof}
For any $x\notin(B\cup C)$, we have 
\begin{align}
 \bigexpecin{x}{\tau^+_B} 
 &= \bigexpecin{x}{\indexfct{\tau^+_B < \tau^+_C}\tau^+_B} + 
 \bigexpecin{x}{\indexfct{\tau^+_C < \tau^+_B}\tau^+_B} \\
 &= \bigexpecin{x}{\indexfct{\tau^+_B < \tau^+_C}\tau^+_B}
 + \Bigexpecin{x}{\indexfct{\tau^+_C < \tau^+_B}\bigpar{\tau^+_C + 
\bigexpecin{X^\eps_{\tau^+_C}}{\tau^+_B}}} \\
 &= \bigexpecin{x}{\tau^+_{B\cup C}} 
 + \Bigexpecin{x}{\indexfct{\tau^+_C < \tau^+_B} 
\bigexpecin{X^\eps_{\tau^+_C}}{\tau^+_B}}\;,
\end{align}
because $\tau^+_B\wedge\tau^+_C = \tau^+_{B\cup C}$. Since sample paths are 
continuous, we have $X^\eps_{\tau^+_C} \in \partial C$, so that   
\begin{equation}
\label{eq:threeset1} 
 \bigexpecin{A}{\tau^+_B}
 \leqs \bigexpecin{A}{\tau^+_{B\cup C}}
 + \bigprobin{A}{\tau^+_C < \tau^+_B} \bigexpecin{\partial C}{\tau^+_B}\;.
\end{equation} 
Exchanging the roles of $A$ and $\partial C$, we get 
\begin{equation}
\label{eq:threeset2} 
 \bigexpecin{\partial C}{\tau^+_B}
 \leqs \bigexpecin{\partial C}{\tau^+_{A\cup B}}
 + \bigprobin{\partial C}{\tau^+_A < \tau^+_B} \bigexpecin{A}{\tau^+_B}
 \leqs \bigexpecin{\partial C}{\tau^+_{A\cup B}}
 + \bigexpecin{A}{\tau^+_B}\;.
\end{equation} 
Substituting~\eqref{eq:threeset2} in~\eqref{eq:threeset1} and solving for 
$\bigexpecin{A}{\tau^+_B}$ gives the claimed result. 
\end{proof}

Let us outline how this result is useful to prove the upper bound 
in~\eqref{eq:expec_exit_bistable}. Let $A$ and $B$ be open sets containing 
respectively $x^*$ and $\cP$, and let $C$ be the complement of a sufficiently 
large bounded set, containing both $A$ and $B$. We further assume that there 
are no attractors for the deterministic dynamics in $C$. Then the 
following holds. 
\begin{enumerate}
\item 	The probability $\bigprobin{A}{\tau^+_C < \tau^+_B}$ can be shown to be 
smaller than, say, $\frac12$, if the quasipotential on $\partial C$ is large 
enough. This follows from an argument of the same type as in the proof of 
Theorem~\ref{thm:exit} (see~\cite[Proposition~5.14]{BG12a}). 
\item	The expectation $\bigexpecin{\partial C}{\tau^+_{A\cup B}}$ can be 
shown to be at most of order $\log(\eps^{-1})$, by using the fact that sample 
paths are likely to remain close to the deterministic orbit starting at the 
same point. The logarithmic behaviour is due to the time possibly spent near 
the saddle point. 
\item 	It follows that $\bigexpecin{A}{\tau^+_B}$ behaves like 
\begin{equation}
 \bigexpecin{A}{\tau^+_{B\cup C}} 
 = \bigexpecin{A}{\tau_{(B\cup C)^c}}
 \leqs \bigexpecin{C^c}{\tau^+_{(B\cup C)^c}}\;,
\end{equation} 
which can be estimated by a Markov chain agument, as in the proof of 
Theorem~\ref{thm:exit}. 
\end{enumerate}

\begin{remark}
The assumption that $C$ does not contain any attractor is important to control 
the value of $\bigexpecin{\partial C}{\tau^+_{A\cup B}}$. Otherwise, there are 
cases where $C$ contains an attractor that is much more stable than those in 
$A$ and $B$, completely dominating the expectation of $\tau^+_{A\cup B}$ and 
$\tau^+_B$. 
\end{remark}

\begin{remark}[Eyring--Kramers law]
In the gradient case 
\begin{equation}
 \6X_t = -\nabla U(X_t)\6t + \sqrt{\eps} \6W_t\;,
\end{equation} 
the bistable situation occurs when $U$ has exactly two local minima $x^*$ and 
$y^*$, separated by a saddle point $z^*$. As pointed out in 
Example~\ref{ex:gradient_SDE_LDP}, the mean transition time from $x^*$ to $y^*$ 
obeys the Arrhenius law 
\begin{equation}
 \lim_{\eps\to0} \eps\log \bigexpecin{x^*}{\tau^+_{\cB_\delta(y^*)}}
 = 2 \bigbrak{U(z^*) - U(x^*)}\;.
\end{equation} 
However, using more sophisticated techniques, such as potential 
theoretic-methods relying on the links between PDEs and exit times seen in 
Chapter~\ref{chap:SDE_PDE}, one can prove the so-called Eyring--Kramers law 
\begin{equation}
 \bigexpecin{x^*}{\tau^+_{\cB_\delta(y^*)}} 
 = \frac{2\pi}{\abs{\lambda_-(z^*)}}
 \sqrt{\frac{\abs{\det\Hess U(z^*)}}{\det\Hess U(x^*)}}
 \e^{2 [U(z^*) - U(x^*)]/\eps} 
 \bigbrak{1 + \Order{\eps}}\;,
\end{equation} 
where $\Hess U(x)$ denotes the Hessian matrix of $U$ at $x$, and 
$\lambda_-(z^*)$ denotes the unique negative eigenvalue of $\Hess U(z^*)$. 
See for instance~\cite{Berglund_irs_MPRF} for an overview, 
and~\cite{Bovier_denHollander_book} for a comprehensive account of the 
potential-theoretic approach allowing to prove this result. 
See~\cite{Bouchet_Reygner_16,Landim_Mariani_Seo_17} for some extensions to 
non-gradient cases. 
\end{remark}


\cleardoublepage
\phantomsection
\bibliographystyle{alpha}
\addcontentsline{toc}{chapter}{Bibliography}
\bibliography{KESM}

\vfill

\bigskip\bigskip\noindent
{\small
Nils Berglund \\
Institut Denis Poisson (IDP) \\ 
Universite d'Orleans, Universite de Tours, CNRS -- UMR 7013 \\
B\^atiment de Mathematiques, B.P. 6759\\
45067~Orleans Cedex 2, France \\
{\it E-mail address: }
{\tt nils.berglund@univ-orleans.fr} \\
{\tt https://www.idpoisson.fr/berglund} 

\end{document}